\documentclass{scrartcl}
\usepackage[T1]{fontenc}
\usepackage[utf8]{inputenc}
\usepackage{lmodern}
\usepackage{amsmath}
\usepackage{amsfonts}
\usepackage{amssymb}
\usepackage{mathrsfs}
\usepackage{dsfont}
\usepackage{semantic}
\usepackage{extarrows}
\usepackage{stmaryrd}
\usepackage{url}
\usepackage{xspace}
\usepackage{microtype}
\usepackage{graphicx}
\usepackage{cite}
\usepackage{scalerel}
\usepackage{xcolor}
\usepackage{amsthm}
\usepackage{thm-restate}
\usepackage{rotating}
\usepackage{multirow}
\usepackage{ifthen}
\usepackage{bigdelim}
\usepackage{tikz}
\usepackage{marginnote} 

\pgfdeclarelayer{background}
\pgfdeclarelayer{foreground}
\pgfsetlayers{background,main,foreground}

\pgfkeys{%
	/tikz/on layer/.code={
		\pgfonlayer{#1}\begingroup
		\aftergroup\endpgfonlayer
		\aftergroup\endgroup
	},
	/tikz/node on layer/.code={
		\pgfonlayer{#1}\begingroup
		\expandafter\def\expandafter\tikz@node@finish\expandafter{\expandafter\endgroup\expandafter\endpgfonlayer\tikz@node@finish}%
	}
}

\usepackage{mathtools}
\usepackage{enumitem}
\usetikzlibrary{shapes.geometric,hobby,calc} 
\usetikzlibrary{decorations.pathmorphing,decorations.pathreplacing}
\usetikzlibrary{decorations.markings,intersections}
\usetikzlibrary{decorations.text}
\usetikzlibrary{shapes.misc}
\usetikzlibrary{decorations,shapes}
\usetikzlibrary{patterns,intersections,through,backgrounds}
\RequirePackage{etoolbox}
\usepackage{tikzmacros}
\usepackage{bbding}
\usepackage{xfrac}
\usepackage{caption}
\usepackage[subrefformat=parens]{subcaption} 
\usepackage[makeroom]{cancel} 
\usepackage{nicefrac}
\usepackage{algorithm}
\usepackage[noend]{algpseudocode}

\usepackage{hyperref}
\usepackage{cleveref}

\usepackage{macros}

\newif\ifcomment
\commentfalse
\commenttrue

\title{A directed flat wall theorem excluding a crossrow grid}
\DeclareRobustCommand{\authorthing}{Meike Hatzel\thanks{TU Darmstadt, Germany. E-mail: \href{research@meikehatzel.com}{research@meikehatzel.com}. Meike Hatzel's research was supported by the Institute for Basic Science (IBS-R029-C1).}
\and
Ken-ichi Kawarabayashi\thanks{National Institute of Informatics \& The University of Tokyo, Tokyo, Japan, \texttt{k\_keniti@nii.ac.jp}.  Supported by JSPS Kakenhi 26K21777 and JP25K24465, and by JST ASPIRE JPMJAP2302.}
\and
Stephan Kreutzer\thanks{TU Berlin, Germany}
\and
Sebastian Wiederrecht\thanks{School of Computing, KAIST, Daejeon, South Korea. E-mail: \href{wiederrecht@kaist.ac.kr}{wiederrecht@kaist.ac.kr}}}
\author{\authorthing}
\date{}

\begin{document}

\maketitle

\begin{abstract}
    The graph minor project contains the most influential results in recent undirected graph theory research.
    There has been progress in recent years in generalising some of their results to directed graphs, with the directed grid theorem of Kawarabayashi and Kreutzer [STOC~'15] and the directed flat wall theorem of Giannopoulou, Kawarabayashi, Kreutzer, and Kwon [SODA~'22].

    We discuss the two different versions of the existing directed flat wall theorem and their drawbacks.
    Then, we present an alternative directed flat wall theorem that excludes a different digraph as a butterfly minor.
    This new theorem lies \say{in between} the two existing ones and as such does not have either of these drawbacks.
    
    The proof of our flat wall theorem is based on the one by Giannopoulou, Kawarabayashi, Kreutzer, and Kwon~[SODA~'22], which has been adapted by Giannopoulou and Wiederrecht [STOC~'24].
    Here we make further adjustments to match our setting.
\end{abstract}

\section{Introduction}

The graph minor structure theorem is a cornerstone in the graph minor series by Robertson and Seymour.
It is concerned with the question of what the structure of graphs excluding a fixed minor looks like.
If one excludes a planar graph, one obtains a class of graphs of bounded \treewidth~\cite{graphminorsV}.
But the class of grids shows that the \treewidth is not bounded when excluding a non-planar graph.
So, what is the structure of graphs excluding non-planar graphs?
For $K_5,$ this is answered by Wagner's theorem~\cite{wagner1937}, which states that every four-connected graph is planar if and only if it excludes $K_5$ as a minor.
This can equivalently be stated as: every graph excluding $K_5$ can be obtained from the moebius ladder on eight vertices, also called the Wagner graph, and planar graphs by an operation called \emph{small clique sums}.


The graph minor structure theorem~\cite{graphminorsXVII} states that every graph that excludes a non-planar graph $H$ as a minor can be decomposed by a tree decomposition with additional requirements.
The bags of this tree decomposition do not have to have bounded width, but the adhesion, that is, the intersection between two bags, has bounded size.
Additionally, the bags induce subgraphs that can be further decomposed into several 4-connected graphs.
These 4-connected graphs, after deleting a small set of apex vertices, admit a drawing into a surface in which $H$ does not have a drawing without crossings, such that there are only few areas, the \emph{vortices}, that contain crossings.


The structure described by the graph minor structure theorem is parametrically necessary and sufficient.
Not only does every graph that excludes a non-planar graph as a minor have the described structure; additionally, graphs of the described structure do not contain a large clique as a minor.


In directed structure theory, there is an analogue to the grid theorem establishing that the cylindrical grid is an obstruction to directed \treewidth~\cite{kawarabayashi2015directedgridtheorem}.
But the property that a digraph is planar if and only if it is a minor of the cylindrical grid does not hold.
The cylindrical grid does play a role similar to the grid in undirected graphs and thus, the cylindrical wall a similar one to the wall.
However, every planar undirected graph is contained as a minor in an undirected wall.
This is not true for digraphs.
The class of cylindrical walls does not contain all planar digraphs, not even all strongly planar digraphs, as butterfly minors.
In general, highly non-planar digraphs can be of directed \treewidth one, as is witnessed by the class of DAGs.
So, neither does excluding a planar digraph imply small directed \treewidth, nor does even small directed \treewidth imply planarity.

Directed structure theory evolved further in the past few years, introducing directed tangles~\cite{directedtangles2020} and also a directed flat wall theorem~\cite{giannopoulou2020directed}.
However, building a directed structure theorem based on this flat wall theorem leads to difficulties as its properties differ from the undirected flat wall theorem.
The main obstacle to transferring the results leading to the undirected structure theorem into the directed setting is the fact that there is no two-path theorem.
A digraph can be highly connected and still not allow for two disjoint paths between two given pairs of vertices~\cite{Thomassen1991non-2-linked}.
In the undirected case, the two-path theorem yields that certain parts of a graph allow for a drawing without crosses into a closed disk, which makes up a crucial step in the undirected structure theorem.

We present an alternative directed flat wall theorem and provide some intuition why we think it yields a better base for a structure theorem.

\section{Preliminaries}

Additionally, we define \MarginOnly{$[k]$}$[k] \coloneqq \Set{1, \dots, k}$ for all $k \in \N.$
For two sets $X$ and $Y$ and a function $\alpha: X \to Y,$ we often use the shortcut $\Fkt{\alpha}{X'} \coloneqq \Set{\Fkt{\alpha}{x} \mid x \in X'}$ for a subset~$X' \subseteq X.$

\paragraph{Digraphs} A \define{digraph} $D$ has a vertex set \Margin{$\V{D}$} and an edge set \Margin{$\E{D}$}.
Let $e = \Brace{u,v}$ be an edge of a digraph $D,$ then we call $u$ the \define{tail} of $e$ and $v$ the \define{head} of $e.$
Let $v$ be a vertex in $D,$ by \Margin{$\Out{D}{v}$} we denote the set of \define{out-edges} of $v,$ that is edges with $v$ as tail.
The \define{out-degree} of $v$ is defined by \MarginOnly{\resizebox{\marginparwidth}{!}{$\Outdeg{D}{v}$}}$\Outdeg{D}{v} \coloneqq \Abs{\Out{D}{v}}.$
Similarly, by \Margin{$\In{D}{v}$} we denote the set of \define{in-edges} of $v,$ that is edges with $v$ as head, and define the \define{in-degree} of $v$ by \MarginOnly{$\Indeg{D}{v}$}$\Indeg{D}{v} \coloneqq \Abs{\In{D}{v}}.$
We omit the index, if the context clearly provides the digraph $D.$

\paragraph{Paths and linkages} A \defineSorted{(directed) path}{directed path} $P$ of \defineSubindex{length}{directed path@(directed) path} $k$ in a directed graph~$D$ is a sequence of distinct vertices $v_1,\dots,v_{k+1}$ such that $\Brace{v_{i},v_{i+1}} \in \E{D}$ for all $1 \leq i \leq k.$
The vertex $v_1$ is called the \defineSubindex{start-vertex}{directed path@(directed) path} of $P,$ denoted \MarginOnly{$\Start{P}$}$\Start{P},$ while the vertex $v_{k+1}$ is the \defineSubindex{end-vertex}{directed path@(directed) path} of $P,$ denoted \MarginOnly{$\End{P}$}$\End{P}.$
We call $P$ a $\Start{P}$-$\End{P}$-path.
We often identify $P$ with the subgraph $\Brace{\Set{v_1,\dots,v_{k+1}},\Set{\Brace{v_i,v_{i+1}}\mid 1 \leq i \leq k}}.$
Let $X\subseteq\V{D}$ be a set of vertices in $D.$
A directed path $P$ is \defineSubindex{disjoint}{directed path@(directed) path} from $X$ if $\V{P}\cap X = \emptyset.$
It is \defineSubindex{internally disjoint}{directed path@(directed) path} from $X$ if $\V{P}\cap X \subseteq \Set{\Start{P},\End{P}}.$
An \defineSorted{$X$-path}{Xpath} is a directed path $P$ of length at least one that is internally disjoint from $X$ and $\Start{P},\End{P} \in X.$
For a subgraph $D' \subseteq D$ we also write $D'$-path instead of $\V{D'}$-path.
Two paths $P$ and $P'$ are \define{disjoint} if $\V{P} \cap \V{P'} = \emptyset$ and they are \define{internally disjoint} if their only intersections are start- and end-vertices, that is, $\V{P} \cap \V{P'} \subseteq \Set{\Start{P},\Start{P'},\End{P},\End{P'}}.$
A \defineSorted{(directed) cycle}{directed cycle} of length $k$ is a directed path of length $k-1$ such that $\Brace{\End{P},\Start{P}}\in\E{D}.$
A collection $\mathcal{L}$ of pairwise disjoint paths is called a \define{linkage}.
We say $\mathcal{L}$ is an $A$-$B$-linkage if every $L \in \mathcal{L}$ is an $a$-$b$-path for some $a \in A$ and $b \in B.$
We call a collection $\mathcal{L}$ of paths a \define{half-integral linkage} if every vertex of the graph occurs in at most two paths of $\mathcal{L}.$

A digraph~$D$ that is obtained from an undirected graph~$G$ by adding exactly one direction of every edge, that is, $\V{D}=\V{G}$ and for all $\Set{u,v} \in \E{G}$ we have either $\Brace{u,v} \in \E{D}$ or $\Brace{v,u} \in \E{D},$ is called an \define{orientation} of $G.$
A digraph~$D$ is a \defineSorted{(directed) tree}{directed tree} or \define{out-branching} rooted at a vertex $r$ if it is the orientation of an undirected tree $T$ rooted at $r$ such that $\Indeg{D}{v}= 1$ for every $v\in\V{D}\setminus \Set{r}.$

A digraph $D$ is \define{strongly connected} if for every two vertices $x,y \in \V{D}$ there is an $x$-$y$-path and a $y$-$x$-path in $D.$
The digraph $D$ is \define{weakly connected} if $\undirected{D}$ is connected.
A \define{strongly connected component} or \define{strong component} of $D$ is a maximal strongly connected subgraph of $D.$

\paragraph{Butterfly minors and models} In directed graphs, there are several ways to define containment relations that generalise minors in undirected graphs.
The most prominent one in structure theory is the one of \emph{butterfly minors}.
The idea behind the definition is that the contraction of an edge should not create new paths, that is, if there is no path between two vertices before the contraction, then there is no such path after the contraction either.
To this end, we call an edge $e$ \define{butterfly contractible} if $e$ is the only out-edge of its tail $u$ or the only in-edge of its head $v.$
A digraph is a \define{butterfly minor} of $D$ if it can be obtained from a subgraph of $D$ by contracting butterfly contractable edges.

An alternative way to describe this concept is via minor models.
An \define{in-branching} is the \define{reverse graph} (obtained by reversing the direction of all edges) of an out-branching, and an \define{in-out-branching} is obtained by identifying the root of an in-branching and the root of an out-branching.
A \define{butterfly (minor) model} of a digraph $D'$ in a digraph $D$ is a function $\mu$ mapping every vertex of $D'$ to a subgraph of $D$ and every edge of $D'$ to an edge of $D$ such that
\begin{enumerate}
	\item $\Fkt{\mu}{v}$ is an in-out-branching for every $v \in \V{D'},$
	\item $\Fkt{\mu}{v}$ is disjoint from $\Fkt{\mu}{u}$ for distinct $u,v \in \V{D'},$
	\item $\Fkt{\mu}{e} \in \E{D}$ for all $e \in \E{D'},$
	\item $\Fkt{\mu}{e} \neq \Fkt{\mu}{e'}$ for distinct $e,e' \in \E{D'},$ and
	\item if $e=\Brace{u,v} \in \E{D'},$ then $\Start{\Fkt{\mu}{e}}$ lies in the out-branching of $\Fkt{\mu}{u}$ and $\End{\Fkt{\mu}{e}}$ lies in the in-branching of $\Fkt{\mu}{v}.$
\end{enumerate}

\paragraph{Directed separations}
For two vertex subsets $A$ and $B$ in a digraph $D$ the tuple $\Separation{A}{B}{}{}$ is a \define{directed separation} if $A\cup B = \Fkt{V}{D}$ and there are no edges with tail in $A\setminus B$ and head in $B\setminus A$ or no edges with tail in $B \setminus A$ and head in $A\setminus B.$
The set $S\coloneqq A\cap B$ is called the \defineSubindex{separator}{directed separation} and $\Abs{S}$ is the \define{order} of the separation.
In case there are no edges with tail in $A\setminus B$ and head in $B\setminus A,$ we write $\Separation{A}{B}{}{l}$ indicating that edges are allowed to go from $B\setminus A$ to $A\setminus B.$
Similarly, we write \Margin{$\Separation{A}{B}{}{r}$} if no edge in $D$ has its tail in $B\setminus A$ and its head in $B\setminus A.$

\paragraph{Directed \treewidth}
\label{sec:dtw}
As a directed analogue to \treewidth Reed~\cite{reed1999dtw} and Johnson, Robertson, Seymour and Thomas~\cite{johnson2001directedTreeWidth} introduced the concept of directed \treewidth and conjectured that a directed version of the grid theorem holds for this directed width measure and a directed version of a grid.

For an arborescence, or directed tree, $T,$\MarginOnly{$\remEdgeLower{T}{e}$} the removal of an edge $e = \Brace{t_1,t_2}$ splits $T$ into two sub-arborescence: $\remEdgeLower{T}{e}$ containing $t_2$ and \Margin{$\remEdgeUpper{T}{e}$} containing $t_1$ as well as the root of $T.$
We also write \Margin{$T_d$} for the sub-arborescence rooted at the vertex $d.$

\begin{restatable}[Directed \treewidth]{definitionx}{DirectedTreewidthDefinition}
	\label{def:dtw}
	A \define{directed tree decomposition} of a digraph~$D$ is a triple
	$\Brace{T,\beta,\gamma}$ where $T$ is a directed tree, $\beta: \V{T} \to 2^{\V{D}}$ maps every vertex $t$ of $T$ to a set
	$\Fkt{\beta}{t} \subseteq \V{D}$ called the \defineSubindex{bag}{directed tree decomposition} \emph{at $t$} and $\gamma: \E{T} \to 2^{\V{D}}$ maps every edge $e$ of $T$ to a set $\Fkt{\gamma}{e} \subseteq \V{D}$ called the \defineSubindex{guard}{directed tree decomposition} \emph{at $e$} such that the following hold
	\begin{enumerate}
		\item $\Set{\Fkt{\beta}{t} \mid t \in \V{T}}$ is a partition of $\V{D}$ (with possibly empty classes), and
		\item for all $e \in \E{T}$ there is no closed walk in $D-\gamma(e)$ containing a vertex of $\Fkt{\beta}{\remEdgeLower{T}{e}}$ and a vertex of $V(D)- \Fkt{\beta}{\remEdgeLower{T}{e}}$.
	\end{enumerate}
	For every vertex~$t \in \V{T}$ we define $\Fkt{\Gamma}{t} \coloneqq \Fkt{\beta}{t} \cup \bigcup_{\incident{e}{t}} \Fkt{\gamma}{e}.$
	The \defineSubindex{width}{directed tree decomposition} of a directed tree decomposition is defined by $\max \Set{\Abs{\Fkt{\Gamma}{t}} \mid t \in  \V{T}}.$
	The directed \treewidth of $D,$ denoted $\dtw{D},$ is defined as the smallest $k \in \N$ such that $D$ has a directed tree decomposition of width	$k.$ 
\end{restatable}

A \define{cylindrical grid} $\CylGrid{k}$ of order $k$ consists of $k$ concentric directed cycles and $2k$ paths connecting the cycles in alternating directions, see \cref{fig:cyl_grid} for an example, as follows.
The $k$ directed disjoint cycles $C_1,\dots,C_k$ have length $2k$ each.
The cycle $C_i$ has the vertex set $\Set{v^i_1,\dots,v^i_{2k}}$ with the natural cyclic ordering.
The paths are of two different kinds, we have the \emph{in-paths} $\inPath_1,\dots,\inPath_k$ and the \emph{out-paths} $\outPath_1,\dots,\outPath_k$ as follows.

\begin{align*}
	\inPath_j &= v^k_{2j},v^{k-1}_{2j},\dots,v^2_{2j},v^1_{2j} \text{ for all } j \in \Set{1,\dots,k}\\
	\outPath_j &= v^1_{j},v^{2}_{j},\dots,v^{k-1}_{j},v^k_{j} \text{ for all } j \in \Set{1,\dots,k}.
\end{align*}
The paths $\inPath_i$ and $\outPath_i$ together build the $i$-th \emph{row} of $W,$ which we also denote $\row_i.$

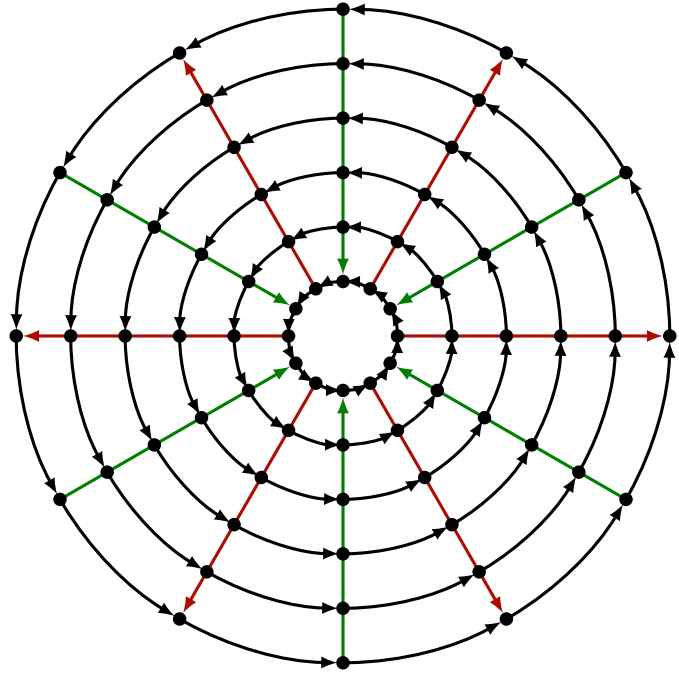
\begin{figure}[ht]
	\centering
	\resizebox{0.6\textwidth}{!}{%
	\begin{tikzpicture}[scale=0.8]
		\CylindricalGrid{6}{1}{myGreen}{myRed}
	\end{tikzpicture}}
	\caption{A cylindrical grid of order 6 with \textcolor{myGreen}{in-paths} and \textcolor{myRed}{out-paths}.}
	\label{fig:cyl_grid}
\end{figure}

After being open for over 15 years this conjecture was finally confirmed to be true by Kawarabayashi and Kreutzer~\cite{kawarabayashi2015directedgridtheorem}.

\begin{theorem}[Kawarabayashi and Kreutzer~\cite{kawarabayashi2015directedgridtheorem}]
	\label{thm:directed_grid_theorem}
	There is a function $\DirectedGridThmFunction{} : \N \to \N$ such that every digraph $D$ either satisfies $\dtw{D}\leq\ k,$ or contains the cylindrical grid of order $\DirectedGridThmFunction{k}$ as a butterfly minor.
\end{theorem}

While the function they provide is exponential, there is a polynomial bound for planar directed graphs~\cite{2019planardirectedgridtheorem}.
Additionally, Campos, Lopes, Maia, and Sau~\cite{2022FPTdirectedgridtheorem} adapted the result into an FPT algorithm.

\section{The directed wall}
\label{sec:old_flat_wall}

Let us take a closer look at the existing directed flat wall theorem by Giannopoulou, Kawarabayashi, Kreutzer, and Kwon~\cite{giannopoulou2020directed}.
In order to do so, we have to introduce how a wall in the directed setting is defined alongside some additional related definitions.

An \define{elementary cylindrical wall} of order $k$ is a graph, which is the union of $k$ directed disjoint cycles $C_1,\dots, C_k$ of length $4k$ each and $2k$ disjoint paths.
The cycle $C_i$ has the vertex set $\Set{v^i_1,\dots,v^i_{4k}}$ with the natural cyclic ordering.
The paths are of two different kinds, we have the \define{in-paths} $\inPath_1,\dots,\inPath_k$ and the \define{out-paths} $\outPath_1,\dots,\outPath_k$ as follows, also see~\cref{fig:wall} for an illustration.
\begin{align*}
	\inPath_j &= v^k_{4j-1},v^k_{4j},v^{k-1}_{4j-1},v^{k-1}_{4j},\dots,v^1_{4j-1},v^1_{4j} \text{ for all } j \in [k]\\
	\outPath_j &= v^1_{4(j-1)+1},v^1_{4(j-1)+2},v^{2}_{4(j-1)+1},v^{2}_{4(j-1)+2},\dots,v^k_{4(j-1)+1},v^k_{4(j-1)+2}\\
	& \text{ for all } j \in [k].
\end{align*}

\begin{figure}[!ht]
	\centering
	\begin{tikzpicture}[scale=0.5, decoration={
			markings,
			mark=at position 0.75 with {\arrow{latex}}}]
		\tikzstyle{w}=[circle,draw,fill=black!50,inner sep=0pt,minimum width=3pt]

		\foreach \y in {1, 2, 3, 4}{
			\node at (-1, 8-2*\y+1) {$\outPath_{\y}$};
			\node at (-1, 8-2*\y) {$\inPath_{\y}$};
		}
		
		\foreach \x in {0, 2, 4}{
			\foreach \y in {1,3,5,7}{
				\draw[postaction={decorate}] (\x, \y+1)-- (\x, \y);	
				\draw[postaction={decorate}] (\x+1, \y)-- (\x+1, \y-1);	
				
				\draw[postaction={decorate}] (\x, \y)-- (\x+1, \y);	
				\draw[postaction={decorate}] (\x+1, \y)-- (\x+2, \y);	
				
				\draw[postaction={decorate}] (\x+1, \y-1)-- (\x, \y-1);	
				\draw[postaction={decorate}] (\x+2, \y-1)-- (\x+1, \y-1);	
				
			}
		}
		\foreach \x in {0, 6}{
			\foreach \y in {1,3,5,7}{
				
				\draw[very thick, postaction={decorate}] (\x, \y+1)-- (\x, \y);	
				\draw[very thick, postaction={decorate}] (\x+1, \y)-- (\x+1, \y-1);	
				
				\draw[very thick, postaction={decorate}] (\x, \y)-- (\x+1, \y);	
				
				\draw[very thick, postaction={decorate}] (\x+1, \y-1)-- (\x, \y-1);	
			}
		}
	
		\foreach \i in {0,...,7}
		{
			\foreach \j in {0,...,7}
			{
				\node[vertex,scale=0.5] at (\i,\j) {};
			}
		}

		\foreach \x/\d in {1/thick,2/thin,3/thin,4/thick}{
			\draw[\d,-] (0+2*\x-2, 0) .. controls (0+2*\x-2, -2.5+0.5*\x) ..  (7, -2.5+0.5*\x);
			\draw[\d,-latex] (7, -2.5+0.5*\x) .. controls (12-\x, -2.5+0.5*\x) and (12-\x, 0)  .. (12-\x, 4) ;
			\draw[\d,-] (12-\x, 4)  .. controls (12-\x, 8) and (8+2.5-0.5*\x, 11-0.5*\x)  ..  (7, 11-0.5*\x);
			\draw[\d,-] (7, 11-0.5*\x) .. controls (0+2*\x-2, 11-0.5*\x) ..  (0+2*\x-2, 8) ;
			
			\pgfmathtruncatemacro\y{5-\x}
			\node[anchor=south east] at (8-2*\x, 8) {$C_{\y}$};
			
		}

	\end{tikzpicture}
	\caption{The cylindrical wall of order four.
		The perimeters are depicted using thick edges.}
	\label{fig:wall}
\end{figure}
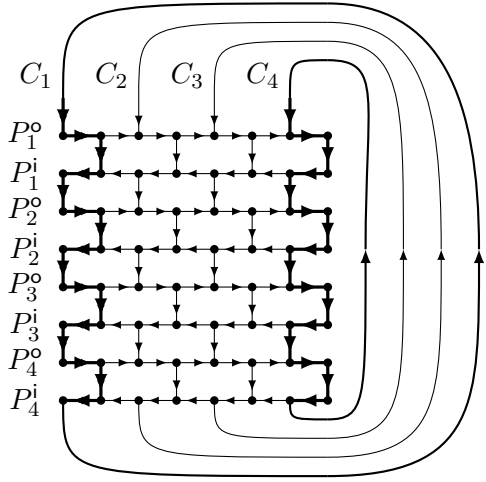

The paths $\inPath_i$ and $\outPath_i$ together make up the $i$-th \define{row} of $W,$ which we also denote $\row_i.$
We consider rows modulo the wall size, that is, for $\ell > k$ we consider $\inPath_{\ell}\coloneqq \inPath_{\Brace{\ell -1}\bmod k +1},$ $\outPath_{\ell}\coloneqq \outPath_{\Brace{\ell -1}\bmod k +1}$ as well as $\row_{\ell}\coloneqq \row_{\Brace{\ell -1}\bmod k +1}.$
The \define{perimeter} of $W,$ denoted $\per{W},$ contains all vertices of $C_1$ and $C_k.$
The vertices of $C_1$ we also refer to as the \define{inner perimeter}, $\perIn{W},$ and we call the vertices of $C_k$ the \define{outer perimeter}, $\perOut{W}.$
The graph that is obtained from $W$ by removing the perimeter (and the edges incident to its vertices) is called the \defineSubindex{interior}{elementary cylindrical wall} of the wall, we write $\interior{W}.$
A \define{subwall} $W[i,j]$ of $W$ for $1 \leq i < j \leq k$ is the graph obtained by the union of the cycles $C_i,\dots,C_j$ and the subpaths $\Subpath{\inPath_x}{C_i}{C_j}$ and $\Subpath{\outPath_x}{C_i}{C_j}$ for all $1 \leq x \leq k.$
We sometimes use $Q_1, \dots,Q_k$ to refer to the subpaths of $C_1,\dots,C_k$ starting on $\outPath_1$ and ending on $\inPath_k.$

A \defineSorted{(cylindrical) wall}{cylindrical wall}\index{wall} is a subdivision of an elementary cylindrical wall.
The \define{branch vertices} of a wall are the vertices that have out- or in-degree two.
Let $W$ be a cylindrical wall of order $k.$
We say that $W$ \define{grasps} a butterfly minor model $\mu$ of $\K{t}$ if for every $v\in\V{\K{t}}$ there exists a pair $i_v,j_v \in [k]$ such that $\V{Q_{i_v}}\cap\V{\inPath_{j_v}}\subseteq \V{\Fkt{\mu}{v}}$ or $\V{Q_{i_v}}\cap\V{\outPath_{j_v}}\subseteq \V{\Fkt{\mu}{v}},$ that is, the model of every vertex contains a branch vertex of $W.$

\begin{restatable}[Strip]{definition}{StripDefinition}
	\label{def:strip}
	A \define{strip} between column $i$ and $j$ of a wall $W,$ for $1 \leq i < j \leq k,$ is the subgraph of $W$ containing all rows $\row_i,\dots,\row_j$ and the subpaths $\Subpath{Q_1}{\row_i}{\row_j},\dots,\Subpath{Q_k}{\row_i}{\row_j}.$
	The \defineSubindex{height}{strip} of such a strip is $j-i+1.$
\end{restatable}

\begin{figure}[!ht]
	\centering
	\resizebox{0.8\textwidth}{!}{%
    \includegraphics{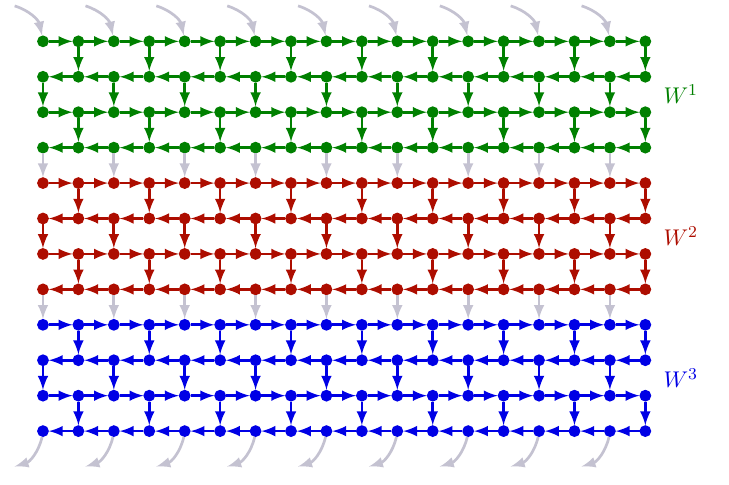}
    }
	\caption{The wall divided into three strips $W^1,$ $W^2$ and $W^3.$}
	\label{fig:strips}
\end{figure}

The boundaries of the faces of a wall, except for the two faces containing $C_1$ and $C_k,$ are called \define{bricks}.

\begin{restatable}[Slice]{definition}{SliceDefinition}
	\label{def:slice}
	Let $k \in \N$ be a positive integer and $W$ be a cylindrical wall of order $k.$
	A \define{slice} $W'$ of $W$ is a cylindrical wall containing the vertical paths $Q_i, \dots,$ $Q_{i+\ell}$ for all $i \in [k]$ and some $\ell \in [k-i],$ and the horizontal paths $\InducedSubgraph{P_1^1}{Q_i, \dots, Q_{i+\ell}},$ $\dots,$ $\InducedSubgraph{P_k^2}{Q_i, \dots, Q_{i+\ell}}.$
	We say that $W'$ is the \emph{slice of $W$ between $Q_i$ and $Q_{i+\ell}$} and that $W'$ is of \defineSubindex{width}{slice} $\ell+1.$
\end{restatable}

\begin{figure}[!ht]
	\centering
	\resizebox{0.8\textwidth}{!}{%
    \includegraphics{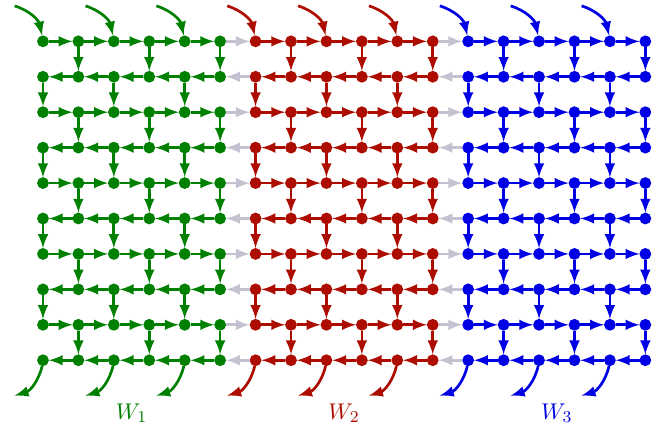}
    }
	\caption{The wall divided into three slices $W_1,$ $W_2$ and $W_3.$}
	\label{fig:slices}
\end{figure}

We define a distance measure with respect to a cylindrical wall based on how many paths lie between two given vertices.

\begin{definition}[$W$-Distance]
	\label{def:Wdistance}
	Let $k \in \N$ be a positive integer and $W$ be a cylindrical wall of order $k.$
	Given two vertices $u,v\in\V{W},$ we say that they have \defineSorted{$W$-distance}{Wdistance} at least $i$ if there exist $i$ distinct vertical or $i$ distinct horizontal paths whose removal separates $u$ and $v$ in $W.$
\end{definition}

Most of these definitions can naturally be used for cylindrical grids as well and we do so at times.
Cylindrical grids and walls are closely related.
From a wall one can obtain a cylindrical grid of the same order by butterfly contraction.
From a cylindrical grid one can obtain a wall by deleting subpaths of the in- and out-paths alternatingly, thereby loosing a factor two in the order.
Therefore, we can obtain the following theorem using the function \Margin{$\DirectedWallThmFunction{}$}~$:~\N \to \N$ obtained from the function $\DirectedGridThmFunction{}$ provided by \cref{thm:directed_grid_theorem}:
\begin{equation*}
	\DirectedWallThmFunction{k} \coloneqq 2\DirectedGridThmFunction{k}.
\end{equation*}

\begin{theorem}[Kawarabayashi and Kreutzer, 2015 \cite{kawarabayashi2015directedgridtheorem}]
	\label{thm:directed_wall_theorem}
	There is a function $\DirectedWallThmFunction{} : \N \to \N$ such that every digraph $D$ either satisfies $\dtw{D}\leq\ k,$ or contains the cylindrical wall of order $\DirectedWallThmFunction{k}$ as a butterfly minor.
\end{theorem}

The flat wall theorem for directed graphs by Giannopoulou et al.~\cite{giannopoulou2020directed} ensures the existence of a wall with a bounded number of cross-rows when excluding a large clique as a minor.
We prove a statement for digraphs excluding a cross-row grid instead, which allows us to get rid of all forward crosses in the wall by deleting a small apex set.
We use the following definitions that were introduced by Giannopoulou et al.

\newcommand{\Triadic}{\mathcal{W}}
\begin{definition}[Triadic Partitions]
	\label{def:triadic_partition}
	Let $k \in \N$ be a positive integer and $W$ be a cylindrical wall of order $3k.$
	The \define{triadic partition} of $W$ is the tuple
	\begin{equation*}
		\Triadic = \Brace{W,k,W_1,W_2,W_3,W^1,W^2,W^3}
	\end{equation*}
	such that for each $i \in [3],$ $W_i$ denotes the \hyperref[def:slice]{slice} of $W$ between $Q_{k\Brace{i-1}+1},$ see~\cref{fig:slices} and $Q_{ik},$ and $W^i$ denotes the strip of $W$ between the rows $k\Brace{i-1}+1$ and $ik,$ see~\cref{fig:strips}.
\end{definition}

We often consider paths starting and ending in a brick.
In order to be able to construct paths to and from their start- and end-vertices, we need parts of the wall around the brick, thus there is the concept of tiles centred at a brick.

\newcommand{\centre}[1]{\variablestyle{C}_{#1}}

\begin{restatable}[Tiles]{definition}{TileDefinition}
	\label{def:tile}
	For $i,j \in [k]$ and $d \geq 1,$ the \define{tile} $\tile_{i,j,d}$ of $W$ is defined as the subgraph of $W$ obtained by the union
	\begin{align*}
		\bigcup_{i \leq \ell \leq i+2d+1} \Subpath{Q_{\ell}}{\row_j}{\row_{j+2d+1}} \cup \bigcup_{j \leq \ell \leq j+2d+1} \Subpath{\row_{\ell}}{Q_i}{Q_{i+2d+1}}.
	\end{align*}
	We say $i$ is the \defineSubindex{column index}{tile} of $\tile_{i,j,d},$ and $j$ is the \defineSubindex{row index}{tile} of $\tile_{i,j,d}.$
	Also, $d$ is the \defineSubindex{width}{tile} of $\tile_{i,j,d}.$
	See \cref{fig:tile} for an illustration.
	
	The \defineSubindex{perimeter}{tile} of $\tile_{i,j,d}$ is given by 
	\begin{align*}
		\tile_{i,j,d} \cap \Brace{Q_i \cup Q_{i+2d+1} \cup \outPath_j \cup \inPath_{j+2d+1}}.
	\end{align*}
	We call $Q_i$ the \defineSubindex{left path}{tile} of the perimeter, $Q_{i+2d+1}$ its \defineSubindex{right path}{tile}, $\outPath_j$ the \defineSubindex{upper path}{tile} of the perimeter, and finally $\inPath_{j+2d+1}$ its \defineSubindex{lower path}{tile}.
	
	The \defineSubindex{corners}{tile} of a tile are the vertices $a,b,c,d \in \V{\tile_{i,j,d}}$ where
	\begin{itemize}
		\item $a,$ the \emph{upper left corner}, is the common starting point of $\tile_{i,j,d} \cap Q_i$ and $\tile_{i,j,d} \cap \outPath_j,$
		\item $b,$ the \emph{upper right corner}, is the end of $\tile_{i,j,d} \cap \outPath_j$ and the starting point of $\tile_{i,j,d} \cap Q_{i+2d+1},$
		\item $c,$ the \emph{lower left corner}, is the common end of $\tile_{i,j,d} \cap Q_i$ and $\tile_{i,j,d} \cap \inPath_{j+2d+1},$ and
		\item $d,$ the \emph{lower right corner}, is the end of $\tile_{i,j,d} \cap Q_{i+2d+1}$ and the starting point of $\tile_{i,j,d} \cap \inPath_{j+2d+1}.$
	\end{itemize} 
	
	The \defineSubindex{centre}{tile} of $\tile_{i,j,d}$ is the boundary of the unique brick $\centre{\tile_{i,j,d}}$ of $W$ whose boundary consists of vertices from $Q_{i+d+1},$ $Q_{i+d+2},$ $\inPath_{j+d+1},$ and $\outPath_{j+d+2}.$
	All vertices of $\tile_{i,j,d}$ which are not in the centre and not on the perimeter of $\tile_{i,j,d}$ are called \emph{internal}.
\end{restatable}

\begin{figure}[thb]
	\centering
	\resizebox{\textwidth}{!}{%
    \includegraphics{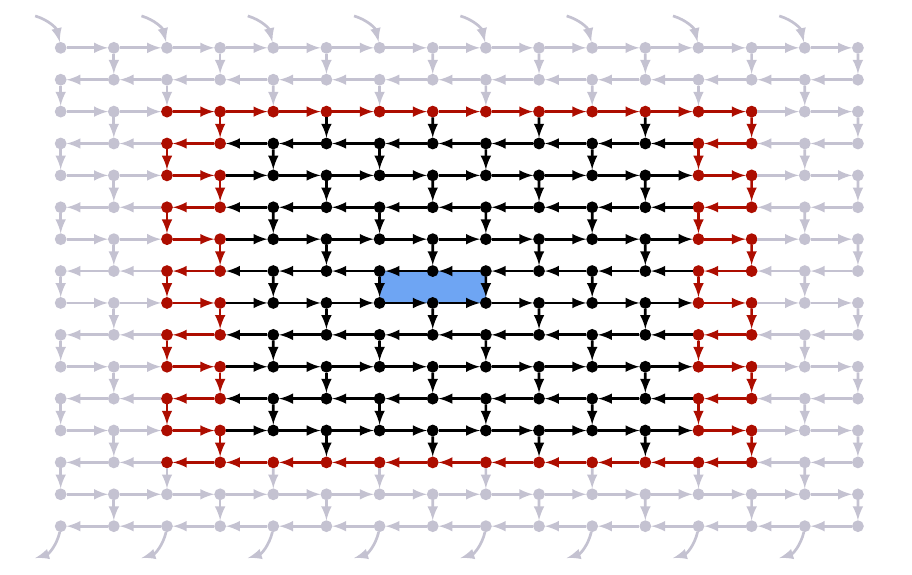}
    }
	\caption{A tile of of width two in a wall of order eight.
		The centre brick is filled in \textcolor{myLightblue!80!blue}{blue} and the perimeter is drawn in \textcolor{myRed}{red}.}\label{fig:tile}
\end{figure}

Please note that by this definition, only bricks lying between $\inPath_i$ and $\outPath_{i+1}$ for some $i\in [k]$ can be the centre of a tile.
However, if we take the mirror image of the unique embedding of the wall along $Q_1,$ we obtain a new embedding, switching in- and out-paths.
We call these two possible embeddings the two \define{parametrisations} of a wall, we write \MarginOnly{$\Fkt{\pi}{W}$}$\Fkt{\pi}{W} = \variablestyle{o},$ if the cycles encounter the out-path of a row before the in-path and $\Fkt{\pi}{W} = \variablestyle{i}$ otherwise.
This means that we can define for every brick $F$ of $W$ a tile $\tile_F$ such that $F$ is the centre of $\tile_F.$

Giannopoulou et al.~use the following concept of flatness.

\begin{definition}[weakly flat]
	\label{def:weakly_flat}
	Let $D$ be a digraph, $W$ a wall in $D,$ $a : \N\to\N$ a function and $t \in \N.$
	The wall $W$ is \defineSorted{$\Fkt{a}{t}$-weakly flat}{weakly flat} with directed \treewidth bounded by $d$ if there is a subset $A \subseteq \V{D}$ with $\Abs{A} \leq \Fkt{a}{t}$ such that $W$ is a wall in $D-A$ and the following~hold:
	\renewcommand{\labelenumi}{\textbf{\theenumi}}
	\renewcommand{\theenumi}{(\roman{enumi})}
	\begin{enumerate}[labelindent=0pt,labelwidth=\widthof{\ref{last-item-weakly-flat}},itemindent=1em]
		\item \label{weakly-flat:separation} There is an undirected separation $\Separation{X}{Y}{}{}$ in $D$ such that $A \cup \per{W} = X \cap Y$ and $\interior{W}$ is contained in $X\setminus Y,$ as well as every vertex in $X \setminus Y$ reaches a vertex in $\interior{W}$ or is reachable from it.
		\item \label{weakly-flat:jumps}\label{weakly-flat:extensions} For every path $Q$ in $D-A$ that has both endpoints in $\interior{W}$ but is internally disjoint from $\interior{W}$ there is a brick $B$ in $W$ such that the boundary of $B$ contains both endpoints of $Q.$
		Moreover, for every brick $B$ of $W$ let $V_B$ be the set of vertices of $D-A$ that appear as internal vertices of a path $Q$ with both endpoints on $B$ but internally disjoint from $B.$
		The strongly connected components of $\InducedSubgraph{D}{C},$ which are called the \define{extensions} of $B,$ have directed \treewidth at most $d.$
		\item \label{weakly-flat:crosses} If $\tile$ is a tile of width five of $W$ and $c_1$ is its upper left corner, $c_2$ its upper right corner, $d_1$ its lower left corner and $d_2$ its lower right corner, then there are no two paths $Q_1$ and $Q_2$ in $\Brace{D-A}-\Brace{W-\tile}$ such that $Q_1$ is a $c_1$-$d_2$-path and $Q_2$ is a $c_2$-$d_1$-path. \label{last-item-weakly-flat}
	\end{enumerate}
	For a function $G: \N \to \N,$ the wall $W$ is \defineSorted{$\Fkt{g}{t}$-nearly-$\Fkt{a}{t}$-weakly flat}{nearly weakly flat} with directed \treewidth bounded by $d$ if it satisfies \cref{weakly-flat:separation,weakly-flat:jumps} and there are at most $\Fkt{g}{t}$ rows whose tiles do not satisfy \cref{weakly-flat:crosses}, we also refer to such a row as \emph{cross-row}.
\end{definition}

Giannopoulou et al.~obtain two different flat wall theorems for directed graphs.
The~first excludes a transitive tournament in order to get rid of cross-rows entirely.
The second only excludes the clique, thus cross-rows may still exist.

\begin{theorem}[Directed weakly flat wall theorem~\cite{giannopoulou2020directed}]
	\label{thm:directed_weakly_flat_wall}
	There exist functions $d \colon \N \times \N \to \N$ and $a \colon \N \to \N$ such that for every directed graph $D$ and all $k,t \in \N$ one of the following is true:
	\renewcommand{\labelenumi}{\textbf{\theenumi}}
	\renewcommand{\theenumi}{(\roman{enumi})}
	\begin{enumerate}[labelindent=0pt,labelwidth=\widthof{\ref{last-item-weakly-flat-wall}},itemindent=1em]
		\item $\dtw{D} < \Fkt{d}{k,t},$
		\item $D$ contains a tournament of order $t$ as a butterfly minor
		\item there is a directed cylindrical wall $W$ of order $k$ in $D$ that is $\Fkt{a}{t}$-weakly flat with directed \treewidth bounded by $\Fkt{d}{k,t}.$ \label{last-item-weakly-flat-wall}
	\end{enumerate}
\end{theorem}

\begin{theorem}[Directed nearly-weakly flat wall theorem~\cite{giannopoulou2020directed}]
	\label{thm:directed_nearly_weakly_flat_wall}
	There exist functions $d \colon \N \times \N \to \N,$ $a \colon \N \to \N$ and $g \colon \N \to \N$ such that for every directed graph $D$ and all $k,t \in \N$ one of the following is true:
	\renewcommand{\labelenumi}{\textbf{\theenumi}}
	\renewcommand{\theenumi}{(\roman{enumi})}
	\begin{enumerate}[labelindent=0pt,labelwidth=\widthof{\ref{last-item-nearly-weakly-flat-wall}},itemindent=1em]
		\item $\dtw{D} < \Fkt{d}{k,t},$
		\item $D$ contains $\K{t}$ as a butterfly minor
		\item there is a directed cylindrical wall $W$ of order $k$ in $D$ that is $\Fkt{g}{t}$-nearly-$\Fkt{a}{t}$-weakly flat with directed \treewidth bounded by $\Fkt{d}{k,t}.$ \label{last-item-nearly-weakly-flat-wall}
	\end{enumerate}
\end{theorem}


Both of these theorems have their drawbacks.
\Cref{thm:directed_nearly_weakly_flat_wall} only finds a wall that still contains cross-rows.
So, it still contains highly non-planar behaviour.
\Cref{thm:directed_weakly_flat_wall} ensures a wall that does not contain any cross-rows, but the excluded structure of a transitive tournament is in itself not strongly connected.
So, the transitive tournament is of directed \treewidth one.
This means that any structure theorem using the transitive tournament as the excluded butterfly minor does not provide sufficiency, that is, every graph having the described structure also excludes the transitive tournament as a butterfly~minor.

\section{The cross-row grid}
\label{sec:cross-row_grid}

Here, we consider a structure lying in between the transitive tournament and the clique: the \emph{cross-row grid}.
It lies in between the two other structures in the sense that every transitive tournament is a butterfly minor of a cross-row grid and every cross-row grid is a butterfly minor of a clique.

Let $D$ be cylindrical grid of order $k$ with cycles $C_1,\dots,C_k$ and in-paths $\inPath_1,\dots,\inPath_k$ and out-paths $\outPath_1, \dots, \outPath_k.$
Let $a_1,\dots,a_k$ be the vertices of $\outPath_1$ and $b_1,\dots,b_k$ be the vertices of $\inPath_1$ in order of occurrence along the path.
A \define{cross-row grid} of order $k,$ written $\CrossRowGrid{k},$ is obtained from $D$ by adding the edges $\Set{\Brace{a_j,b_{k-j}} \mid 1 \leq j \leq k-1} \cup \Set{\Brace{a_j,b_{k-\Brace{j-2}}} \mid 2 \leq j \leq k}.$
See \cref{fig:cross-row_grid} for an example.

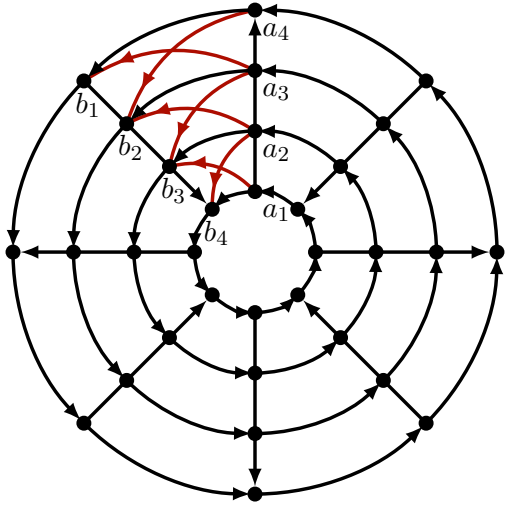
\begin{figure}[!ht]
	\centering
	\begin{tikzpicture}[decoration={
			markings,
			mark=at position 0.8 with {\arrow{latex}}},scale=0.8]
		\CylindricalGrid{4}{1}{}{}
		\begin{pgfonlayer}{background}
		\foreach[evaluate={\j=int(\i+1);}] \i in {1,2,3}
		{
			\draw[edge,myRed,bend right,postaction={decorate}] ({2*360/8}:\j) to ({3*360/8}:\i);
			\draw[edge,myRed,bend right,postaction={decorate}] ({2*360/8}:\i) to ({3*360/8}:\j);
		}
		\foreach[evaluate={\j=int(5-\i);}] \i in {1,2,3,4}
		{
			\node at ($({2*360/8}:\i)+(315:0.5)$) {$a_{\i}$};
			\node at ($({3*360/8}:\i)+(280:0.4)$) {$b_{\j}$};
		}
		\end{pgfonlayer}
	\end{tikzpicture}
	\caption{A cross-row grid of order four.}
	\label{fig:cross-row_grid}
\end{figure}

We start by proving that every transitive tournament is contained in a large enough cross-row grid as a butterfly minor.
In order to construct such a transitive tournament minor we introduce the following tool.
Let $D$ be a digraph with $2t$ designated vertices $\Brace{i_1,\dots,i_t, o_1,\dots,o_t}.$
If for every $s \in [t-1]$ there exist $t$ disjoint paths $P_1,\dots,P_t$ in $D$ such that $P_s$ is an $i_s$-$o_{s+1}$-path, $P_{s+1}$ is an $i_{s+1}$-$o_s$-path and for all $s' \notin\Set{s,s+1}$ the path $P_{s'}$ is an $i_{s'}$-$o_{s'}$-path, then $D$ is called a \defineSorted{$\Brace{i_1,\dots,i_t}$-$\Brace{o_1,\dots,o_t}$-switch}{switch}.

We define the function
\begin{equation*}
	\switchesfortournamentfunction{t} \coloneqq \frac{1}{2}\Brace{t^2-t}-2,
\end{equation*}
which gives for every number $t$ the number of switches necessary to a build transitive tournament minor on $t$ vertices.

\begin{lemma}
	\label{lem:tournament_using_switches}
	Let $D$ be a digraph that contains $\switchesfortournamentfunction{t}$ disjoint subgraphs $D_1, \dots,$ $D_{\Fkt{f}{t}}$ such that $D_j$ is an $(i^j_1,\dots,i^j_t)$-$(o^j_1,\dots,o^j_t)$-switch.
	For every $1 \leq j < \switchesfortournamentfunction{t}$ let $\mathcal{L}_j = (P^j_1,\dots,P^j_t)$ be an $(o^j_1,\dots,o^j_t)$-$(i^{j+1}_1,\dots,i^{j+1}_t)$-linkage with $P^j_s$ is an $o^j_s$-$i^{j+1}_s$-path and let $\mathcal{L}_0= \Brace{P^0_1,\dots,P^0_t}$ be a linkage ending in the vertices $\Brace{i^{1}_1,\dots,i^{1}_t}$ and $\mathcal{L}_{\switchesfortournamentfunction{t}+1} = (P^{\switchesfortournamentfunction{t}+1}_1, \dots,$ $P^{\switchesfortournamentfunction{t}+1}_t)$ be a linkage starting in the vertices $(i^{\switchesfortournamentfunction{t}}_1,\dots,$ $i^{\switchesfortournamentfunction{t}}_t)$ such that
	$\mathcal{L}_0,\dots,$ $\mathcal{L}_{\switchesfortournamentfunction{t}}$ are internally disjoint from the switches.
	Also, let $\mathcal{R}_i \coloneqq \Set{R^{a,a+1}_i \mid 1 \leq a < t}$ for all $1 \leq i \leq \switchesfortournamentfunction{t}$ such that $R^{a,a+1}_i$ starts on the $a$-th path and ends on the $a+1$-th path of $\mathcal{L}_i$ and is disjoint to the switches and internally disjoint to the linkages $\mathcal{L}_0,\dots,\mathcal{L}_{\switchesfortournamentfunction{t}}.$
	Additionally, we require that the linkages $\mathcal{L}_1, \dots, \mathcal{L}_{\switchesfortournamentfunction{t}}$ are pairwise disjoint.
	Then, $D$ contains a $\Kdir{t}$ as a butterfly minor.
\end{lemma}
\begin{proof}
	We build butterfly contractible models for every vertex of $\Kdir{t}.$
	Let $v_1,\dots,v_t$ be the vertices of $\Kdir{t}$ such that all edges are of the form $\Brace{v_i,v_j}$ with $i < j.$
	The model of $v_k$ consists of a path $M_k,$ $k-1$ edges $e^{\variablestyle{in}(k)}_1,\dots,e^{\variablestyle{in}(k)}_{k-1}$ and $t-k$ edges $e^{\variablestyle{out}(k)}_{k+1},\dots,e^{\variablestyle{out}(k)}_{t}$ such that the vertices $\Head{e^{\variablestyle{in}(k)}_1},$ $\dots,$ $\Head{e^{\variablestyle{in}(k)}_{k-1}},$ $\Tail{e^{\variablestyle{out}(k)}_{k+1}},$ $\dots,$ $\Tail{e^{\variablestyle{out}(k)}_{t}}$ occur on $M_k$ in that order, see \cref{fig:tournament_minor_vertex_model} for an illustration.
	
	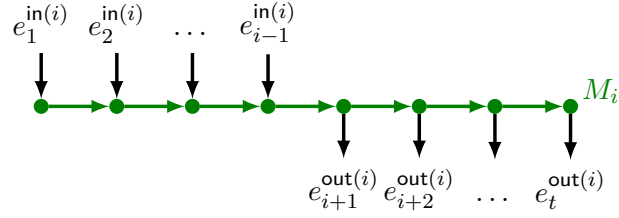
\begin{figure}[!ht]
		\centering
		\begin{tikzpicture}
			\def\hor{1}
			\def\vert{1.1}
			\foreach \i in {1,...,8}
			{
				\node[vertex,myGreen] (v-\i) at ($(\i*\hor,0)$) {};			
			}
			\foreach[evaluate={\j=int(\i+1);}] \i in {1,...,7}
			{
				\draw[directededge,myGreen] (v-\i) to (v-\j);			
			}
			\node (v-M) at ($(v-8)+(0.4,0.2)$) {\textcolor{myGreen}{$M_i$}};
			
			\node (e-1) at ($(1*\hor,\vert)$) {$e^{\variablestyle{in}(i)}_{1}$};
			\node (e-2) at ($(2*\hor,\vert)$) {$e^{\variablestyle{in}(i)}_{2}$};
			\node (e-3) at ($(3*\hor,\vert)$) {\vphantom{$e^{\variablestyle{in}(i)}_{1}$}$\dots$};
			\node (e-4) at ($(4*\hor,\vert)$) {$e^{\variablestyle{in}(i)}_{i-1}$};
			\foreach \i in {1,...,4}
			{
				\draw[directededge] (e-\i) to (v-\i);		
			}
			
			\node (e-5) at ($(5*\hor,-\vert)$) {$e^{\variablestyle{out}(i)}_{i+1}$};
			\node (e-6) at ($(6*\hor,-\vert)$) {$e^{\variablestyle{out}(i)}_{i+2}$};
			\node (e-7) at ($(7*\hor,-\vert)$) {\vphantom{$e^{\variablestyle{in}(i)}_{1}$}$\dots$};
			\node (e-8) at ($(8*\hor,-\vert)$) {$e^{\variablestyle{out}(i)}_{t}$};
			\foreach \i in {5,...,8}
			{
				\draw[directededge] (v-\i) to (e-\i);		
			}
		\end{tikzpicture}
		\caption{The butterfly minor model for a vertex of the transitive tournament.}
		\label{fig:tournament_minor_vertex_model}
	\end{figure}
	
	We construct $M_1,\dots,M_t$ inductively.
	We also define a function $f_k$ helping us to keep track of the path we construct.
	The path $M^1_k$ is defined as $P^0_k,$ which ends in the vertex $i^1_{f_k(1)},$ and we define $f_k(1) \coloneqq k$ for all $1 \leq k \leq t.$
	
	Now assume we have constructed a path $M^j_k$ ending in $i^j_{f_k(j)}$ for all $1 \leq k \leq t.$
	Let
	\begin{equation*}
		x \coloneqq \begin{cases}
			j \bmod \Brace{t-1} &, \text{ if } j \bmod \Brace{t-1}<0\\
			t-1 &, \text{ if } j \bmod \Brace{t-1}=0.
		\end{cases}
	\end{equation*}
	We use $D_j$ to switch the paths arriving at $i^j_x$ and $i^j_{x+1}.$
	So let $\mathcal{Q}^j$ be the $(i^j_1,\dots,i^j_t)$-$(o^j_1,\dots,o^j_t)$-linkage within $D_j$ that contains an $i^j_y$-$o^j_y$-path for every $y \notin \Set{x,x+1}$ and an $i^j_x$-$o^j_{x+1}$-path as well we an $i^j_{x+1}$-$o^j_x$-path.
	Obtain $M^{j+1}_k$ from $M^j_k$ by appending the path $Q$ from $\mathcal{Q}^j$ starting in $\End{M^j_k}$ and then the path $P^{j}_{k'} \in \mathcal{L}_{j+1}$ starting in $\End{Q}.$
	Define $f_k(j+1) \coloneqq k'.$
	
	Finally, we define $M_k \coloneqq M^{\switchesfortournamentfunction{t}}_k.$
	As $M_1,\dots,M_t$ each contain exactly one path from each linkage $\mathcal{L}_1,\dots,\mathcal{L}_{t},\mathcal{Q}_1,\dots,\mathcal{Q}_t$ they yield a $\Brace{i^1_1,\dots,i^1_t}$-$(o^{\switchesfortournamentfunction{t}}_1,\dots,$ $o^{\switchesfortournamentfunction{t}}_t)$-linkage.
	
	\begin{figure}[!ht]
		\centering
		\resizebox{\textwidth}{!}{%
			\begin{tikzpicture}
				\node (constr) at (-4,0) {\begin{tikzpicture}
						\def\hdist{2}
						\def\vdist{0.4}
						\def\width{0.7}
						\node (a-l-1) at (0,0) {};
						\node (b-l-1) at ($(a-l-1)+(\hdist,0)$) {};
						\node (c-l-1) at ($(b-l-1)+(\hdist,0)$) {};
						\node (d-l-1) at ($(c-l-1)+(\hdist,0)$) {};
						\node (e-l-1) at ($(d-l-1)+(\hdist,0)$) {};
						\foreach \x in {a,b,c,d,e}
						{
							\foreach[evaluate={\j=int(\i-1);}] \i in {2,3,4}
							{
								\node (\x-l-\i) at ($(\x-l-\j)+(0,-\vdist)$) {};
							}
							\foreach \i in {1,2,3,4}
							{
								\node (\x-r-\i) at ($(\x-l-\i)+(\width,0)$) {};
							}
						}
						
						\foreach \x in {a,b,c,d,e}
						{
							\draw[myRed,thick,rounded corners=true] ($(\x-l-1)+(-0.1,\vdist)$) rectangle ($(\x-r-4)+(0.1,-\vdist)$);
						}
						\node at ($(b-r-4)!0.5!(c-l-4)-(0,0.8)$) {\textcolor{myRed}{switches}};
						
						\foreach \i in {1,2,3,4}
						{
							\node (s-\i) at ($(a-l-\i)-(\hdist/2,0)$) {};
							\node (t-\i) at ($(e-r-\i)+(\hdist/2,0)$) {};
						}			
						
						\draw[->,-latex] (s-1) -- (a-l-1.center) -- (a-r-2.center) -- (b-l-2.center) -- (b-r-3.center) -- (c-l-3.center) -- (c-r-4.center) -- (t-4);
						\draw[markeddirected={myOrange}{black}] (s-2) -- (a-l-2.center) -- (a-r-1.center) -- (d-l-1.center) -- (d-r-2.center) -- (e-l-2.center) -- (e-r-3.center) -- (t-3);
						\draw[->,-latex] (s-3) -- (b-l-3.center) -- (b-r-2.center) -- (d-l-2.center) -- (d-r-1.center) -- (t-1);
						\draw[->,-latex] (s-4) -- (c-l-4.center) -- (c-r-3.center) -- (e-l-3.center) -- (e-r-2.center) -- (t-2);
						
						\draw[markeddirected={myOrange}{myGreen}] ($(s-1)!0.5!(a-l-1)$) to ($(s-2)!0.5!(a-l-2)$);
						\draw[directededge,myGreen] ($(a-r-2)!0.5!(b-l-2)$) to ($(a-r-3)!0.5!(b-l-3)$);
						\draw[directededge,myGreen] ($(b-r-3)!0.5!(c-l-3)$) to ($(b-r-4)!0.5!(c-l-4)$);
						\draw[markeddirected={myOrange}{myGreen}] ($(c-r-1)!0.5!(d-l-1)$) to ($(c-r-2)!0.5!(d-l-2)$);
						\draw[markeddirected={myOrange}{myGreen}] ($(d-r-2)!0.5!(e-l-2)$) to ($(d-r-3)!0.5!(e-l-3)$);
						\draw[directededge,myGreen] ($(e-r-1)!0.5!(t-1)$) to ($(e-r-2)!0.5!(t-2)$);
						
						\node at ($(s-1)+(-0.1,0.1)$) {\tiny{1}};
						\node at ($(s-2)+(-0.1,0.1)$) {\tiny{2}};
						\node at ($(s-3)+(-0.1,0.1)$) {\tiny{3}};
						\node at ($(s-4)+(-0.1,0.1)$) {\tiny{4}};
						
						\node at ($(a-r-1)+(0.3,0.1)$) {\tiny{2}};
						\node at ($(a-r-2)+(0.3,0.1)$) {\tiny{1}};
						\node at ($(a-r-3)+(0.3,0.1)$) {\tiny{3}};
						\node at ($(a-r-4)+(0.3,0.1)$) {\tiny{4}};
						
						\node at ($(b-r-1)+(0.3,0.1)$) {\tiny{2}};
						\node at ($(b-r-2)+(0.3,0.1)$) {\tiny{3}};
						\node at ($(b-r-3)+(0.3,0.1)$) {\tiny{1}};
						\node at ($(b-r-4)+(0.3,0.1)$) {\tiny{4}};
						
						\node at ($(c-r-1)+(0.3,0.1)$) {\tiny{2}};
						\node at ($(c-r-2)+(0.3,0.1)$) {\tiny{3}};
						\node at ($(c-r-3)+(0.3,0.1)$) {\tiny{4}};
						\node at ($(c-r-4)+(0.3,0.1)$) {\tiny{1}};
						
						\node at ($(d-r-1)+(0.3,0.1)$) {\tiny{3}};
						\node at ($(d-r-2)+(0.3,0.1)$) {\tiny{2}};
						\node at ($(d-r-3)+(0.3,0.1)$) {\tiny{4}};
						\node at ($(d-r-4)+(0.3,0.1)$) {\tiny{1}};
						
						\node at ($(e-r-1)+(0.3,0.1)$) {\tiny{3}};
						\node at ($(e-r-2)+(0.3,0.1)$) {\tiny{4}};
						\node at ($(e-r-3)+(0.3,0.1)$) {\tiny{2}};
						\node at ($(e-r-4)+(0.3,0.1)$) {\tiny{1}};
				\end{tikzpicture}};
				\node (K4) at (4,0) {
					\begin{tikzpicture}
						\def\dist{1.3}
						\node[vertex,label=315:{$v_1$}] (v-1) at (0,0) {};
						\node[vertex,label=45:{$v_2$}] (v-2) at (0,\dist) {};
						\node[vertex,label=135:{$v_3$}] (v-3) at (-\dist,\dist) {};
						\node[vertex,label=225:{$v_4$}] (v-4) at (-\dist,0) {};
						
						\draw[markeddirected={myOrange}{black}] (v-1) to (v-2);
						\draw[markeddirected={myOrange}{black}] (v-2) to (v-3);
						\draw[markeddirected={myOrange}{black}] (v-2) to (v-4);
						\draw[directededge] (v-1) to (v-3);
						\draw[directededge] (v-1) to (v-4);
						\draw[directededge] (v-3) to (v-4);
						
						\node (label) at ($(v-1)!0.3!(v-4) +(0,-0.8)$) {$\Kdir{4}$};
				\end{tikzpicture}};
		\end{tikzpicture}}
		\caption{An example how the construction looks for finding a $\Kdir{4}$ butterfly minor in the given graph, the model for the vertex $v_2$ is highlighted in \textcolor{myOrange!90!red}{orange}.}
		\label{fig:tournament_minor_using_switches_ex}
	\end{figure}
	
	We divide the linkages $\mathcal{L}_0,\dots,\mathcal{L}_{\switchesfortournamentfunction{t}+1}$ into $t-1$ consecutive blocks $\mathcal{B}_1,\dots,\mathcal{B}_{t-1}$ such that $\mathcal{B}_i$ contains $t-i$ linkages.
	Let $\mathcal{B}_i = \Brace{\mathcal{L}_{i_1},\dots,\mathcal{L}_{i_{t-i}}}$ and $\mathcal{R}_{i_j} \coloneqq \Set{R^{a,a+1}_{i_j} \mid 1 \leq a < t}.$
	We consider the path $R^{f_i(i_j),f_i(i_j)},$ which starts in $M_i$ and ends in $M_{i+j}$ due to our construction earlier.
	We add its first edge as $e^{\variablestyle{out}}_{i+j}$ to the model of $v_i$ and its last edge as $e^{\variablestyle{in}}_i$ to the model of $v_{i+j}.$
	The remaining path is the model of $\Brace{i,i+j}.$
	
	As we collect all the in-edges for the model of $v_k$ within the blocks $\mathcal{B}_1,\dots,\mathcal{B}_{k-1}$ and only collect the out-edges in block $\mathcal{B}_k,$ all in-edges are met by $M_k$ before the out-edges, making the model of $v_i$ butterfly contractible.
	See~\cref{fig:tournament_minor_using_switches_ex} for an illustration of the~construction.
\end{proof}

Now, we can prove that for every transitive tournament there is a sufficiently large cross-row grid containing it as a butterfly minor. 

\begin{lemma}
	\label{lem:cross-row_grid_contains_tournament.}
	The \hyperlink{cross-row grid}{cross-row grid} of order $\switchesfortournamentfunction{t} \cdot t +2$ contains $\Kdir{t}$ as a butterfly minor.
\end{lemma}
\begin{proof}
	We define $k \coloneqq \switchesfortournamentfunction{t}\cdot t +2.$
	Let $a_1,\dots,a_k$ be the vertices of $\outPath_1$ and $b_1,\dots,b_k$ be the vertices of $\inPath_1$ in order of appearance along the paths.
	We define $\switchesfortournamentfunction{t}$ many subgraphs $D_1,\dots,D_{\switchesfortournamentfunction{t}},$ where $D_i$ is the subgraph of $D$ containing the vertices $a_{k-\Brace{i-1}t},\dots,a_{k-it+1},b_{\Brace{i-1}t},\dots,b_{it}$ and all edges between them.
	Note that $D_i$ is a $\Brace{a_{k-\Brace{i-1}t},\dots,a_{k-it+1}}$-$\Brace{b_{\Brace{i-1}t},\dots,b_{it}}$-switch.
	
	We choose $\mathcal{L}_0$ to be the subpaths of $C_{k},\dots,C_{k-t+1}$ starting in $\inPath_{k}$ and ending in the vertices $a_{k},\dots,a_{k-t+1}.$
	For $\mathcal{L}_i$ with $1 \leq i \leq \switchesfortournamentfunction{t}$ we construct paths from $(b_{\Brace{i-1}t}, \dots,$ $b_{it})$
	along $(C_{\Brace{i-1}t}, \dots,$ $C_{it})$
	until meeting $(\inPath_{k-1-\Brace{i-1}t}, \dots,$ $\inPath_{k-it}),$
	then following the paths $(\inPath_{k-1-\Brace{i-1}t}, \dots,$ $\inPath_{k-it})$
	until meeting the cycles $(C_{k-\Brace{i}t}, \dots,$ $C_{k-\Brace{i+1}t+1})$
	which the paths of $\mathcal{L}_i$ then follow until finally reaching $(a_{k-\Brace{i}t}, \dots,$ $a_{k-\Brace{i+1}t+1}).$
	Next, we define the linkage $\mathcal{L}_{\switchesfortournamentfunction{t}+1}$ as the subpaths of $(C_{\Brace{\switchesfortournamentfunction{t}-1}t},\dots,$ $C_{\switchesfortournamentfunction{t}t})$ starting in $(b_{\Brace{\switchesfortournamentfunction{t}-1}t},\dots,$ $b_{\switchesfortournamentfunction{t}t})$ and ending on $\inPath_{2}.$
	
	Finally we construct families $\mathcal{R}_i$ for all $1 \leq i \leq \switchesfortournamentfunction{t}+1.$
	For $1 < i \leq \switchesfortournamentfunction{t}+1$ we obtain the subpaths between the paths in $\mathcal{L}_i$ from $\inPath_2.$
	And for $i=1$ we obtain them from $\inPath_k.$
	
	Then, we have everything we need to apply \cref{lem:tournament_using_switches} and thus, obtain the desired $\Kdir{t}$ as a butterfly minor.
\end{proof}

Clearly, for every cross-row grid there is bidirected clique containing it as a minor.

\begin{lemma}
	\label{lem:clique_contains_cross-row-wall}
	The bidirected clique $\K{\cliqueToCrossRowGridFunction{t}}$ contains the cross-row grid $\CrossRowGrid{t}$ as a butterfly minor.
\end{lemma}
\begin{proof}
	This directly follows from the cross-row grid of order $t$ having $2t^2$ vertices.
	We fix a bijection $\varphi: \V{\K{\cliqueToCrossRowGridFunction{t}}} \to \V{\CrossRowGrid{t}}$ and then delete from $\K{\cliqueToCrossRowGridFunction{t}}$ every edge $\Brace{u,v}$ with $\Brace{\Fkt{\varphi}{u}, \Fkt{\varphi}{v}} \notin \E{\CrossRowGrid{t}}.$
	This yields the desired butterfly minor.
\end{proof}

Basically, a cross-row grid consists of many slices in a wall that each contain a local cross.
We formalise this concept by the following definition.

\begin{restatable}{definition}{CrossSliceDefinition}
	\label{def:cross_slice}
	A \define{cross} in a slice of a wall $W$ between $Q_i$ and $Q_{i+\ell}$ are two paths $J_1$ and $J_2$ with $\Start{J_1},\End{J_2} \in \V{Q_i}$ and $\Start{J_2},\End{J_1} \in \V{Q_{i+\ell}}$ and for the rows $\row_{a}$ containing $\Start{J_1},$ $\row_{a'}$ containing $\End{J_1},$ $\row_{b}$ containing $\Start{J_2}$ and $\row_{b'}$ containing $\End{J_2}$ we have $\max\Set{a,a'}-\min\Set{b,b'} > 1.$
	We call a slice with a cross in it a \define{cross~slice}.
\end{restatable}

In order to construct the cross-row grid as a minor in a given digraph it suffices to show that there is a wall in the graph that contains enough cross slices and still has some rows that remain untouched by the crosses.

\begin{lemma}
	\label{lem:many_cross_slices_yield_minor}
	Let $D$ be a digraph, $W$ a cylindrical wall in $D$ and $i \in\N.$
	If $W$ contains $t$ disjoint slices which each contain a cross, and there are $t$ rows in $W$ such that all $t$ crosses lie in a strip not containing these $t$ rows, then $W$ contains $\CrossRowGrid{t}$ as a butterfly minor.
\end{lemma}
\begin{proof}
	Let $S_i,\dots,S_t$ be $t$ \hyperref[def:slice]{slices} such that $S_i$ is the slice between $Q_{j_i}$ and $Q_{j_i + \ell_i}$ and contains a cross $J^i_1$ and $J^i_2.$
	Let $W_{\row}$ be the \hyperref[def:strip]{strip} containing all crosses but not the rows $\row'_1,\dots,\row'_{t}$ and let $\row_a$ and $\row_b$ such that $W_{\row}$ is the strip between $\row_a$ and $\row_b.$
	We construct a new row containing $t$ crosses by valid butterfly minor operations within the subwall between $\outPath_a$ and $\inPath_b$ as follows.
	First, we delete all in- and out-paths within $W_{\row}$ except for $\outPath_a$ and $\inPath_b.$
	Next, for every slice $S_i$ we can contract the following paths into single vertices:
	\begin{enumerate}
		\item the subpath of $Q_{j_i}$ starting on $\outPath_a$ and ending in $\Start{J^i_1},$ we refer to this contraction vertex as $\Start{J^i_1}',$
		\item the subpath of $Q_{j_i}$ starting in $\End{J^i_2}$ and ending on $\inPath_b,$ we refer to this contraction vertex as $\End{J^i_2}',$
		\item the subpath of $Q_{j_i+\ell_i}$ starting on $\outPath_a$ and ending in $\Start{J^i_2},$ we refer to this contraction vertex as $\Start{J^i_2}',$ and
		\item the subpath of $Q_{j_i+\ell_i}$ starting in $\End{J^i_1}$ and ending on $\inPath_b,$ we refer to this contraction vertex as $\End{J^i_1}'.$
	\end{enumerate}
	This leaves us with one row consisting of out-path $\outPath_a$ and in-path $\inPath_b$ containing $t$ \hyperref[def:cross_slice]{crosses}.
	In order to make this a proper cross-row grid, we delete all vertical paths but $Q_{j_1},\dots,Q_{j_t}$ and $Q_{j_t+\ell_t}.$
	Now, for every $i \in [t],$ we can contract the subpath of $\outPath_a$ that starts in $\Start{J^i_2}'$ and ends in $\Start{J^{i+1}_1}'$ into a single vertex.
	Similarly, for every $i \in [t],$ we can contract the subpath of $\inPath_b$ that starts in $\End{J^{i+1}_2}'$ and ends in $\End{J^{i}_1}'$ into a single~vertex.
	
	This newly constructed row together with the rows $\row'_1,\dots,$ $\row'_{t}$ and the vertical paths $Q_{j_1},\dots,$ $Q_{j_t}$ and $Q_{j_t+\ell_t}$ now yields the claimed $\CrossRowGrid{t}$ butterfly minor.
\end{proof}

\section{Drawings and weak renditions}

Now that we have seen how the cross-row grid relates to the traditional excluded minors, we give some context for surfaces and embeddings.
Formally, a \define{surface} is a compact 2-dimensional manifold with or without a boundary.
However, the only two surfaces we explicitly use are the sphere $\sphere$ and the torus.

\begin{definition}
	A \define{drawing (with crossings)} of a digraph $D$ on a surface $\Sigma$ is a tuple $\Gamma = \Brace{U,V,E,\varphi}$ such that
	\begin{itemize}
		\item $\varphi: V \cup E \to \V{D}\cup \V{E}$ is a bijection such that $\Reduct{\varphi}{E}$ is a bijection between $E$ and $\E{D}$ and $\Reduct{\varphi}{V}$ is a bijection between $V$ and $\V{D}$
		\item $V \subseteq U \subseteq \Sigma$ and $V \cup \bigcup_{e\in E} e = U$
		\item for every $e\in E,$ $e = \Fkt{h_e}{(0, 1)},$ where $h_e : [0, 1] \to U$ is a homeomorphism onto its image with $\Fkt{h_e}{0}, \Fkt{h_e}{1} \in V$
		\item $V$ is disjoint from every $e \in E$
		\item $\Brace{u,v}\in\E{D}$ if and only if $\Fkt{\varphi}{\Fkt{h_e}{0}} = u$ and $\Fkt{\varphi}{\Fkt{h_e}{1}} = v$
		\item if $e,e' \in E$ are distinct, then $e\cap e'$ is finite.
	\end{itemize}
	
	If $e,e' \in E$ with $e \cap e' \neq \emptyset$ and $e \neq e',$ then we say $e$ and $e'$ \define{cross}.
	A drawing is \define{cross-free} if there are no two elements in $E$ that cross.
\end{definition}

We say a drawing $\Gamma_1 = \Brace{U_1,V_1,E_1,\varphi_1}$ of a digraph $D$ is \defineSubindexExtraTarget{consistent}{drawing (with crossings)}{drawingconsistent} with a drawing $\Gamma_2 = \Brace{U_2,V_2,E_2,\varphi_2}$ of a subgraph $D' \subseteq D$ if $U_2 \subseteq U_1,$ $V_2 \subseteq V_1,$ $E_2\subseteq E_1$ and $\Fkt{\varphi_1}{v} = \Fkt{\varphi_2}{v}$ for all $v \in \V{D'}.$

We often consider drawings of parts of some digraph into a closed disk, in that case we are interested in which vertices are drawn into the boundary in order to describe interaction with the remaining graph.
To this end we introduce the following definition, which basically generalises the concept of \emph{societies} by Kawarabayashi, Thomas and Wollan~\cite{kawarabayashi2020quickly}.

\begin{restatable}[Society]{definition}{SocietyDefinition}
	\label{def:society}
	Let $\Omega$ be a cyclic order of the elements of some set and let $\V{\Omega}$ denote this set.
	A \define{society} is a pair \MarginOnly{$\Brace{D,\Omega}$}$\Brace{D,\Omega},$ where $D$ is a digraph, and $\Omega$ is a cyclic order with $\V{\Omega}\subseteq\V{D}.$
	
	A \define{cylindrical society} is a tuple \MarginOnly{\resizebox{\marginparwidth}{!}{$\Brace{D,\Omega_1,\Omega_2}$}} $\Brace{D,\Omega_1,\Omega_2},$ where $D$ is a digraph and $\Omega_1,$ $\Omega_2$ are cyclic orders with $\V{\Omega_i}\subseteq\V{D}$ for both $i\in[2]$ and $\V{\Omega_1}\cap\V{\Omega_2} = \emptyset.$
	
	A subset $X \subseteq \V{\Omega}$ is called a segment if there are no vertices $x_1,x_2 \in X$ and $y_1,y_2 \in \V{\Omega}$ such that $x_1,y_1,x_2,y_2$ occur in $\Omega$ in that order.
	The ordering $\Omega$ naturally induces a linear ordering on its segments. 
	We write \Margin{$a\Omega b$} for the unique segment of $\Omega$ that has $a$ as its first vertex and $b$ as its last vertex.
\end{restatable}

The following definition helps us to describe the structure inside a society, especially non-planar structure.

\begin{definition}[Transactions]
	Let $\Brace{D,\Omega}$ be a society.
	A path $P$ is an \emph{$\Omega$-path} if $\V{P} \cap \V{\Omega} = \Set{\Start{P},\End{P}},$ that is, $P$ is a $\V{\Omega}$-path.
	A linkage $\mathcal{P}$ in $D$ is called a \define{transaction} in $\Brace{D,\Omega}$ if every $P \in\mathcal{P}$ is an {$\Omega$-path} and there are two disjoint segments $X$ and $Y$ of $\Omega$ such that $\Set{\Start{P}\mid P\in \mathcal{P}} \subseteq X$ and $\Set{\End{P}\mid P \in \mathcal{P}} \subseteq Y.$
	The \define{depth} of $\Brace{D,\Omega}$ is defined as the maximum order of a transaction in $\Brace{D,\Omega}.$
	
	The two endpoints $\Start{P}$ and $\End{P}$ of an $\Omega$-path $P$ split $\Omega$ into two segments.
	If another $\Omega$-path $P'$ has its start-vertex in the one and its end-vertex in the other segment, then $P$ and $P'$ build a \define{cross} in $\Brace{D,\Omega}.$
	A transaction is called \define{planar} if no two paths in it build a cross.
\end{definition}

For societies in undirected graphs having no cross and being 4-connected suffices to ensure a planar embedding of the graph into a disk.
This is not sufficient in digraphs.
Still, we would like to reduce the non-planarity of the graphs to smaller, more controllable regions.
The next definition allows us to fix a part of a given graph that allows for a drawing without crossings in the sphere.

\begin{restatable}[Skeleton]{definition}{SkeletonDefinition}
	\label{def:skeleton}
	Let $D$ be a digraph, $r \in \N$ and $W$ be a wall of order $r$ in $D.$
	A skeleton of $D$ is a tuple $\skeleton = \Brace{\Gamma,W,\mathsf{L}}$ with
	\begin{enumerate}
		\item $\mathsf{L}$ is a family of linkages in $D,$
		\item $\Gamma$ is a planar drawing of the graph $\rigidgraph{\skeleton} \coloneqq  W \cup \bigcup_{\mathcal{L} \in \mathsf{L}} \mathcal{L},$
		\item if $\mathsf{L} = \emptyset,$ then the two faces bound by the wall perimeters are called the \emph{big faces} of $\skeleton,$ otherwise there exists a linkage $\mathcal{L} \in \mathsf{L}$ such that 
		\begin{equation*}
			\skeleton' = \Brace{\Reduct{\Gamma}{W\cup\Brace{\mathsf{L}\setminus\Set{\mathcal{L}}}}, W, \Brace{\mathsf{L}\setminus\Set{\mathcal{L}}}}
		\end{equation*}
		is a skeleton and $\mathcal{L}$ is a planar transaction on a big face $f$ of $\skeleton'.$
		In the ladder case, the \emph{big faces} of $\skeleton$ are the big faces of $\skeleton'$ without $f$ but adding all faces bound by $f$ and one of the two outer paths of $\mathcal{L}.$
	\end{enumerate}
	We call the skeleton $\skeleton$ \defineSubindex{centred at the wall $W$}{skeleton}.
	We call \MarginOnly{$\rigidgraph{\skeleton}$}$\rigidgraph{\skeleton}$ the \define{rigid subgraph} of $\skeleton.$
	The \defineSubindex{order}{skeleton} of $\mathcal{S}$ is half the minimum order of an element in $\mathsf{L} \cup \Set{W}.$
\end{restatable}

Intuitively, in order to describe that non-planar behaviour of the digraph around the rigid subgraph of a skeleton is restricted, we want to say that closed curves in the rigid part have a separating property.
We cannot demand them to yield proper separations, so we describe a slightly weaker notion where we cut out the curve together with some part of the rigid graph around it.

A \define{noose} of a digraph $D$ within a fixed drawing $\Gamma$ into a surface $\Sigma$ is a closed curve that bounds a disk and only intersects $\Gamma$ in vertices of $D.$

\begin{definition}
	Let $\skeleton \coloneqq \Brace{\Gamma,W,\mathsf{L}}$ be a skeleton of a digraph $D$ in the sphere.
	Let $\Gamma^{+}$ be a drawing of $D$ in $\Sigma$ that is \hyperlink{drawingconsistent}{consistent} with $\Gamma$ and $C$ a noose of $\rigidgraph{\skeleton}$ within $\Gamma^{+}.$
	Then, the graph $\rigidgraph{\skeleton}-C$ consists of at most two weakly 2-connected components $H'_1$ and $H'_2.$
	For $i \in [2],$ let $H_i$ be the weakly 2-connected component not containing $C$ obtained by removing the face of $H'_i$ that is no face in $\Gamma$ from $\rigidgraph{\skeleton}.$
	We define \MarginOnly{$H_i^{+}$}$H_i^{+}$ to be the subgraph of $D$ induced by $\V{H_i} \cup \bigcup_{\substack{f \text{ face in }\Reduct{\Gamma^{+}}{H_i}\text{ and in }\Gamma^{+}}} \Set{v \mid \text{the vertex $v$ is drawn into $f$ by $\Gamma^{+}$}}.$
\end{definition}

Using this we can define our demands on a drawing of the digraph around a wall in the sphere.

\begin{restatable}[$\Sigma$-decomposition]{definition}{SigmaDecompositionDefinition}
	\label{def:sphere_decomposition}
	\label{def:sigma_decomposition}
	Let $\skeleton \coloneqq \Brace{\Gamma,W,\mathsf{L}}$ be a skeleton of a digraph $D$ in the surface $\Sigma$ and $A \subseteq \V{D}$ a set of vertices in $D$ such that $\V{\rigidgraph{\skeleton}} \subseteq \V{D}\setminus A.$
	Let $\Gamma^{+}$ be a drawing of the strongly connected component of $D-A$ containing $\rigidgraph{\skeleton}$ in $\Sigma$ that is consistent with $\Gamma.$
	The tuple $\rho = \Brace{\Gamma^{+},\skeleton}$ is a \defineSorted{$\Sigma$-decomposition}{Sigma-decomposition} of $D$ in $\Sigma$ with apex set $A,$ if
	For all undirected cycles $C$ in $\rigidgraph{\skeleton}$ there is no directed path from $H^{+}_1$ to $H^{+}_2$ as well as no directed path from $H^{+}_2$ to $H^{+}_1.$
	
	Let $N(\rho),$ the set of \define{nodes} of $\rho,$ be the set of all vertices in $\rigidgraph{\skeleton}.$
	If $\Gamma = \Brace{U,V,E,\varphi},$ we refer to $\varphi^{-1}$ by $\varphi_{\rho}.$
	
	Let $\Sigma$ and $\Sigma'$ be two surfaces, where $\Sigma'$ has strictly higher genus.
	Let $\rho = \Brace{\Gamma^{+},\skeleton}$ be a $\Sigma$-decomposition and $\rho'= \Brace{\Gamma'^{+},\skeleton'}$ be a $\Sigma'$-decomposition such that $\rigidgraph{\skeleton} = \rigidgraph{\skeleton'}.$
	If $\Gamma'^{+}$ is consistent with $\Gamma^{+},$ then we say $\rho'$ is \defineSubindex{centred at}{Sigma-decomposition@$\Sigma$-decomposition} $\rho.$
\end{restatable}

There are certain crosses we cannot forbid by excluding a cross-row grid, so we need a notion of flatness that accommodates these crosses.
In order to achieve this, we define renditions into a disk that internally are nearly planar except for small exceptions which cover the mentioned crosses.

\begin{restatable}[Weak rendition]{definition}{WeakRenditionDefinition}
	\label{def:weak_rendition}
	Let $\Brace{D,\Omega}$ be a society and $\Delta$ be a closed disk, possibly with an open hole.
	A weak rendition of $\Brace{D,\Omega}$ into $\Delta$ is a $\sphere$-decomposition $\rho = \Brace{\Gamma^{+},\skeleton=\Brace{\Gamma,W,\emptyset}}$ such that
	\begin{enumerate}
		\item $\V{\Gamma^{+}} \subseteq \Delta,$
		\item one of the cyclic orderings of $\boundary{\Delta}$ maps to the image of $\Fkt{\varphi_{\rho}}{N(\rho) \cap \boundary{\Delta}} = \V{\Omega}.$
	\end{enumerate}
	
	Now, let $\Delta'$ be obtained from a closed disk $\Delta''$ by removing an open disk disjoint from the boundary of $\Delta''$ and let $\Brace{D,\Omega_1,\Omega_2}$ be a cylindrical society.
	Let $B_1$ and $B_2$ be the two closed curves in $\Delta'$ whose union equals $\boundary{\Delta'}.$
	Observe that for each $i\in[1,2]$ the curve $B_i$ bounds a closed disk $\Delta_i'$ with an open hole that contains $\Delta'.$
	A \define{weak rendition in $\Delta'$} is a $\sphere$-decomposition $\Brace{\Gamma^{+},\skeleton=\Brace{\Gamma,W,\emptyset,\emptyset}}$ such that for both $i\in[2],$ $\delta$ is a cylindrical rendition of $\Brace{D,\Omega_i}$ in the disk $\Delta_i'.$
	We say that $\Brace{D,\Omega_1,\Omega_2}$ has a \emph{weak rendition in the disk} if there exists $\Delta'$ as above such that $\Brace{D,\Omega_1,\Omega_2}$ has a weak rendition in $\Delta'.$
\end{restatable}

These tools enable us to describe drawings of digraphs that are consistent with the unique embedding of some wall they contain and additionally does have restricted non-planar behaviour with respect to this wall.

\section{Flat wall theorem}
\label{sec:flat_wall}

Having introduced weak renditions we can define our notion of flatness, which is slightly different from the notion of \cref{def:weakly_flat}.
In particular, it uses two directed separations instead of one undirected separation and it uses our concept of weak renditions, that is, removing a \hyperlink{noose}{noose} in the wall separates the two remaining parts of the wall.

\begin{definition}[Bridge]
	\label{def:bridge}
	Let $H$ be a subgraph of a digraph $D.$
	A weakly connected component $C$ in $D-H$ is called an \defineSorted{$H$-bridge}{Hbridge} if
	\begin{enumerate}
		\item there is a non-empty set $O \subseteq \V{C}$ such that for every $o \in O$ there is a vertex $h \in H$ with $\Brace{o,h}\in\E{D},$ $h$ is called an \define{out-attachment},
		\item there is a non-empty set $I \subseteq \V{C}$ such that for every $i \in I$ there is a vertex $h \in H$ with $\Brace{h,i}\in\E{D},$ $h$ is called an \define{in-attachment}, and
		\item for every vertex $x \in \V{C}$ there is a directed path $P$ in $C$ such that $\Start{P} \in I,$ $\End{P} \in O$ and $x \in \V{P}.$ \qedhere
	\end{enumerate}
\end{definition}

\begin{restatable}[Flat wall]{definition}{FlatWallDefinition}
	\label{def:flat_wall}
	Let $D$ be a digraph, $A\subseteq\V{D}$ be a set of vertices and $W \subseteq D-A$ be a wall of order $k+4.$
	We define $W^{-} \subseteq W$ to be the wall of order $k$  such that $W^{-}$ is disjoint from $Q_1,Q_2,Q_{k+3}$ and $Q_{k+4}.$
	Moreover, we define the \define{border} of $W$ as \MarginOnly{\resizebox{\marginparwidth}{!}{$\border{W}$}}$\border{W} \coloneqq W - W^{-}.$
	Let $D'$ be the strongly connected component of $D$ containing $W.$
	\renewcommand{\marginnotevadjust}{-2.8ex}
	The \define{compass} of $W,$ written \MarginOnly{\resizebox{\marginparwidth}{!}{$\compass{W}$}} $\compass{W}$ is the union of the border of $W$ and all $W^{-}$-bridges in $D'- \border{W}.$
	\renewcommand{\marginnotevadjust}{0ex}
	We say that $W$ is a \define{flat wall under $A$} if 
	\renewcommand{\labelenumi}{\textbf{\theenumi}}
	\renewcommand{\theenumi}{(F\arabic{enumi})}
	\begin{enumerate}[labelindent=0pt,labelwidth=\widthof{\ref{last-item-flat-wall}},itemindent=1em]
		\item \label{flat:apex_set} $\V{W}\cap A=\emptyset,$
		\item \label{flat:separation} there are two directed separations $\Separation{Y_1}{X_1}{}{r}$ and $\Separation{X_2}{Y_2}{}{r}$ in $D'$ such that $X_1$ and $X_2$ both contain $W$ and $Y_1 \cap X_1 = A \cup \border{W} = X_2 \cap Y_2$ and additionally in $D'-\Brace{A \cup \border{W}}$ every vertex in $X_1$ is reachable from $W^{-}$ and every vertex in $X_2$ reaches $W^{-},$
		\item \label{flat:rendition} the cylindrical society $\Brace{\compass{W},\Omega_1,\Omega_2}$ with the sets $\V{\Omega_1}=\V{Q_1}$ and $\V{\Omega_2}=\V{Q_{k+4}}$ has a \hyperref[def:weak_rendition]{weak rendition} $\rho = \Brace{\Gamma^{+},\skeleton=\Brace{\Gamma,W,\emptyset}}$ in the disk,
		\item \label{flat:torus} $D'$ has a torus decomposition centred at $\rho.$ \label{last-item-flat-wall}
	\end{enumerate}
	In case $A = \emptyset,$ we say that $W$ is \define{flat}.
\end{restatable}

Our main theorem is a directed flat wall theorem that excludes the cross-row grid as a butterfly minor.

\begin{restatable}[Directed flat wall theorem]{theorem}{directedFlatWallThm}
	\label{thm:directed_flat_wall}
	There exist functions $\WallFkt{} \colon \N \times \N \rightarrow \N$ and $\WallApexFkt{} \colon \N \rightarrow \N$ such that for all integers $r,t \geq 1$ and all digraphs $D$ that do not contain $\CrossRowGrid{t}$ for every $\WallFkt{r,t}$-wall $W$ in $D$ there exist a set $A \subseteq \V{D}$ with $\Abs{A} \leq \WallApexFkt{t}$ and an $r$-wall $W'\subseteq W-A$ which is flat under $A.$
\end{restatable}

We refer to the functions $\WallFkt{}$ and $\WallApexFkt{}$ from \cref{thm:directed_flat_wall} as such globally.
Please note that this yields a in some respect stronger statement than \cref{thm:directed_nearly_weakly_flat_wall,thm:directed_weakly_flat_wall}, because it refers to every wall in the digraph.
While \cref{thm:directed_nearly_weakly_flat_wall,thm:directed_weakly_flat_wall} provide the existence of their respective flat wall in the given digraph $D$, our result \cref{thm:directed_flat_wall} provides the flat wall inside \emph{every} wall that is large enough.
This way one can start out with any large enough wall in a given digraph and can be sure that it contains a flat wall, which is a desirable property especially for algorithmic usage.
Additionally \cref{thm:directed_flat_wall} obtains a wall that does allow for any cross-rows by excluding a strongly connected digraphs as a butterfly minor, thus combining the features of \cref{thm:directed_nearly_weakly_flat_wall,thm:directed_weakly_flat_wall}.
We mostly follow the arguments and techniques from \cite{giannopoulou2020directed}.
Some refinements are due to Gianopoulou and Wiederrecht~\cite{giannopoulou2021flat}.
In some places the proofs need adjustments propagating through non-trivial steps of the arguments, therefore, we present the whole proof containing parts from both~\cite{giannopoulou2020directed, giannopoulou2021flat}.

\begin{restatable}[Tiling]{definition}{TilingDefinition}
	\label{def:tiling}
	A \define{tiling} is a family of pairwise disjoint \hyperref[def:tile]{tiles}, and a tiling is said to \define{cover} a subwall $W'$ of $W$ if every branch vertex of $W'$ occurs in one of the tiles of the tiling.
	For every function $f_w: \N \to \N$ and all $\xi,$ $\xi' \in [\Fkt{f_w}{t}+1]$ we define the tiling $\tiling_{W,k,\Fkt{f_w}{t},\xi,\xi'}$ with tiles of width $\Fkt{f_w}{t}.$
	In order to do so, we define two function, the \define{column function}
	\begin{equation*}
		\columnFunction{p}{\xi,f_w} \coloneqq \Brace{k+1-\xi} + \Brace{p-1}\Brace{2\cdot \Fkt{f_w}{t}+1},
	\end{equation*}
	and the \define{row function}
	\begin{equation*}
		\rowFunction{q}{\xi',f_w} \coloneqq \Brace{1+\xi'} + \Brace{q-1}\Brace{2\cdot \Fkt{f_w}{t}+1}.
	\end{equation*}
	For both the column and the row function we omit $f_w,$ $\xi$ and $\xi'$ from the indices if they are clearly provided by the context.
	Then, define
	\begin{align*}
		\tiling_{W,k,\Fkt{f_w}{t},\xi,\xi'} \coloneqq{}
		\Big\{\tile_{\columnFunction{p}{},\rowFunction{q}{},\Fkt{f_w}{t}} \mid &1 \leq p \leq \Ceil{\frac{k+\xi-1}{2\cdot \Fkt{f_w}{t}+1} + 1},\\
		& 1 \leq q \leq \Ceil{\frac{3k - \Brace{1+\xi'}}{2\cdot \Fkt{f_w}{t}+1} + 1}\Big\}.
	\end{align*}
	See \cref{fig:tiling} for an illustration.
\end{restatable}

\begin{figure}[thb]
	\centering
	\resizebox{.9\textwidth}{!}{%
    \includegraphics{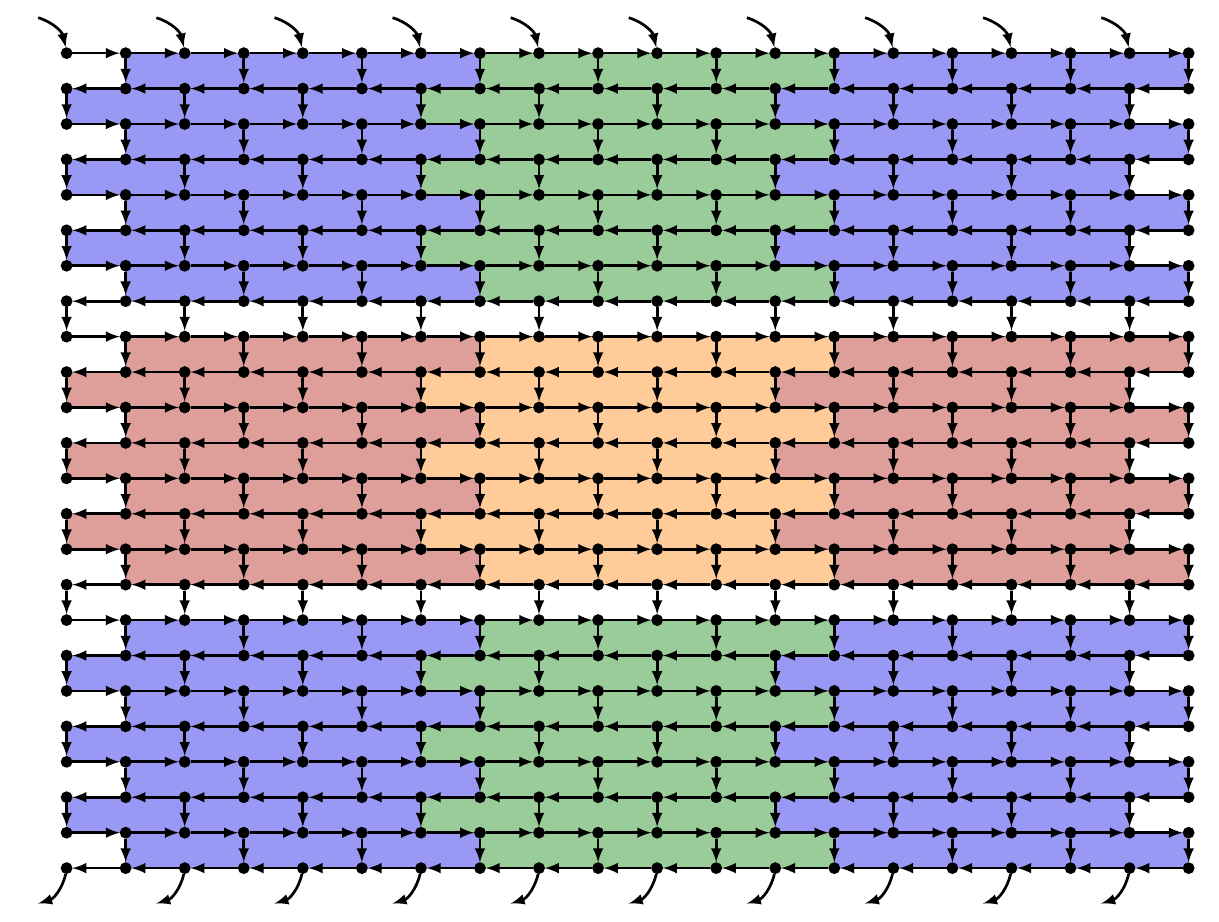}
    }
	\caption{A tiling of a wall together with a four-colouring of it.
		We also see that the yellow tile in the middle is surrounded by eight tiles of different colours.}\label{fig:tiling}
\end{figure}

Note that every \hyperref[def:tiling]{tiling} $\tiling_{W,k,\Fkt{f_w}{t},\xi,\xi'}$ covers $W_2.$
Moreover, every brick of $W_2$ that lies between the two paths of $\row_i$ for some $i \in [3k]$ is the centre of some tile $\tile'$ of some tiling $\tiling' \in \tiling_{W,k,\Fkt{f_w}{t},\xi,\xi'}.$
Hence, if we switch the parametrisation of $W,$ we are able to find in total $2\Brace{\Fkt{f_w}{t}+1}^2$ many tilings that cover $W_2,$ and every brick of $W_2$ is the centre of some tile in one of these tilings.
This number of tilings becomes relevant in the proof of \cref{thm:directed_flat_wall}.

A \defineExtraindex{colouring}{of a tiling} of $\tiling$ is a partition of $\tiling$ into four classes, namely $\Class_1,\Class_2,\Class_3,$ and $\Class_4$ as follows.
For every $i \in \left[\Ceil{\frac{k+\xi-1}{2\Fkt{f_w}{t}+1}}+1\right]$ and every $j\in \left[\Ceil{\frac{3k-\xi'-1}{2\Fkt{f_w}{t}+1}}+1\right]$ we assign to $\tile_{\columnFunction{i}{},\rowFunction{j}{},\Fkt{f_w}{t}}$ the colour $\Brace{i\bmod 2}+2\Brace{j\bmod 2}+1.$
This means that to tiles where $\columnFunction{i}{}$ and $\rowFunction{j}{}$ are even we assign the colour one, to tiles where $\rowFunction{j}{}$ is even but $\columnFunction{i}{}$ is odd we assign the colour three, and so on, see \cref{fig:tiling} for an example.
Hence, every column and every row is two-chromatic, and between each pair of tiles from the same colour that share a row or a column, there is a tile of a different colour that separates those tiles in their respective row or column.
Additionally, the eight tiles surrounding a tile $\tile$ are all of a different colour than $\tile$ itself.

We use tilings in several different ways, and sometimes it is necessary to \say{zoom out} of our current wall, \ie to abstract over some of the horizontal paths and vertical cycles in order to obtain a more streamlined version of our wall.

\begin{definition}[Walls from a Tiling]
	\label{def:tiling_wall}
	Let $k,d \in \N$ be positive integers, $W$ a cylindrical wall of order $3k$ with the triadic partition $\Triadic = (W, k, W_1, W_2, W_3, W^1, W^2,$ $W^3),$ and $\tiling$ a \hyperref[def:tiling]{tiling} of width $d$ that covers $W_2.$
	Moreover, let $\widetilde{W}$ be some \hyperref[def:slice]{slice} of $W_2$ and let $I_Q$ be the largest set of integers such that for every $i\in I_Q$ the vertical cycle $Q_i$ contains vertices of a tile from $\tiling$ which intersects $\widetilde{W}.$
	Let \MarginOnly{$\Fkt{\widetilde{W}}{\tiling}$}$\Fkt{\widetilde{W}}{\tiling}$ be the union of the cycles $Q_i$ with $i \in I_Q,$ and the paths $\InducedSubgraph{\inPath_j}{\CondSet{Q_h}{h\in I_Q}}$ and $\InducedSubgraph{\outPath_j}{\CondSet{Q_h}{h\in I_Q}}$ for every $j\in [3k].$
	We call $\Fkt{\widetilde{W}}{\tiling}$ the \index{extension of $\protect{\widetilde{W}}$ that covers $\tiling$}\emph{extension of $\widetilde{W}$ that covers $\tiling.$}
	
	Now, let $\Set{\Class_1,\dots,\Class_4}$ be a four-colouring of $\tiling$ and $i\in [4]$ be a fixed colour.
	Then, let $J_Q \subseteq [3k]$ be the largest set of integers such that for all $j \in J_Q$ the vertical cycle $Q_j$ of $\Fkt{\widetilde{W}}{\tiling}$ does not contain a vertex of some tile from $\Class_i.$
	Similarly, let $J_P\subseteq [3k]$ be the largest set of integers such that for every $j \in J_P,$ none of the two paths from $\row_j$ contains a vertex of a tile from $\Class_i.$
	
	By \MarginOnly{$\InducedSubgraph{\widetilde{W}}{\tiling,i}$}$\InducedSubgraph{\widetilde{W}}{\tiling,i}$ we denote the subgraph of $W$ induced by the union of the cycles $Q_i'$ with $i'\in J_Q,$ and the paths $\InducedSubgraph{\inPath_j}{\CondSet{Q_h}{h\in J_Q}}$ and $\InducedSubgraph{\outPath_j}{\CondSet{Q_h}{h\in J_Q}}$ for every $j \in J_P.$
	We say that $\InducedSubgraph{\widetilde{W}}{\tiling,i}$ is the \defineSubindex{$i$-th $\tiling$-slice}{slice} of $\widetilde{W}.$
\end{definition}

Note that the width of $\Fkt{\widetilde{W}}{\tiling}$ is at most by $2d$ greater than the width of $\widetilde{W}.$
In $\InducedSubgraph{\widetilde{W}}{\tiling,i},$ we essentially cut away the tiles of $\Class_i.$
This operation gives us a slice $W'$ of some cylindrical wall for which the perimeter of every tile in $\Class_i$ is the perimeter of some brick.
Next, we find a tiling of $W'$ such that every tile of $\Class_i$ that belongs to $\widetilde{W}$ is captured by the centre of some tile in the new tiling.

\begin{definition}[Tier II Tiling]\label{def:tilingII}
	Let $t,k,k'\in\N$ be positive integers with $k\geq k'$ and $f\colon\N\rightarrow\N$ be some function.
	Let $W$ be a cylindrical wall of order $3k$ with its \hyperref[def:triadic_partition]{triadic partition} $\Triadic = \Brace{W,k,W_1,W_2,W_3,W^1,W^2,W^3},$ and \hyperref[def:tiling]{$\tiling=\tiling_{W,k,f,\xi,\xi'}$} for some $\xi,\xi' \in [\Fkt{f}{t}+1],$ as well as $\Set{\Class_1,\dots,\Class_4}$ be a four-colouring of $\tiling$ and $i\in [4]$ be a fixed colour.
	Moreover, let $\widetilde{W}$ be a \hyperref[def:slice]{slice} of $W_2$ of width $k'$ such that no \hyperref[def:tile]{tile} of $\Class_i$ contains a vertex of $\per{\widetilde{W}}$ and $\widetilde{\tiling}$ be the collection of all tiles from $\tiling$ that contain a vertex of $\widetilde{W}.$
	
	The \define{tier II tiling} of width $\Fkt{f}{t}$ for the slice $\InducedSubgraph{\widetilde{W}}{\tiling,i},$ denoted by \MarginOnly{\resizebox{\marginparwidth}{!}{$\TierIITiling{\tiling,i,\Fkt{f}{t}}{}$}} $\TierIITiling{\tiling,i,f}{\widetilde{W}},$ is defined as the unique tiling of $\InducedSubgraph{\widetilde{W}}{\tiling,i}$ such that every $\tile \in\Class_i\cap\widetilde{\tiling}$ is in the interior of the centre of some tile of $\TierIITiling{\tiling,i,\Fkt{f}{t}}{\widetilde{W}}.$
\end{definition}

Since every tile in $\tiling$ consists of $2\Fkt{f}{t}+2$ path pairs, the tiling $\TierIITiling{\tiling,i,\Fkt{f}{t}}{\widetilde{W}}$ is well defined and does in fact cover all of $\InducedSubgraph{\widetilde{W}}{\tiling,i}.$

Now, we fix some terminology for paths in the graph that cause non-planar behaviour with respect to the wall.
Let $k,w\in\N$ be positive integers, $W$ be a wall of order $k$ and $W'$ be a \hyperref[def:slice]{slice} of $W.$
We call a $C_1$-$C_k$- or $C_k$-$C_1$-path $P$ a \emph{jump over all of $W$} if it is internally disjoint from $W.$
A \define{perimeter jump} over a wall $W$ is a path that starts or ends at a tile of $W$-distance at least $2$ to $\per{W}$ and is otherwise disjoint of $W.$
A $\V{W'}$-path $P$ is called a \define{jump over $W'$} if $\E{P} \cap \E{W'}= \emptyset.$
We say that a directed $\V{W'}$-path $P$ is a \defineSorted{$w$-long jump over $W'$}{w-long jump over W} if for all $\xi,\xi'\in [w+1]$ the start-vertex and the end-vertex of $P$ belong to distinct tiles $\tile_s$ and $\tile_t$ of the tiling $\tiling_{W,k,w,\xi,\xi'}.$

\subsection{Selected results from \texorpdfstring{\cite{giannopoulou2020directed}}{Giannopoulou et al.}}
\label{sec:two_lemmata}

In this \namecref{sec:two_lemmata} we state some results from \cite{giannopoulou2020directed} which provide the corner stones for our proof of \cref{thm:directed_flat_wall}.

\begin{theorem}[Giannopoulou et al.~\cite{giannopoulou2020directed}]
	\label{thm:X_paths}
	Let $D$ be a digraph and $X\subseteq\V{D}.$
	For all positive $k \in\N,$ there are $k$ pairwise vertex disjoint $X$-paths in $D,$ or there exists a set $S \subseteq \V{D}$ of size at most $2k$ such that every $X$-path in $D$ contains a vertex of $S.$
	
	Furthermore, there is a polynomial time algorithm which, given a digraph $D$ and a set $X \subseteq \V{D},$ outputs $k$ pairwise disjoint $X$-paths, or a set $S\subseteq\V{D}$ of size at most $2k$ as above.
\end{theorem}

\begin{theorem}[Giannopoulou et al.~\cite{giannopoulou2020directed}]
	\label{thm:half_integral}
	Let $k\in\N$ be a positive integer, $D$ be a digraph, and $X,Y\subseteq\V{D}.$
	If $\mathcal{P}$ is a half-integral $X$-$Y$-linkage of order $2k$ in $D,$ then there exists an $X$-$Y$-linkage $\mathcal{J}$ of order $k$ such that $\V{\mathcal{J}}\subseteq\V{\mathcal{P}}.$
\end{theorem}

The following is a combination of Lemmas 4.3 to 4.8 from \cite{giannopoulou2020directed} and a proof can be found in the proof of Lemma 4.9 in \cite{giannopoulou2020directed}.
The only difference between Lemma 4.9 from \cite{giannopoulou2020directed} and the statement below is that we extract the last subcase of Case 1 in its proof as a potential outcome.
In fact the statement contains two lemmata that are merely two analogue cases of the same situation.
We mark the changes one needs to perform to obtain the second statement in \textcolor{myLightblue!40!blue}{blue} and in parentheses.

\begin{lemma}[Giannopoulou et al.~\cite{giannopoulou2020directed}]
	\label{lem:directed_long_jumps}
	There exist functions $f_w\colon\N\rightarrow\N,$ $f_P\colon\N\rightarrow\N,$ and $f_W\colon\N\rightarrow\N$ such that for every $t\in\N$ the following holds.
	Let 
	\begin{itemize}
		\item $D$ be a digraph,
		\item $W$ be a cylindrical wall of order $3k$ with $k \geq \Fkt{f_W}{t}$ in $D,$
		\item $\Triadic = \Brace{W,k,W_1,W_2,W_3,W^1,W^2,W^3}$ be the \hyperref[def:triadic_partition]{triadic partition} of $W,$ and
		\item \hyperref[def:tiling]{$\tiling=\tiling_{W,k,\Fkt{f_w}{t},\xi,\xi'}$} for some $\xi,\xi'\in[\Fkt{f_w}{t}+1].$
	\end{itemize}
	If there exists a subfamily $\tiling'$ of $\tiling$ and a family $\Jumps$ of pairwise disjoint directed paths in $D$ with the following properties:
	\begin{enumerate}
		\item Every member of $\Jumps$ is a $\V{W}$-path, 
		\item $\Abs{\tiling'}=\Abs{\Jumps}=\Fkt{f_P}{t},$
		\item for every $\tile_{\columnFunction{p}{},\rowFunction{q}{},\Fkt{f_w}{t}} \neq \tile_{\columnFunction{p'}{},\rowFunction{q'}{},\Fkt{f_w}{t}} \in \tiling'$ we have $\max\Set{\Abs{p-p'},\Abs{q-q'}}\geq 2,$
		\item there exists a bijection $\JumpsLeaving{} \colon \tiling' \rightarrow \Jumps$ such that the start-vertex of the path $\JumpsLeaving{\tile}$ belongs to the \hyperref[def:tile]{centre} of $\tile$ for every $\tile \in \tiling',$ \textcolor{myLightblue!40!blue}{(there exists a bijection $\JumpsArriving{} \colon \tiling' \rightarrow \Jumps$ such that the end-vertex of the path $\JumpsArriving{\tile}$ belongs to the centre of $\tile$ for every $\tile \in \tiling',$)}
		\item for all $\tile \in \tiling',$ where $\V{\tiling'}=\bigcup_{\tile'\in\tiling'}\V{\tile'}$ the intersection $\V{\JumpsLeaving{\tile}} \cap \V{\tiling'}$ contains exactly the end-vertex of $\JumpsLeaving{\tile}$ \textcolor{myLightblue!40!blue}{($\V{\JumpsArriving{\tile}} \cap \V{\tiling'}$ contains exactly the start-vertex of $\JumpsArriving{\tile}$)}, and finally
		\item the end-vertices \textcolor{myLightblue!40!blue}{(start-vertices)} of the paths in $\Jumps$ are of mutual \hyperref[def:Wdistance]{$W$-distance} at least four. 
	\end{enumerate}
	Then, at least one of the following is true.
	\renewcommand{\labelenumi}{\textbf{\theenumi}}
	\renewcommand{\theenumi}{(\roman{enumi})}
	\begin{enumerate}
		\item \label{directed_long_jumps:i} $D$ has a $\K{t}$-butterfly minor grasped by $W,$
		\item \label{directed_long_jumps:ii} there exists a family of tiles $\tiling''\subseteq\tiling'$ all contained in a single \hyperref[def:strip]{strip} $S\subseteq W$ of height equal to the height of the tiles in $\tiling$ such that
		\begin{itemize}
			\item we can number $\tiling''=\Set{\tile_1,\dots,\tile_h}$ such that $S-\tile_i$ has one component containing exactly the tiles $\tile_1,\dots,\tile_{i-1}$ for each $i\in [h],$
			\item $\Abs{\tiling''} \geq \Fkt{f_P}{t}^{\frac{1}{4}},$
			\item for every $i\in[h-1]$ the tiles $\tile_i$ and $\tile_{i+1}$ are separated in $W$ by a \hyperref[def:slice]{slice} of width equal to the width of the tiles in $\tiling,$ and
			\item there is a family $\Jumps'\subseteq\Jumps$ with $\Abs{\Jumps'}=\Abs{\tiling''}$ such that for each $\tile \in\tiling''$ we have $\JumpsLeaving{\tile}\in\Jumps'$ \textcolor{myLightblue!40!blue}{($\JumpsArriving{\tile}\in\Jumps'$)}, and
			\item for each $i\in[h]$ the end-vertex of $\JumpsLeaving{\tile_i}$ \textcolor{myLightblue!40!blue}{(start-vertex of $\JumpsArriving{\tile_i}$)} lies in the component of $S-\tile_i$ that contains no tiles of $\tiling''$ if $i=1,$ or in the slice of $S$ separating $\tile_{i-1}$ and $\tile_i$ otherwise.
		\end{itemize}
		\item \label{directed_long_jumps:iii} there exists a family of tiles $\tiling''\subseteq\tiling'$ all contained in a single \hyperref[def:strip]{strip} $S\subseteq W$ of height equal to the height of the tiles in $\tiling$ such that $\tiling''$ and $S$ meet the properties of outcome \cref{directed_long_jumps:ii} after switching the parametrisation of $W.$
	\end{enumerate}
\end{lemma}

We fix the functions $f_w,$ $f_P,$ and $f_W$ from \cref{lem:directed_long_jumps} and, by using the bounds provided by the original proofs~\cite{giannopoulou2020directed}, we obtain the following rough estimates:
{
\renewcommand{\labelenumi}{\textbf{\theenumi}}
\renewcommand{\theenumi}{(\roman{enumi})}
\begin{enumerate}
	\item \label{def:f_w} $\Fkt{f_w}{t}=2^9t^{10},$
	\item \label{def:f_P} $\Fkt{f_P}{t}=2^7t^8,$ and
	\item \label{def:f_W} $\Fkt{f_W}{t}=2^{32+t^{30}}.$
\end{enumerate}}

\Cref{lem:directed_long_jumps} is already powerful enough to guarantee a large cross-row grid as a butterfly minor in case we find many long jumps over a sufficiently large wall.
This is due to the following proof sketch, which we formalise over the remainder of the \namecref{sec:flat_wall}.
In case \cref{lem:directed_long_jumps} yields the existence of a $\K{\cliqueToCrossRowGridFunction{t}}$ butterfly minor, this minor contains a cross-row grid of order $t$ as a butterfly minor, by \cref{lem:clique_contains_cross-row-wall}.
Otherwise, \cref{lem:directed_long_jumps} produces many pairwise disjoint slices of the wall $W,$ mutually far apart from each other, and each of them having a neighbouring slice together with a fairly long jump into this neighbouring slice.
Such a slice and its neighbour together yield a slice of twice the width that contains a cross.
By \cref{lem:many_cross_slices_yield_minor}, this yields a cross-row grid as a butterfly minor.

\subsection{Removing long jumps}

In order to provide some formal basis to the proof sketch from before, we prove two auxiliary results, each providing ways to find cross-row grid minors from long jumps over our wall.


\begin{definition}[Auxiliary Digraph Type \rom{1}]
	\label{def:auxiliary_digraph_I}
	Let $t,k,k',w\in\N$ be positive integers such that $k \geq  k' \geq 2\Fkt{f_W}{t}+4\Fkt{f_P}{t}\cdot\Brace{2w+1},$ $w\geq 2\Fkt{f_w}{t},$ and $\xi,\xi'\in [w+1].$
	Let $D$ be a digraph containing a cylindrical wall $W$ of order $3k$ with its \hyperref[def:triadic_partition]{triadic partition} $\Triadic=\Brace{W,k,W_1,W_2,W_3,W^1,W^2,W^3},$ and a tiling \hyperref[def:tiling]{$\tiling=\tiling_{W,k,w,\xi,\xi'}$}.
	Let $\Set{\Class_1,\dots,\Class_4}$ be a four-colouring of $\tiling,$ $i\in[4]$ and $W'\subseteq W$ be a \hyperref[def:slice]{slice} of width $k'$ of $W_2.$
	At last, let us denote by $\tiling'$ the family of tiles from $\tiling$ that share a vertex with $W'$ and let $\Class_i' \coloneqq \tiling' \cap \Class_i.$
	Then, \MarginOnly{$\Fkt{D^{\rom{1}}_i}{W'}$}$\Fkt{D^{\rom{1}}_i}{W'}$ is the digraph obtained from $D$ by performing the following construction steps for every $\tile\in\Class_i'$:
	\begin{enumerate}
		\item add new vertices $x_{\tile}^{\text{in}}$ and $x_{\tile}^{\text{out}},$
		\item for every vertex $u$ in the centre of $\tile$ introduce the edges $\Brace{u,x_{\tile}^{\text{in}}}$ and $\Brace{x_{\tile}^{\text{out}},u},$ and then
		\item delete all internal vertices of $\tile.$
	\end{enumerate}
	Additionally, we define the sets 
	\begin{align*}
		X_{\rom{1}}^{\text{out}}&\coloneqq \CondSet{x_{\tile}^{\text{out}}}{\tile\in\Class_i'}\text{, and}\\
		X_{\rom{1}}^{\text{in}}&\coloneqq \CondSet{x_{\tile}^{\text{in}}}{\tile\in\Class_i'}. \qedhere
	\end{align*}
\end{definition}

We establish that every outcome of \cref{lem:directed_long_jumps} yields a sufficiently large cross-row grid.

\begin{lemma}
	\label{claim:long_jumps_imply_minor}
	Every outcome of \cref{lem:directed_long_jumps} yields $\CrossRowGrid{\cliqueToCrossRowGridFunctionInverse{t}}$ as a butterfly minor.
\end{lemma}
\begin{proof}
	If the outcome is \cref{directed_long_jumps:i}, then, by \cref{lem:clique_contains_cross-row-wall}, we can find a butterfly model of $\CrossRowGrid{\cliqueToCrossRowGridFunctionInverse{t}}$ within the $\K{t}$ butterfly minor model.
	
	Hence, we may assume the outcome is \cref{directed_long_jumps:ii} or \cref{directed_long_jumps:iii}.
	Since these two cases are symmetric it suffices to only consider outcome \cref{directed_long_jumps:ii}.
	In that case, we find a family $\tiling'' \subseteq \tiling'$ of size $h \geq \Fkt{f_P}{t}^{\frac{1}{4}} \geq \frac{1}{3}t$ contained in a single \hyperref[def:strip]{strip} $S \subseteq W$ which is of the same height as the tiles in $\tiling''.$
	We can number the tiles $\tiling'' = \Set{\tile_1,\dots,\tile_h}$ such that $S-\tile_i$ has one component containing exactly the tiles $\tile_1,\dots,\tile_{i-1}$ for each $i\in [h].$
	For each $i \in [h]$ let $S_i\subseteq W$ be the slice of width equal to the width of the tiles in $\tiling''$ whose intersection with $S$ is exactly the tile $\tile_i.$
	Additionally, for each $i \in [h]$ there is a \hyperref[def:slice]{slice} $H_i$ of $W$ containing $S_i$ and both the start- and the end-vertex of the jump $J_i \coloneqq \JumpsLeaving{\tile_i}$ such that $H_i$ and $H_j$ are disjoint if $i\neq j.$
	By \cref{lem:local_jump_yields_cross_slice}, $J_i$ yields a cross slice $H'_i$ in $H_i$ such that the cross lies within the strip $S.$
	As there are more than $t$ rows that do not lie in $S$ and we have more than $\cliqueToCrossRowGridFunctionInverse{t}$ such slices, by \cref{lem:many_cross_slices_yield_minor} we obtain a $\CrossRowGrid{\cliqueToCrossRowGridFunctionInverse{t}}$ as a butterfly minor.
\end{proof}

Let $L$ and $P$ be directed paths.
We say that $P$ is a \define{long jump of $L$} if $P$ is a $w$-long jump over $W$ and $P\subseteq L.$
Additionally, we say that $P$ is a \define{jump of $L$} if $P$ is a directed $W$-path.

The following \namecref{lem:longjumps1} deals with long jumps between one colour class and the remaining three colour classes of a tiling colouring.

\begin{lemma}\label{lem:longjumps1}
	Let $t,k,k',w\in\N$ be positive integers with $k \geq k' \geq 2 \Fkt{f_W}{\cliqueToCrossRowGridFunction{t}} + 2^{16} \Fkt{f_P}{\cliqueToCrossRowGridFunction{t}} + 2 + 2w,$ $w\geq 2\Fkt{f_w}{\cliqueToCrossRowGridFunction{t}}+2^7\Fkt{f_P}{\cliqueToCrossRowGridFunction{t}},$ and $\xi,\xi'\in [w+1].$
	Let $D$ be a digraph containing a cylindrical wall $W$ of order $3k$ with its \hyperref[def:triadic_partition]{triadic partition} $\Triadic = \Brace{W,k,W_1,W_2,W_3,W^1,W^2,W^3},$ and a tiling \hyperref[def:tiling]{$\tiling=\tiling_{W,k,w,\xi,\xi'}$}.
	Let $\Set{\Class_1,\dots,\Class_4}$ be a four-colouring of $\tiling,$ $i\in [4]$ and $W'\subseteq W$ be a \hyperref[def:slice]{slice} of width $k'$ of $W_2.$
	
	Now let $\tiling'$ be the family of all tiles of $\tiling$ that are completely contained in $W'$ and let $\widetilde{W}$ be the smallest slice of $W$ that contains all tiles from $\tiling'.$
	
	Consider the \hyperref[def:auxiliary_digraph_I]{auxiliary digraph} $\Fkt{D_i^{\rom{1}}}{\widetilde{W}}$ with $X_{\rom{1}}^{\text{out}}$ and $X_{\rom{1}}^{\text{in}}.$
	Additionally, we construct the set $Y_{\rom{1}}$ as follows.
	Let $Q$ and $Q'$ be the two cycles of $\per{W_2}.$
	For every $j \in \left[\frac{3k}{4} \right],$ $Y_{\rom{1}}$ contains exactly one arbitrarily chosen vertex of $Q\cap \outPath_{4j},$ $Q\cap \inPath_{4j+2},$ $Q'\cap \outPath_{4j},$ and $Q'\cap \inPath_{4j+2}$ each.

	If there exists a family $\mathcal{L}$ of pairwise disjoint directed paths with $\Abs{\mathcal{L}}= 2^7\Fkt{f_P}{\cliqueToCrossRowGridFunction{t}}$ such that either
	\begin{itemize}
		\item $\mathcal{L}$ is a family of directed $X_{\rom{1}}^{\text{out}}$-$Y_{\rom{1}}$-paths, or
		\item $\mathcal{L}$ is a family of directed $Y_{\rom{1}}$-$X_{\rom{1}}^{\text{in}}$-paths,
	\end{itemize}
	then $D$ contains $\CrossRowGrid{t}$ as a butterfly minor.
\end{lemma}

\begin{proof}
	In the statement we insert $\cliqueToCrossRowGridFunction{t}$ as the arguments for all functions.
	This means, due to \cref{claim:long_jumps_imply_minor}, any application of \cref{lem:directed_long_jumps} yields a $\CrossRowGrid{t}$ as a butterfly minor.
	Thus, we construct a cylindrical wall $W'''\subseteq W$ of sufficient size, together with a family of $\Fkt{f_P}{\cliqueToCrossRowGridFunction{t}}$ directed $W'''$-paths that meet the requirements of \cref{lem:directed_long_jumps}.
	
	Without loss of generality, let us assume $\mathcal{L}$ is a family of directed $X_{\rom{1}}^{\text{out}}$-$Y_{\rom{1}}$-paths.
	The other case, $Y_{\rom{1}}$-$X_{\rom{1}}^{\text{out}}$-paths, follows by similar arguments.
	
	\textbf{Short overview of the proof:}
	Towards our goal, we first show that we can use $\mathcal{L}$ to construct a half-integral $X_{\rom{1}}^{\text{out}}$-$Y_{\rom{1}}$-linkage $\mathcal{L}_1$ such that
	\begin{enumerate}
		\item $\Abs{\mathcal{L}_1}=2^7\Fkt{f_P}{\cliqueToCrossRowGridFunction{t}},$
		\item there exists a family $\mathcal{F}\subseteq\tiling'$ with $\Abs{\mathcal{F}}\leq 2^7\Fkt{f_P}{\cliqueToCrossRowGridFunction{t}},$ and
		\item for every $L\in\mathcal{L}_1,$ every end-vertex $u$ of a jump of $L$ with $u\in\V{\widetilde{W}}$ belongs to a tile from $\Class_i'\cup\mathcal{F}.$
	\end{enumerate}
	Second, we use \cref{thm:half_integral} to obtain a family $\mathcal{L}_2$ of pairwise disjoint directed $X_{\rom{1}}^{\text{out}}$-$Y_{\rom{1}}$-paths of size $2^6\Fkt{f_P}{\cliqueToCrossRowGridFunction{t}}$ from $\mathcal{L}_1.$
	Third, we remove the cycles and paths of $\widetilde{W}$ that meet tiles from $\mathcal{F}$ and obtain a new slice $W'''$ of some cylindrical wall.
	For this slice, we construct a \hyperref[def:tiling]{tiling} and a \hyperref[def:tilingII]{tier II tiling} as well as a half-integral linkage $\mathcal{L}_3$ of size $2^4\Fkt{f_P}{\cliqueToCrossRowGridFunction{t}}$ from $\mathcal{L}_2.$
	The linkage $\mathcal{L}_3$ connects the \hyperref[def:tile]{centres} of some tiles in the tier II tiling to vertices of $\widetilde{W}'$ such that their end-vertices and start-vertices are mutually far enough apart and every path in $\mathcal{L}_3$ is internally disjoint from the new wall $W'''.$
	Another application of \cref{thm:half_integral} then yields the family $\mathcal{L}_4$ of long jumps necessary for an application of \cref{lem:directed_long_jumps}.
	
	\textbf{The construction of $\mathcal{L}_1$ and $\mathcal{F}.$}
	We do construct $\mathcal{L}_1$ and $\mathcal{F}$ iteratively starting with $\mathcal{L}'\coloneqq \mathcal{L},$ $\mathcal{L}_1\coloneqq \emptyset,$ and $\mathcal{F}\coloneqq\emptyset.$
	As long as $\mathcal{L}'$ is non-empty, perform the following actions.
	
	Select some path $L\in\mathcal{L}'.$
	In case $L$ is internally disjoint from $\widetilde{W},$ add $L$ to $\mathcal{L}_1$ and remove it from $\mathcal{L}'.$
	Otherwise, let $v_L$ be the first vertex of $L$ that belongs to $\widetilde{W},$ but not to a tile from $\Class_i'.$
	\renewcommand{\labelenumi}{\textbf{\theenumi}}
	\renewcommand{\theenumi}{(\roman{enumi})}
	\begin{enumerate}
		\item \label{L_1:i} If $v_L$ does not belong to any tile from $\mathcal{F},$ let $\tile \in\tiling\setminus\Class_i$ be the tile that contains $v_L$ and add $\tile$ to $\mathcal{F}.$
		Let $R$ be a shortest directed path from $v_L$ to $Y_{\rom{1}}$ in $W$ such that $R$ avoids all vertices of $W$ that are contained in two different paths of $\mathcal{L}_1$ and that is internally disjoint from $Lv_L.$
		Now add $Lv_LR$ to $\mathcal{L}_1$ and remove $L$ from $\mathcal{L}'.$
		Note that such a path $R$ must exist because the paths in $\mathcal{L}$ are pairwise disjoint, $\tile$ was not used for such a re-routing before, and $w$ and $k'$ are chosen sufficiently large in proportion to $2^7\Fkt{f_P}{\cliqueToCrossRowGridFunction{t}}.$
		Also note that the path $R$ is exactly the subpath that might cause $\mathcal{L}_1$ to be half-integral.
		However, because the paths in $\mathcal{L}'$ are pairwise disjoint, we can be sure that $R$ never meets a vertex contained in two distinct paths of $\mathcal{L}_1.$
		
		\item \label{L_1:ii} If $v_L$ belongs to a tile from $\mathcal{F},$ follow along $v_LL$ until we encounter a vertex $u_L$ for which one of the following is true:
			\begin{enumerate}
				\item \label{u_L:a} $u_L$ belongs to a tile $\tile$ from $\tiling\setminus\Brace{\Class_i\cup\mathcal{F}},$ or
				\item \label{u_L:b} every internal vertex of $u_LL$ belongs to $W-\widetilde{W}$ or to some tile from $\Class_i\cup\mathcal{F}.$ 
			\end{enumerate}
		If \cref{u_L:a}, add $\tile$ to $\mathcal{F}$ and then repeat the instructions from \cref{L_1:i} but replace $v_L$ with $u_L.$
		Otherwise, \cref{u_L:b} holds and we remove $L$ from $\mathcal{L}'$ and add it to $\mathcal{L}_1.$
	\end{enumerate}
	During this construction, for every $L\in\mathcal{L},$ we added at most one tile to $\mathcal{F}$ and thus $\Abs{\mathcal{F}}\leq\Abs{\mathcal{L}}.$
	Note that, by construction, $\mathcal{L}_1$ is indeed a half-integral linkage from $X_{\rom{1}}^{\text{out}}$ to $Y_{\rom{1}}.$
	Moreover, we may assume that every $L$ meets each tile in $\mathcal{F}$ in at most $2^7\Fkt{f_P}{\cliqueToCrossRowGridFunction{t}}+1$ horizontal path pairs and vertical cycles, because otherwise we can replace it by a shorter path through $W$ itself.
	
	\textbf{Obtaining $\mathcal{L}_2.$}
	We apply \cref{thm:half_integral} to obtain a family $\mathcal{L}_2$ of pairwise disjoint directed $X_{\rom{1}}^{\text{out}}$-$Y_{\rom{1}}$-paths with $\V{\mathcal{L}_2}\subseteq\V{\mathcal{L}_1}$ and $\Abs{\mathcal{L}_2}=2^6\Fkt{f_P}{\cliqueToCrossRowGridFunction{t}}.$
	
	\textbf{Constructing $\mathcal{L}_3.$}
	Let us consider $W''\coloneqq \InducedSubgraph{\widetilde{W}}{\tiling,i}$ together with the tiling $\tiling''\coloneqq\TierIITiling{\tiling,i,w}{\widetilde{W}}$ and a four-colouring $\Set{\widetilde{\Class}_1,\dots,\widetilde{\Class}_4}.$
	Note that the width of $\widetilde{W}$ is at most $2w$ smaller than the width of $W'.$
	Thus, by choice of $k',$ we obtain that $W''$ is a slice of width $k''\geq \Fkt{f_W}{\cliqueToCrossRowGridFunction{t}}+2^7\Fkt{f_P}{\cliqueToCrossRowGridFunction{t}}\Brace{2w+1}+1$ of some cylindrical wall of order $3k''$ that is completely contained in $W.$
	For each $L\in\mathcal{L}_2$ let $\tile^1_L\in \Class_i$ such that $\Start{L}$ belongs to $\tile^1_L.$
	Let $K^1_L\in \tiling''$ be the tile whose centre is the perimeter of $\tile^1_L.$
	Choose any vertex $\Start{L}'$ of degree three in $W''$ that is not contained in any path of $\mathcal{L}_2,$ and let $R_L$ be a directed path from $\Start{L}'$ to $\Start{L}$ within $\tile^1_L.$
	Let $\mathcal{L}_3$ be the, possibly again half-integral, family of directed paths resulting from concatenating the new paths and the corresponding original path from $\mathcal{L}_2.$
	
	\textbf{Finding $\mathcal{L}_4$ and $W'''.$}
	There exists $j\in[4]$ such that at least $2^4\Fkt{f_P}{\cliqueToCrossRowGridFunction{t}}$ of the paths from $\mathcal{L}_3$ start at the centre of a tile from $\widetilde{\Class}_j.$
	Let $\mathcal{L}_3' \subseteq\mathcal{L}_3$ be a family of exactly $2^4\Fkt{f_P}{\cliqueToCrossRowGridFunction{t}}$ such paths.
	Next, let us consider the family $\mathcal{F}.$
	Let $W'''$ be the subgraph of $W''$ induced by all vertical cycles and horizontal path pairs in $W''$ that do not contain a vertex of some tile in $\mathcal{F}$ that belongs to a path in $\mathcal{L}_3'.$
	Since $\Abs{\mathcal{F}}\leq 2^7\Fkt{f_P}{\cliqueToCrossRowGridFunction{t}}$ and each tile in $\mathcal{F}$ meets a path in $\mathcal{L}_3'$ in at most $2^7\Fkt{f_P}{\cliqueToCrossRowGridFunction{t}}+1$ such cycles and pairs of horizontal paths, it follows that $W'''$ is a slice of width $k'''\geq\Fkt{f_W}{\cliqueToCrossRowGridFunction{t}}+2$ of some cylindrical wall $W^{*}\subseteq W$ of order $3k'''.$
	Moreover, $W^{*}$ can be partitioned into three \hyperref[def:slice]{slices} of width $k'''$ as in its \hyperref[def:triadic_partition]{triadic partition}, such that $W'''$ is the slice in the middle.
	Let us rename the paths and cycles of $W^{*}$ such that $Q^{*}_1,\dots,Q^{*}_{3k'''}$ are the vertical paths of $W^{*},$ $\inPathStar_1,\dots,\inPathStar_{3k'''}$ its in-paths and $\outPathStar_1,\dots,\outPathStar_{3k'''}$ its out-paths.
	We construct the set $Y^{*}$ as follows:
	For every $j\in\left[\frac{3k'''}{4} \right],$ $Y^{*}$ contains exactly one vertex of $Q_1^{*} \cap P^{*1}_{4j},$ $Q_1^{*} \cap P^{*2}_{4j+2},$ $Q_{3k'''}^{*} \cap P^{*1}_{4j},$ and $Q_{3k'''}^{*} \cap P^{*2}_{4j+2}$ each.
	
	Similarly to $\mathcal{L}_1,$ we now construct $\mathcal{L}''_3$ iteratively from $\mathcal{L}'_3.$
	We start with $\mathcal{L}''_3$ being empty.
	Let $L\in\mathcal{L}_3'$ be any path and $t_L$ be the first vertex after $\Start{L}$ that $L$ shares with either $W'''$ or $W^{*}-W'''.$
	If $t_L\in\V{W'''},$ add $Lt_L$ to $\mathcal{L}_3''.$
	Otherwise, let $b_L$ be the last vertex of $L$ in $W^{*}-W'''.$
	Then, we can find a path $R_L$ in $W$ from $b_L$ to a vertex $t^{*}_L$ of $Y^{*}$ such that $t^{*}_L$ is of $W^{*}$-distance at least four to every endpoint of every path already in $\mathcal{L}_3'',$ $R_L$ is internally disjoint from $L,$ and $R_L$ does not contain a vertex that is contained in two distinct paths from $\mathcal{L}_3''.$
	Add $LR_L$ to $\mathcal{L}_3''.$
	Finally, $\mathcal{L}_3''$ is a half-integral linkage from the set $S^{*} \coloneqq \Start{\mathcal{L}_3'}$ to $Y^{*}$ of size $2^4\Fkt{f_P}{\cliqueToCrossRowGridFunction{t}},$ and thus by \cref{thm:half_integral} we can find a family $\mathcal{L}_4$ of pairwise disjoint directed paths from $S^{*}$ to $Y^{*}$ with $\V{\mathcal{L}_4}\subseteq\V{\mathcal{L}_3''}$ that is of size $2^3\Fkt{f_P}{\cliqueToCrossRowGridFunction{t}}.$
	It follows that all paths in $\mathcal{L}_4$ are internally disjoint from $W'''.$
	
	Let us consider the tiles of $\widetilde{\Class}_i$ whose centres contain a vertex of $S^{*}.$
	Since $W'''$ might be a proper subgraph of $W'',$ $\tiling''$ is not necessarily a tiling of $W'''.$
	Each tile $\tile \in \tiling'',$ however, contains a tile $\tile'$ of width $\Fkt{f_w}{\cliqueToCrossRowGridFunction{t}}$ with the same \hyperref[def:tile]{centre}.
	Since $\tile$ is surrounded by at most $8$ tiles from $\mathcal{F}$ in $W',$ we may find, among the $2^3\Fkt{f_P}{\cliqueToCrossRowGridFunction{t}}$ many such tiles, a family $\Jumps$ of $\Fkt{f_P}{\cliqueToCrossRowGridFunction{t}}$ tiles that are pairwise disjoint.
	Thus, because they all are constructed from the family $\widetilde{\Class}_i,$ they meet the distance requirements of the tiles in \cref{lem:directed_long_jumps}.
	Hence, we can apply \cref{lem:directed_long_jumps} and by \cref{claim:long_jumps_imply_minor} we obtain the desired butterfly minor $\CrossRowGrid{t}.$
\end{proof}

The above lemma allows us to argue that enough long jumps that all start, respectively end, in the centre of tiles from a single colour class but end, respectively start, in tiles from the remaining three classes yield a cross-row grid as a butterfly minor.
By utilising a second auxiliary digraph we can make use of \cref{thm:X_paths} to also prove this for long jumps between tiles of the same colour.

\begin{definition}[Auxiliary Digraph Type \rom{2}]
	\label{def:auxiliarydigraphII}
	Let $t,k,k',w\in\N$ be positive integers such that $k\geq k'\geq2\Fkt{f_W}{t},$ $w\geq 2\Fkt{f_w}{t},$ and $\xi,\xi'\in\left[w+1\right].$
	Let $D$ be a digraph containing a cylindrical wall $W$ of order $3k$ with its \hyperref[def:triadic_partition]{triadic partition} $\Triadic=\Brace{W,k,W_1,W_2,W_3,W^1,W^2,W^3},$ and a \hyperref[def:tiling]{tiling} $\tiling=\tiling_{W,k,w,\xi,\xi'}.$
	Let $\Set{\Class_1,\dots,\Class_4}$ be a four-colouring of $\tiling,$ $i\in[4]$ and $W'\subseteq W$ be a \hyperref[def:slice]{slice} of width $k'$ of $W_2$ such that no \hyperref[def:tile]{tile} of $\Class_i$ contains a vertex of the perimeter of $W'.$
	Then, \MarginOnly{$\Fkt{D^{\rom{2}}_i}{W'}$}$\Fkt{D^{\rom{2}}_i}{W'}$ is the digraph obtained from $D$ by performing the following construction steps.
	
	For every $\tile \in \Class_i,$ such that $\tile$ contains a vertex of $W',$ we do the following:
	\begin{enumerate}
		\item add a new vertex $x_{\tile},$ and
		\item for every vertex $v$ that belongs to the interior or the centre of $\tile,$ introduce the edges $\Brace{x_{\tile},v}$ and $\Brace{v,x_{\tile}}.$
	\end{enumerate}
	Once this is done, delete all vertices of $W'$ that do not belong to tiles of $\Class_i$ that are contained in $W'.$
	Let $X_{\rom{2}}^i$ be the collection of all newly introduced vertices $x_{\tile}.$
\end{definition}

\begin{lemma}
	\label{lem:longjumpsphase2}
	Let $t,k,k',w\in\N$ be positive integers, and $\xi,$ $\xi'\in\left[w+1\right]$ where $w \geq 2 \Fkt{f_w}{\cliqueToCrossRowGridFunction{t}}.$
	Let $D$ be a digraph containing a cylindrical wall $W$ of order $3k$ with vertical paths $Q_1,$ $\dots,$ $Q_{3k},$ in-paths $\inPath_1,$ $\dots,$ $\inPath_{3k}$ and out-paths $\outPath_1,$ $\dots,$ $\outPath_{3k},$ where $k \geq k'\geq 4 \Fkt{f_W}{\cliqueToCrossRowGridFunction{t}}^2.$
	Also, let $\Triadic=(W_0,$ $k,$ $W_1,$ $W_2,$ $W_3,$ $W^1,$ $W^2,$ $W^3)$ be its \hyperref[def:triadic_partition]{triadic partition}, $W^{*}\subseteq W_2$ be a \hyperref[def:slice]{slice} of width $k',$ $\tiling=\tiling_{W,k,w,\xi,\xi'}$ be a \hyperref[def:tiling]{tiling}, $\Set{\Class_1,\dots,\Class_4}$ a four-colouring and $i\in[4]$ a fixed colour.
	
	Then, $D$ contains $\CrossRowGrid{t}$ as a butterfly minor, or there exists a set $\MarkedTiles{2}_{i,\xi,\xi'}\subseteq\tiling$ with $\Abs{\MarkedTiles{2}_{i,\xi,\xi'}}\leq 8\Fkt{f_P}{\cliqueToCrossRowGridFunction{t}}$ and a set $Z^2_{i,\xi,\xi'}\subseteq\V{D-W}$ with $\Abs{Z^2_{i,\xi,\xi'}}\leq 8\Fkt{f_P}{\cliqueToCrossRowGridFunction{t}}$ such that every directed $\V{W_0}$-path in $D-Z^2_{i,\xi,\xi'}$ whose start- and end-vertex belong to different tiles of $\Class_i$ contains a vertex of some tile in $\MarkedTiles{2}_{i,\xi,\xi'}.$
\end{lemma}

\begin{proof}
	Let $W'$ be the largest \hyperref[def:slice]{slice} of $W^{*}$ such that no tile of $\Class_i$ contains a vertex of $\per{W'}.$
	Let us consider the \hyperref[def:auxiliary_digraph_I]{auxiliary digraph} $\Fkt{D^{\rom{2}}_i}{W'}$ with the set $X_{\rom{2}}^i$ of newly added vertices.
	By applying \cref{thm:X_paths} to the set $X_{\rom{2}}^i$ in $\Fkt{D^{\rom{2}}_i}{W'},$ we either find a set $Z$ of size at most $8\Fkt{f_P}{\cliqueToCrossRowGridFunction{t}}$ that hits all directed $X_{\rom{2}}^i$-paths, or there exists a family $\mathcal{J}'$ of $4\Fkt{f_P}{\cliqueToCrossRowGridFunction{t}}$ pairwise disjoint directed $X_{\rom{2}}^i$-paths in $\Fkt{D^{\rom{2}}_i}{W'}.$
	
	In case we obtain the latter, by construction of $\Fkt{D^{\rom{2}}_i}{W'},$ no path in $\mathcal{J}'$ contains a vertex of $W_2.$
	Back in the digraph $D,$ let us consider the \hyperref[def:tilingII]{tier II tiling} $\tiling''\coloneqq \TierIITiling{\tiling,i,w}{W'}$ of width $w$ of $W''\coloneqq \InducedSubgraph{W'}{\tiling,i}.$
	Note that the width of $W'$ is at most $2w$ smaller than the width of $W^{*}.$
	Thus, by choice of $k,$ $W''$ contains a cylindrical wall $W'''$ of order $\Fkt{f_W}{\cliqueToCrossRowGridFunction{t}}$ such that the perimeter of every tile $\tile \in\Class_i$ for which $x_{\tile}$ is an endpoint of some path in $\mathcal{J}'$ bounds a cell of $W'''.$
	Let $\tiling'''$ be a tiling of $W'''$ such that the perimeter of every $\tile\in\Class_i$ for which $x_{\tile}$ is a start- or end-vertex of a path in $\mathcal{J}'$ is the centre of some tile in $\tiling''.$
	We now consider a four-colouring $\Set{\Class_1',\dots,\Class_4'}$ of $\tiling'''.$
	Then, there exists $j\in[4]$ and a family $\mathcal{J}''$ of size $\Fkt{f_P}{\cliqueToCrossRowGridFunction{t}}$ such that the start-vertex of every path in $\mathcal{J}''$ belongs to a tile of $\Class_i$ whose perimeter is the centre of a tile in $\Class_j'.$
	For every $J''\in\mathcal{J}''$ do the following:
	Let $\tile_1,\tile_2\in\Class_i$ be the two tiles such that $J''$ is a directed $x_{\tile_1}$-$x_{\tile_2}$-path.
	Then, find a directed $W'''$-path $J$ that starts on the perimeter of $\tile_1$ and ends on the perimeter of $\tile_2.$
	Add the path $J$ to a family $\mathcal{J}.$
	Hence, $\mathcal{J}$ is a family of pairwise disjoint directed $W'''$-paths whose start- and end-vertices all lie on the centres of distinct tiles of $\tiling''$ and that all start at the centres of tiles from $\Class_j'.$
	Thus, we can apply \cref{lem:directed_long_jumps} and by \cref{claim:long_jumps_imply_minor} we find $\CrossRowGrid{t}$ as a butterfly minor.
	
	Therefore, we may assume that we find a set $Z$ of size at most $8\Fkt{f_P}{\cliqueToCrossRowGridFunction{t}}$ that hits all directed $X_{\rom{2}}^i$-paths.
	Let $Z^2_{i,\xi,\xi'}\coloneqq Z\cap \V{D},$ and $\MarkedTiles{2}_{i,\xi,\xi'}\coloneqq\CondSet{\tile \in\tiling}{x_{\tile} \in Z}.$
	Since $\Abs{Z}\leq 2\Fkt{f_P}{\cliqueToCrossRowGridFunction{t}},$ the demanded bounds on the sizes of the two sets are met.
	Moreover, because $Z$ meets every directed $X_{\rom{2}}^i$-path in $\Fkt{D^{\rom{2}}_i}{W'},$ every directed path with start- and end-vertex in distinct tiles of $\Class_i$ which is otherwise disjoint from $W$ contains a vertex from $Z^2_{i,\xi,\xi'}$ or meets a tile from $\MarkedTiles{2}_{i,\xi,\xi'}.$
\end{proof}

\subsection{Proof of \texorpdfstring{\cref{thm:directed_flat_wall}}{Theorem 6.3}}

As we have seen in \cref{lem:many_cross_slices_yield_minor}, finding sufficiently many cross slices in a wall yields a cross-row grid as a minor.
Thus, we introduce the following lemmata, which describe ways to find cross slices in a wall.

\begin{lemma}
	\label{lem:local_jump_yields_cross_slice}
	Let $W$ be a wall and $\mathcal{T}$ a \hyperref[def:tiling]{tiling} of $W$ and $d \in \N^{+}.$
	Let $P$ be a path with start-vertex in the centre of a \hyperref[def:tile]{tile} $\tile_{i,j,d} \in \tiling$ and end-vertex in the centre of a \hyperref[def:tile]{tile} $\tile_{i',j',d} \in \tiling$ with $\Abs{i-i'} \geq d+1,$ then the \hyperref[def:slice]{slice} between $Q_{\min\Set{i,i'}}$ and $Q_{\max\Set{i,i'}+d}$ contains a cross which lies in the \hyperref[def:strip]{strip} between $R_{\min\Set{j,j'}-1}$ and $R_{\max\Set{j,j'}+d+1}.$
\end{lemma}
\begin{proof}
	Without loss of generality we fix the parametrisation of $W$ such that the centres of the tiles have an out-path as upper and an in-path as lower perimeter.
	The proof for the other case can be obtained by swapping in- and out-paths in the proof.
	Furthermore, we assume without loss of generality that $i> i',$ the other case work analogously.
	
	Let $\ell \coloneqq \max\Set{j,j'}+d$ and $u \coloneqq \min\Set{j,j'},$ that is, the \hyperref[def:strip]{strip} between row $\row_{\ell}$ and $\row_u$ contains all vertices of $\tile_{i,j,d}$ and $\tile_{i',j',d}.$
	Now we build a path $P'$ extending $P$ by starting in the upper right corner of $\tile_{i,j,d},$ then within $\tile_{i,j,d}$ reach $\Start{P},$ follow $P$ to $\tile_{i',j',d}$ and then, within $\tile_{i',j',d},$ reach the lower left corner of $\tile_{i',j',d}.$
	
	As $\Abs{i-i'} \geq d+1,$ there is a vertical cycle $Q$ separating the two tiles.
	
	Next, we build a path $P''$ by starting in the intersection of $Q_{i'}$ and $\outPath_{\ell - 1},$ following along $\outPath_{\ell - 1}$ until meeting $Q,$ then following along $Q$ until meeting $\outPath_{u + 1}$ and finally following $\outPath_{u + 1}$ until meeting $Q_{i+d}.$
	
	The two paths $P'$ and $P''$ yield the desired cross.
\end{proof}

\begin{lemma}
	\label{lem:perimeter_jump_yields_cross_slice}
	Let $W$ be a wall and $\mathcal{T}$ a \hyperref[def:tiling]{tiling} of $W.$
	If there is a perimeter jump $P$ over $W,$ then $W$ contains a cross slice.
\end{lemma}
\begin{proof}
	We consider the case where $P$ ends on the perimeter of $W$ (the case where $P$ starts on the perimeter works similarly).
	Let $\tile \in\tiling$ be the \hyperref[def:tile]{tile} in which $P$ starts.
	There are two vertical paths, $Q_i$ and $Q_{i+1},$ separating $\tile$ from $\per{W}.$
	Without loss of generality, assume that the column index of $\tile$ is larger than $i+1$ and $\End{P}$ lies on $Q_1.$
	We consider the slice between $Q_{i+1}$ and the left path $Q_{\ell}$ of the perimeter of $\tile.$
	Let $\inPath_j$ be the lower path of the perimeter of $\tile.$
	Then, $\outPath_{j+1}$ contains a path $P_1$ from $Q_{i+1}$ to $Q_{\ell}$ that is disjoint from $P.$
	Next, we extend $P$ to a path $P_2$ by starting with the shortest path within $\tile$ from $Q_{\ell}$ to $\Start{P},$ then we add $P,$ from $\End{P}$ follow along $Q_1$ until meeting $\outPath_{j+2},$ and then follow $\outPath_{j+2}$ until reaching $Q_{i+1}.$
	Together, $P_1$ and $P_2,$ yield a cross over the slice between $Q_{i+1}$ and $Q_{\ell}.$
\end{proof}

Cross slices are the obstruction to flatness that we use in our proofs in order to construct the cross-row grid.
Thus, we prove next that walls without cross slices do not allow for certain non-planar behaviour.

\begin{lemma}
	\label{lem:no_jump_over_column_K5}
	Let $W$ be a wall of order $k\geq 5$ in a digraph $D$ that is free of cross slices.
	Then, there is no directed path that is internally disjoint from $W$ whose start- and end-vertex are separated by two vertical paths of $W.$
\end{lemma}
\begin{proof}
	Suppose such a path $P$ exists and let $Q_i$ and $Q_{i+1}$ be the vertical paths of the wall separating the start- and the end-vertex of $P.$
	By possibly extending $P$ we can assume that $P$ starts and ends in branch vertices of $W.$
	Additionally, assume without loss of generality that $P$ starts on a vertical path $Q_r$ with $r \geq i+1.$
	
	Let $Q_{\ell}$ be the vertical path containing $\End{Q}$ and $\row_{u}$ be the row containing $\End{Q}.$
	Moreover, let $\row_{b}$ be the row containing $\Start{Q}.$
	We show that the slice between $Q_{\ell}$ and $Q_{r}$ contains a cross.
	To this end we construct a path $P'$ starting in the intersection of $\outPath_{u-1},$ from there following $\outPath_{u-1}$ until reaching $Q_i$ (this subpath might be empty as $i = \ell$ is possible), then we follow $Q_i$ until reaching $\outPath_{b+1},$ then we follow along $\outPath_{b+1}$ until reaching $Q_{r}$ (again this part might be empty as possibly $r = i+1$).
	Together, $P$ and $P'$ yield a cross in the slice between $Q_{\ell}$ and $Q_{r},$ a contradiction.
\end{proof}

\begin{lemma}
	\label{lem:no_jump_over_row_K5}
	Let $W$ be a wall of order $k\geq 5$ in a digraph $D$ that is free of cross slices.
	Then, there is no directed path that is internally disjoint from $W$ whose start- and end-vertex are separated by a row of $W$ and at least one of the endpoints does not lie on $\per{W}.$
\end{lemma}
\begin{proof}
	Suppose there is such a path $P$ and let $\row_j$ be the row separating its start- and end-vertex.
	Let $Q_i$ be the vertical path of $W$ containing $\Start{P}.$
	By \cref{lem:no_jump_over_column_K5}, we obtain that $\End{P}$ lies on $C_{i+1},$ $C_i$ or $C_{i+1}.$
	
	Consider $\End{P}$ lying on $C_{i-1},$ then $P$ together with the subpath of $\outPath_j$ yield a cross in the slice between $C_{i-1}$ and $C_i,$ a contradiction.
	
	Similarly, in case that $\End{P}$ lies on $C_{i+1}$ we obtain a cross in the slice between $C_i$ and $C_{i+1}$ built by $P$ and the subpath of $\inPath_j$ starting on $C_{i+1}$ and ending on $C_i,$ which again yields a contradiction.
	
	Next, consider the case that $\End{P}$ lies on $C_{i}$ as well.
	By assumption, we have $1 < i < k.$
	Let $\row_m$ be the row containing $\Start{P}$ and $\row_{m'}$ be the row containing $\End{P},$ so we have $\Abs{m-m'} \geq 2.$
	We show that there is a cross in the slice between $C_{i-1}$ and $C_{i+1}.$
	First, we construct a path $P_1$ from $C_{i-1}$ to $C_{i+1}.$
	We start in the intersection of $C_{i-1}$ and $\outPath_m$ and follow $\outPath_m$ until reaching $C_i,$ we follow $C_i$ until reaching $\Start{P}$ (this subpath is empty if $\Start{P}$ lies on $\outPath_m$), then we append $P,$ next, from $\End{P}$ we follow $C_i$ again until reaching $\outPath_{m'+1},$ we then follow $\outPath_{m'+1}$ until reaching $C_{i+1}.$
	Second, we choose $P_2$ to be the subpath of $\inPath_{m+1}$ that starts in $C_{i+1}$ and ends in $C_{i-1}.$
	Together $P_1$ and $P_2$ yield a cross over the slice between $C_{i-1}$ and $C_{i+1},$ a contradiction.
\end{proof}

\begin{lemma}
	\label{lem:no_cross_slice_implies_flat}
	Let $D$ be a digraph and $W$ a wall in $D.$
	If $W$ does not contain any cross \hyperref[def:slice]{slices}, then $W$ is \hyperref[def:flat_wall]{flat}.
\end{lemma}
\begin{proof}
	In favour of readability we use $D$ to refer to the strongly connected component of the digraph $D$ containing $W.$
	Let $k+4$ be the order of $W$ and let $W^{-} \subseteq W$ be the subwall of order $k$ contained in $W-\border{W}.$
	
	Let $X'_1$ be the minimal set of vertices containing every vertex that is reachable from $W^{-}$ in $D- \border{W}$ and let $Y_1 \coloneqq \V{D} \setminus X'_1.$
	Then, $\Separation{Y_1}{X_1}{}{r}$ where $X_1 \coloneqq X'_1 \cup \V{\border{W}}$ is one of the required separations for \cref{flat:separation}.
	Let $X'_2$ be the minimal set of vertices containing every vertex that reaches $W^{-}$ in $D-\border{W}$ and let $Y_2 \coloneqq \V{D} \setminus X'_2$
	Then, $\Separation{X_2}{Y_2}{}{r}$ with $X_2 \coloneqq X'_2 \cup \V{\border{W}}$ is the other separation required for \cref{flat:separation}.
	
	Consider the cylindrical society $\Brace{\compass{W},\Omega_1,\Omega_2}$ with sets $\V{\Omega_1}=\V{Q_1}$ and $\V{\Omega_2}=\V{Q_{k+4}}.$
	Let $\Delta'$ be obtained from a closed disk $\Delta''$ by removing an open disk disjoint from the boundary of $\Delta''.$
	Fix an embedding $\Gamma$ of $W$ with $\V{\Omega_1}$ being exactly the vertices in $\boundary{\Delta''}$ and $\V{\Omega_2}$ being exactly the vertices in $\boundary{\Delta'},$ both consistent with the appearance of the vertices along the vertical cycles in $W.$
	Now, $\skeleton\coloneqq \Brace{\Gamma,W,\emptyset}$ is a \hyperref[def:skeleton]{skeleton}.
	We obtain $\Gamma^{+}$ from $\Gamma$ by embedding a \hyperref[def:bridge]{bridge} of a brick with in- and out-attachments in more than one perimeter into such a brick.
	Otherwise, embed bridges into any brick they have an attachment in.
	
	\begin{claim}
		\label{flat:claim_1}
		The tuple $\rho \coloneqq \Brace{\Gamma^{+},\skeleton}$ is a \hyperref[def:sphere_decomposition]{$\sphere$-decomposition}.
	\end{claim}
	\begin{claimproof}
		Suppose there is an undirected cycle $C$ in $\rigidgraph{\skeleton}$ such that there is a directed path $P$ between $H^{+}_1$ and $H^{+}_2.$
		By definition of $H^{+}_1$ and $H^{+}_2,$ $\Start{P}$ can be reached from a vertex of $H^{+}_1$ or $H^{+}_2$ and $\End{P}$ can reach a vertex of $H^{+}_1$ or $H^{+}_2.$
		Thus, we can assume that $\Start{P}$ and $\End{P}$ are vertices of $W.$
		If there are at least two vertical paths separating $\Start{P}$ and $\End{P},$ then by \cref{lem:no_jump_over_column_K5}, we obtain a cross slice, a contradiction.
		So both $\Start{P}$ and $\End{P}$ lie in the slice between $Q_{i}$ and $Q_{i+2}$ for some $i.$
		By possibly prolonging $P$ at both ends, we can also assume that both $\Start{P}$ and $\End{P}$ lie in a row of $W$ and still in the slice between $Q_{i}$ and $Q_{i+2}.$
		By definition of $H^{+}_1$ and $H^{+}_2$ there is a whole row $\row_j$ separating $\Start{P}$ and $\End{P}.$
		By \cref{lem:no_jump_over_row_K5}, this implies that there is a cross slice in $W,$ a contradiction.
	\end{claimproof}
	
	\Cref{flat:claim_1} ensures \cref{flat:rendition}.
	
	
	\begin{claim}
		\label{flat:claim_2}
		The strongly connected component $D'$ has a torus decomposition centred at $\rho.$
	\end{claim}
	\begin{claimproof}
		Suppose there is an undirected cycle $C$ in $\rigidgraph{\skeleton}$ such that there is a directed path $P$ between $H^{+}_1$ and $H^{+}_2.$
		Assume without loss of generality that $H^{+}_2$ contains $D-\compass{W}.$
		By, \cref{flat:claim_1}, there are no jumps between $H^{+}_1$ and $H^{+}_2 \cap \compass{W},$ so $P$ is a path between $H^{+}_1$ and $D - \compass{W}.$
		As $D$ is strongly connected it can be extended into a perimeter jump over $W$ and thus, by \cref{lem:perimeter_jump_yields_cross_slice}, $W$ contains a cross slice, a contradiction. 
	\end{claimproof}
	
	\Cref{flat:claim_2} ensures \cref{flat:torus}, so in the end we obtain that $W$ is flat.
\end{proof}

We need the following local version of Menger's Theorem in order to complete the proof of \cref{thm:directed_flat_wall}.

\begin{theorem}[Menger's Theorem~\cite{menger1927allgemeinen}]
	\label{thm:directedlocalmenger}
	Let $D$ be a digraph and $X,Y\subseteq\V{D}$ be two sets of vertices, then the maximum number of pairwise disjoint directed $X$-$Y$-paths in $D$ equals the minimum size of a set $S\subseteq\V{G}$ such that every directed $X$-$Y$-path in $D$ contains a vertex of $S.$
\end{theorem}

Now we want state the proof of our main theorem.
Simply put, we consider two cases: either there are enough long jumps to build a cross-row grid or we can find a large slice in the wall that does not contain any long jumps.
Then, we divide this slice into smaller parts.
Either we find short jumps in enough such parts to build a cross-row grid again, or we find a part free of any jumps which then yields the flat wall.

Based on this idea the proof is split into two phases.
A vertex of a wall $W$ is said to be \defineSubindex{marked}{vertex of a wall} if it belongs to a separator obtained from \cref{thm:directedlocalmenger} and \cref{lem:longjumps1}, or through \cref{lem:longjumpsphase2}.
A tile is \emph{marked} if it contains a marked vertex or is replaced by a marked vertex in the construction of an auxiliary graph (type \rom{1} or \rom{2}).
That is, we \say{mark} every vertex or tile involved in a long jump.
A \hyperref[def:slice]{slice} of $W$ is said to be \define{clear} if it does not contain vertices that are marked.
This happens in two steps, one for jumps between tiles of different colour and one for jumps between tiles of the same colour.
In each of the two steps we introduce families of marked tiles and vertices.
In the end we can bound the number of columns containing marked tiles.
This gives us a clear slice of $W,$ which we then use in phase two of the proof.
In this second phase we split the clear slice into smaller parts, identify among those a part which does not contain any jumps or crossings, and prove that this part is flat.

\directedFlatWallThm*

\begin{proof}
	Let $r,t\in\N$ be positive integers, $D$ be a digraph and $W$ be a cylindrical wall of order $\WallFkt{t,r},$ where $\WallFkt{t,r}$ will be determined throughout the proof.
	So, we introduce constants $d_1$ and $d_2,$ for which we make more and more assumptions in the form of lower bounds.
	Let us assume
	\begin{equation*}
		\WallFkt{t,r}\geq 3d_1.
	\end{equation*}
	Then, $W$ is a cylindrical wall of order $3d_1$ with vertical paths $Q_1, \dots,$ $Q_{3k},$ in-paths $\inPath_1, \dots,$ $\inPath_{3k}$ and out-paths $\outPath_1, \dots,$ $\outPath_{3k}$ and the triadic partition $\mathcal{W} = (W, d_1, \widetilde{W}_1,$ $\widetilde{W}_2,$ $\widetilde{W}_3,$ $\widetilde{W}^1,$ $\widetilde{W}^2,$ $\widetilde{W}^3).$
	Throughout the proof let us fix
	\begin{equation*}
		w \coloneqq 2\Fkt{f_w}{\cliqueToCrossRowGridFunction{t}}+2^7\Fkt{f_P}{\cliqueToCrossRowGridFunction{t}}.
	\end{equation*}
	
	\hypertarget{phaseone}{\paragraph{Phase~\rom{1}}}
	\hyperlink{phaseone}{Phase~\rom{1}} is divided into $2^{11+8\Fkt{f_P}{\cliqueToCrossRowGridFunction{t}}}\Fkt{f_P}{\cliqueToCrossRowGridFunction{t}}$ \emph{rounds},
	each of which is divided into two steps, \emph{\hyperlink{stepone}{Step~\rom{1}}} and \emph{\hyperlink{stepone}{Step~\rom{2}}}.
	Let $i \in [2^{11+8\Fkt{f_P}{\cliqueToCrossRowGridFunction{t}}} \Fkt{f_P}{\cliqueToCrossRowGridFunction{t}}].$
	After round $i$ is complete we require the following sets and graphs as its \emph{output} which then can be used in round $i+1.$
	\begin{itemize}
		\item $F_{\rom{1},i}\subseteq\V{D}$ such that
		\begin{equation*}
			\Abs{F_{\rom{1},i}}\leq i\Brace{2^{11}\Brace{w+1}^2\Fkt{f_P}{\cliqueToCrossRowGridFunction{t}}+2^5\Brace{w+1}^2\Fkt{f_P}{\cliqueToCrossRowGridFunction{t}}},
		\end{equation*}
			which are the vertices marked in round $i$
		\item $D_{i}\coloneqq D-F_{\rom{1},i},$
		\item $\mathcal{F}_{\rom{1},i}\subseteq \bigcup_{\xi,\xi'\in[w+1]\tiling_{\widetilde{W}_2,d_1,w,\xi,\xi'}}$ of size at most
		\begin{equation*}
			i\Brace{2^{11}\Brace{w+1}^2\Fkt{f_P}{\cliqueToCrossRowGridFunction{t}}+2^5\Brace{w+1}^2\Fkt{f_P}{\cliqueToCrossRowGridFunction{t}}},
		\end{equation*}
			which are the tiles marked in round $i,$ and
		\item a slice $W_i$ of $W_{i-1}$ of width
		\begin{align*}
			\Big(&\Brace{2^{12}\Brace{w+1}^3\Fkt{f_P}{\cliqueToCrossRowGridFunction{t}}+1} \cdot \\
				&\Brace{2^6\Brace{w+1}^3\Fkt{f_P}{\cliqueToCrossRowGridFunction{t}}+1}\Big)^{2^{11+8\Fkt{f_P}{\cliqueToCrossRowGridFunction{t}}}\Fkt{f_P}{\cliqueToCrossRowGridFunction{t}}-i} \cdot d_2
		\end{align*}
		that is clear with respect to $F_{\rom{1},i}$ and $\mathcal{F}_{\rom{1},i},$ such that every long jump over $W_{i}$ in $D_{i}$ contains a vertex of some tile in $\mathcal{F}_{\rom{1},j}\setminus\mathcal{F}_{\rom{1},j-1}$ for every $j\in[i].$
	\end{itemize}
	For $i=0$ we define $D_0\coloneqq D,$ $W_0\coloneqq \widetilde{W}_2$ as a slice of itself, and $F_{\rom{1},i} \coloneqq \emptyset$ as well as $\mathcal{F}_{\rom{1},i} \coloneqq \emptyset.$
	In this context, whenever we ask for a clear slice of the current slice $W_{i}$ or $W'_i$ we ask for a slice $W'$ such that there do not exist $\xi,\xi'\in[w+1]$ whose corresponding \hyperref[def:tiling]{tiling} of $\widetilde{W}_2$ has a \hyperref[def:tile]{tile} $\tile$ that is marked or contains a vertex of any separator set found so far, which satisfies $\V{\tile}\cap\V{W'}\neq\emptyset.$

	To be able to find a slice of width $d_2$ after the last round we therefore must fix
	\begin{align*}
		d_1 \geq \Big(&\Brace{2^{12}\Brace{w+1}^3\Fkt{f_P}{\cliqueToCrossRowGridFunction{t}}+1} \cdot\\
		&\Brace{2^6\Brace{w+1}^3\Fkt{f_P}{\cliqueToCrossRowGridFunction{t}}+1}\Big)^{2^{11+8\Fkt{f_P}{\cliqueToCrossRowGridFunction{t}}}\Fkt{f_P}{\cliqueToCrossRowGridFunction{t}}} \cdot d_2,
	\end{align*}
	and we further assume $d_2\geq 2^{16}\Fkt{f_W}{\cliqueToCrossRowGridFunction{t}}^2$ to make sure we can apply \cref{lem:longjumps1} and \cref{lem:longjumpsphase2} in every round.
	Note that this is not yet the final lower bound on $d_2,$ just an intermediate assumption.
	
	Next we describe the steps we perform in every round.
	Let $i\in[2^{11+8\Fkt{f_P}{\cliqueToCrossRowGridFunction{t}}}\Fkt{f_P}{\cliqueToCrossRowGridFunction{t}}]$ and suppose we are given sets $F_{I,i-1},$ $\mathcal{F}_{I,i-1}$ and graphs $D_{i-1},$ $W_{i-1}$ as input that satisfy the required invariants.
	
	\textbf{\hypertarget{stepone}{Step~\rom{1}}:}
	Let $k^{\rom{1}}_i$ be defined as follows.
	\begin{align*}
		k^{\rom{1}}_i \coloneqq \Big( &\Brace{2^{12}\Brace{w+1}^3\Fkt{f_P}{\cliqueToCrossRowGridFunction{t}}+1} \cdot\\ &\Brace{2^6\Brace{w+1}^3\Fkt{f_P}{\cliqueToCrossRowGridFunction{t}}+1} \Big)^{2^{11+8\Fkt{f_P}{\cliqueToCrossRowGridFunction{t}}}\Fkt{f_P}{\cliqueToCrossRowGridFunction{t}}-\Brace{i-1}} \cdot d_2
	\end{align*}
	For each of the two possible parametrisations of $W,$ and for every possible choice of $\xi,$ $\xi'\in[w+1],$ we consider the \hyperref[def:tiling]{tiling} $\tiling\coloneqq\tiling_{W_{i-1},k^{\rom{1}}_i,w,\xi,\xi'}$ together with its four-colouring $\Set{\Class_1,\dots,\Class_4}.$
	For each $j\in[4],$ we consider the smallest slice $W'$ of $W$ that contains all vertices which belong to some tile of $\tiling.$
	Consider $\Fkt{D_j^{\rom{1}}}{W'}$ to be the auxiliary digraph of type \rom{1} obtained from $D_{i-1}$ with the sets
	\begin{align*}
		X_{\rom{1}}^{\text{out}}&\coloneqq\CondSet{x_T^{\text{out}}}{T\in\Class_j}\text{, and}\\
		X_{\rom{1}}^{\text{in}}&\coloneqq\CondSet{x_T^{\text{in}}}{T\in\Class_j}.
	\end{align*}
	Additionally, we construct the set $Y_{\rom{1}}$ as follows:
	Let $Q$ and $Q'$ be the two cycles of $\per{\widetilde{W}_2}.$
	For every $j\in\left[\frac{3d_1}{4} \right],$ $Y_{\rom{1}}$ contains exactly one vertex of $Q\cap \outPath_{4j},$ $Q\cap \inPath_{4j+2},$ $Q'\cap \outPath_{4j},$ and $Q'\cap \inPath_{4j+2}$ each.
	Then, remove all vertices of $Y_{\rom{1}}$ that do not belong to $D_{i-1}.$
	Note that, by choice of $d_1$ and the bound on $F_{\rom{1},i-1},$ this does not significantly decrease the size of $Y_{\rom{1}}.$
	
	Then, if there is a family of $2^7\Fkt{f_P}{\cliqueToCrossRowGridFunction{t}}$ pairwise disjoint directed $X_{\rom{1}}^{\text{out}}$-$Y_{\rom{1}}$-paths in $\Fkt{D_j^{\rom{1}}}{W'},$ \cref{lem:longjumps1} implies the existence of a $\CrossRowGrid{t}$ butterfly minor, a contradiction.
	So, we may assume that there does not exist such a family and thus, by \cref{thm:directedlocalmenger}, we find a set $Z_1\subseteq\V{\Fkt{D_j^{\rom{1}}}{W'}}$ of size at most $2^7\Fkt{f_P}{\cliqueToCrossRowGridFunction{t}}$ that meets all these paths.
	With a similar argument, we either find a $\CrossRowGrid{t}$ butterfly minor, which would yield a contradiction, or a set $Z_2\subseteq\V{\Fkt{D_j^1}{W'}}$ of size at most $2^7\Fkt{f_P}{\cliqueToCrossRowGridFunction{t}}$ that meets all directed $Y_{\rom{1}}$-$X_{\rom{1}}^{\text{in}}$-paths in $\Fkt{D_j^{\rom{1}}}{W'}.$
	Let $\pi\in[2]$ indicate which of the two parametrisations of $W$ we are currently considering.
	We define the following two sets:
	\begin{align*}
		Z_{\pi,\xi,\xi',j}\coloneqq& \Brace{Z_1\cup Z_2}\cap\V{D}\text{, and}\\
		\mathcal{Z}_{\pi,\xi,\xi',j}\coloneqq& \CondSet{\tile\in\tiling}{\Brace{\V{\tile}\cup\Set{x_\tile^{\text{out}},x_\tile^{\text{in}}}}\cap\Brace{Z_1\cup Z_2}\neq\emptyset}.
	\end{align*}
	Note that $\max\Set{\Abs{Z_{\pi,\xi,\xi',j}},\Abs{\mathcal{Z}_{\pi,\xi,\xi',j}}}\leq 2^8\Fkt{f_P}{\cliqueToCrossRowGridFunction{t}}.$
	
	As we do not find a $\CrossRowGrid{t}$ butterfly minor at any point, the sets $Z_{\pi,\xi,\xi',j}$ and $\mathcal{Z}_{\pi,\xi,\xi',j}$ are well defined for every possible choice of $\pi\in[2],$ $\xi,\xi'\in[w+1],$ and $j\in[4].$
	Using these we define the following two sets of marked vertices and tiles:
	\begin{align*}
		F'_{\rom{1},i}\coloneqq& \bigcup_{\pi\in[2]}\bigcup_{\xi,\xi'\in[w+1]}\bigcup_{j\in[4]}Z_{\pi,\xi,\xi',j}\text{, and}\\
		\mathcal{F}'_{\rom{1},i}\coloneqq& \bigcup_{\pi\in[2]}\bigcup_{\xi,\xi'\in[w+1]}\bigcup_{j\in[4]}\mathcal{Z}_{\pi,\xi,\xi',j}.
	\end{align*}
	Consequently, we have
	\begin{equation*}
		\max\Set{\Abs{F'_{\rom{1},i}},\Abs{\mathcal{F}'_{\rom{1},i}}}\leq 2^{11}\Brace{w+1}^2\Fkt{f_P}{\cliqueToCrossRowGridFunction{t}}.
	\end{equation*}
	
	Removing all vertices in $F'_{\rom{1},i}$ and tiles in $\mathcal{F}'_{\rom{1},i}$ yields a clear slice $W'_i\subseteq W_{i-1}$ of width 
	\begin{align*}
		&\Brace{2^{12}\Brace{w+1}^3\Fkt{f_P}{\cliqueToCrossRowGridFunction{t}}+1}^{2^{11+8\Fkt{f_P}{\cliqueToCrossRowGridFunction{t}}}\Fkt{f_P}{\cliqueToCrossRowGridFunction{t}}-i}\cdot\\
		&\Brace{2^6\Brace{w+1}^3\Fkt{f_P}{\cliqueToCrossRowGridFunction{t}}+1}^{2^{11+8\Fkt{f_P}{\cliqueToCrossRowGridFunction{t}}}\Fkt{f_P}{\cliqueToCrossRowGridFunction{t}}-i+1}\cdot d_2
	\end{align*}
	which does not contain a marked vertex.
	Note that we loose the additional factor of $2\Brace{w+1}$ because we remove whole tiles of width $w$ from $W_{i-1}.$
	Let $D'_i \coloneqq D_{i-1}-F'_{\rom{1},i}.$
	So now we have a preliminary version for all the structures we need to start the next round.
	This concludes \hyperlink{stepone}{Step~\rom{1}} of round $i.$
	
	\begin{claim}
		\label{claim:stepIclaim}
		Every long jump $J$ over $W'_i$ in $D'_i$ whose endpoints belong to tiles of different colour contains a vertex of a tile from $\mathcal{F}_{\rom{1},j}\setminus \mathcal{F}_{\rom{1},j-1}$ for every $j\in[i-1],$ and it contains a vertex of a tile from $\mathcal{F}'_{\rom{1},i}.$
	\end{claim}
	
	\begin{claimproof} 
		Suppose $J$ is also a long jump over $W_{i-1}$ in $D_{i-1},$ then, as $J$ still exists in $D'_{i},$ the tile in whose centre $J$ starts, or the tile in whose centre $J$ ends for some choices of $\pi \in[2],$ $\xi,\xi'\in[w+1],$ and $j\in[4],$ are marked and therefore do not belong to $W'_i.$
		
		Hence, $J$ contains some vertex of $W_{i-1}$ as an internal vertex.
		Let $\tile_{\mathsf{s}}$ be the \hyperref[def:tile]{tile} of $W'_i$ in whose centre $J$ starts, and let $\tile$ be the first tile from the same \hyperref[def:tiling]{tiling} of $W_{i-1},$ that $J$ meets after $\tile_{\mathsf{s}}.$
		Let $J'$ be the shortest subpath of $J$ with endpoints in $\tile_{\mathsf{s}}$ and $\tile.$
		Then, $J'$ is a long jump over $W_{i-1}$ in $D_{i-1}.$
		Therefore, by our assumptions on the input of round $i$ of \hyperlink{phaseone}{Phase~\rom{1}}, the first part of our claim is satisfied.
		Moreover, if $\tile$ has a different colour than $\tile_{\mathsf{s}},$ then $\tile$ is marked.
		So assume $\tile$ has the same colour as $\tile_{\mathsf{s}}.$
		Nonetheless, because $\tile_{\mathsf{s}}$ and the \hyperref[def:tile]{tile} $\tile_{\mathsf{t}}$ that contains the endpoint of $J$ in the current tiling have different colours, $J$ contains a directed subpath $J''$ which is a long jump over $W_{i-1}$ and attaches to tiles of different colour.
		Hence our claim follows.
	\end{claimproof}
	
	With this we are ready for \hyperlink{stepone}{Step~\rom{2}} of round $i.$
	
	\textbf{\hypertarget{steptwo}{Step~\rom{2}}:}
	For this step let
	\begin{align*}
		k^{\rom{2}}_i \coloneqq &\Brace{2^{12}\Brace{w+1}^3\Fkt{f_P}{\cliqueToCrossRowGridFunction{t}}+1}^{2^{11+8\Fkt{f_P}{\cliqueToCrossRowGridFunction{t}}}\Fkt{f_P}{\cliqueToCrossRowGridFunction{t}}-i} \cdot\\
		& \Brace{2^6\Brace{w+1}^3\Fkt{f_P}{\cliqueToCrossRowGridFunction{t}}+1}^{2^{11+8\Fkt{f_P}{\cliqueToCrossRowGridFunction{t}}}\Fkt{f_P}{\cliqueToCrossRowGridFunction{t}}-i+1} \cdot d_2.
	\end{align*}
	We are mainly concerned with the digraph $D'_i.$
	In \hyperlink{stepone}{Step~\rom{2}} it suffices to fix one parametrisation of $W$ because the construction of the \hyperref[def:auxiliarydigraphII]{type \rom{2} auxiliary digraph} leaves the complete interior of tiles that are in the same colour class intact instead of only their \hyperref[def:tile]{centres}.
	For every pair of $\xi,\xi'\in[w+1]$ we consider the \hyperref[def:tiling]{tiling} $\tiling\coloneqq\tiling_{W'_i,k^{\rom{2}}_i,w,\xi,\xi'}$ together with a four-colouring $\Set{\Class_1,\dots,\Class_4}.$
	Then, for every $j\in[4]$ we apply \cref{lem:longjumpsphase2}, which, as it cannot yield a $\CrossRowGrid{t},$ produces two sets
	\begin{align*}
		Z^{2}_{\xi,\xi',j}\subseteq{} &\V{D'_i}\text{ of size at most }2^3\Fkt{f_P}{\cliqueToCrossRowGridFunction{t}}\text{, and}\\
		\mathcal{Z}^{2}_{\xi,\xi',j}\subseteq{} &\tiling\text{ of size at most }2^3\Fkt{f_P}{\cliqueToCrossRowGridFunction{t}},
	\end{align*}
	such that every directed $\V{W'_i}$-path whose endpoints belong to different tiles of $\Class_j,$ contains a vertex of some tile in $\mathcal{Z}^{2}_{\xi,\xi',j}.$
	This allows us to form the two following sets of marked vertices and tiles:
	\begin{align*}
		F''_{\rom{1},i}\coloneqq& \bigcup_{\xi,\xi'\in[w+1]}\bigcup_{j\in[4]}Z^2_{\xi,\xi',j}\text{, and}\\
		\mathcal{F}''_{\rom{1},i}\coloneqq& \bigcup_{\xi,\xi'\in[w+1]}\bigcup_{j\in[4]}\mathcal{Z}^2_{\xi,\xi',j}.
	\end{align*}
	As a result we obtain $\max\Set{\Abs{F''_{\rom{1},i}},\Abs{\mathcal{F}''_{\rom{1},i}}}\leq 2^5\Brace{w+1}^2\Fkt{f_P}{\cliqueToCrossRowGridFunction{t}},$ and we are able to produce the two sets $F_{\rom{1},i}$ of marked vertices and $\mathcal{F}_{\rom{1},i}$ of marked tiles, which are passed on to the next round.
	\begin{align*}
		F_{\rom{1},i}\coloneqq{} & F'_{\rom{1},i}\cup F''_{\rom{1},i}\cup F_{\rom{1},i-1}\text{, and}\\
		\mathcal{F}_{\rom{1},i}\coloneqq{} & \mathcal{F}'_{\rom{1},i}\cup \mathcal{F}''_{\rom{1},i}\cup \mathcal{F}_{\rom{1},i-1}.
	\end{align*}
	The bounds on $F_{\rom{1},i}$ and $\mathcal{F}_{\rom{1},i}$ follow from the bounds on $F'_{\rom{1},i}$ and $F''_{\rom{1},i},$ $\mathcal{F}'_{\rom{1},i}$ and $\mathcal{F}''_{\rom{1},i},$ and the assumptions on the input of round $i$ respectively.
	
	The pigeon hole principle allows us to find a clear slice $W_i\subseteq W'_i$ of width
	\begin{align*}
		&\Big(\Brace{2^{12}\Brace{w+1}^3\Fkt{f_P}{\cliqueToCrossRowGridFunction{t}}+1} \cdot\\
		&\Brace{2^6\Brace{w+1}^3\Fkt{f_P}{\cliqueToCrossRowGridFunction{t}}+1}\Big)^{2^{11+8\Fkt{f_P}{\cliqueToCrossRowGridFunction{t}}}\Fkt{f_P}{\cliqueToCrossRowGridFunction{t}}-i} \cdot d_2,
	\end{align*}
	which does not contain a marked vertex, that is, no vertex from $F_{\rom{1},i}$ or $\mathcal{F}_{\rom{1},i}.$
	Similar to \hyperlink{stepone}{Step~\rom{1}}, we loose the additional factor of $2\Brace{w+1}$ because we remove whole tiles of width $w$ from $W'_i.$
	Finally, let $D_i \coloneqq D'_i-F''_{\rom{1},i},$ which concludes \hyperlink{stepone}{Step~\rom{2}} of round $i.$
	
	\begin{claim}
		\label{claim:stepIIclaim}
		Every long jump over $W_i$ in $D_i$ contains a vertex of some tile in $\mathcal{F}_{\rom{1},j}\setminus\mathcal{F}_{\rom{1},j-1}$ for every $j\in[i].$
	\end{claim}
	\begin{claimproof}
		Let $J$ be a long jump over $W_i$ in $D_i,$ and let $\tiling$ be a \hyperref[def:tiling]{tiling} of $\widetilde{W}_2$ defined by $w$ and some $\xi,\xi'\in[w+1]$ such that $J$ starts at the centre of some \hyperref[def:tile]{tile} $\tile_{\mathsf{s}}\in\tiling.$
		Suppose all tiles of $\tiling$ that contain vertices of $J$ belong to the same colour.
		Then, $J$ must have existed during the corresponding part of \hyperlink{stepone}{Step~\rom{2}} of round $i$ and thus either $\tile_{\mathsf{s}}$ or $\tile_{\mathsf{t}}\in\tiling,$ which is the tile that contains the endpoint of $J,$ must have been marked, a contradiction.
		Therefore $J$ must contain at least one tile of a colour different than the one of $\tile_{\mathsf{s}}.$
		Moreover, we may assume $\tile_{\mathsf{s}}$ and $\tile_{\mathsf{t}}$ to be of the same colour as otherwise we would be done by \cref{claim:stepIclaim}.
		Next, suppose $J$ is also a long jump over $W_{i-1},$ then again $J$ would have been considered during \hyperlink{stepone}{Step~\rom{2}} as a long jump connecting two tiles of the same colour and thus $\tile_{\mathsf{s}}$ or $\tile_{\mathsf{t}}$ would have been marked, a contradiction.
		Therefore, $J$ contains a vertex of some tile from $W_{i-1}.$
		Let $J'$ be a shortest subpath from $\tile_{\mathsf{s}}$ to some tile $\tile$ of $W_{i-1},$ then $J'$ is a long jump over $W_{i-1}$ and thus $J$ contains a vertex of some tile of $\mathcal{F}_{\rom{1},j}\setminus \mathcal{F}_{\rom{1},j-1}$ for every $j\in[i-1]$ by our assumptions on the input of round $i.$
		If $\tile$ has a different colour than $\tile_{\mathsf{s}},$ then $\tile$ would have been marked in \hyperlink{stepone}{Step~\rom{1}} of round $i,$ and if $\tile$ shares the colour of $\tile_{\mathsf{s}},$ then it must have been marked in \hyperlink{stepone}{Step~\rom{2}} of round $i.$
		Either way our claim follows.
	\end{claimproof}
	
	From \cref{claim:stepIIclaim} it follows that we satisfy all requirements for the output of round $i$ and thus, round $i$ is complete.
	We continue until we finish round $2^{11+8\Fkt{f_P}{\cliqueToCrossRowGridFunction{t}}}\Fkt{f_P}{\cliqueToCrossRowGridFunction{t}}$ and obtain the following four objects as its output:
	\begin{itemize}
		\item a slice $W_\rom{1}\coloneqq W_{2^{11+8\Fkt{f_P}{\cliqueToCrossRowGridFunction{t}}}\Fkt{f_P}{\cliqueToCrossRowGridFunction{t}}}$ of width $d_2,$
		\item a set $A_{\rom{1}} \coloneqq F_{\rom{1},2^{11+8\Fkt{f_P}{\cliqueToCrossRowGridFunction{t}}}\Fkt{f_P}{\cliqueToCrossRowGridFunction{t}}}$ with
		\begin{equation*}
			\Abs{A_{\rom{1}}} \leq (\cliqueToCrossRowGridFunction{t})^{28}2^{60+2^{10}(\cliqueToCrossRowGridFunction{t})^8},
		\end{equation*}
		\item a digraph $D_\rom{1}\coloneqq D_{2^{11+8\Fkt{f_P}{\cliqueToCrossRowGridFunction{t}}}\Fkt{f_P}{\cliqueToCrossRowGridFunction{t}}} = D - A_{\rom{1}},$ and
		\item a sequence $\mathcal{F}_{\rom{1},1}\subseteq \mathcal{F}_{\rom{1},2}\subseteq\dots\subseteq \mathcal{F}_{\rom{1},2^{11+8\Fkt{f_P}{\cliqueToCrossRowGridFunction{t}}}\Fkt{f_P}{\cliqueToCrossRowGridFunction{t}}}$ such that for every $i\in[2^{11+8\Fkt{f_P}{\cliqueToCrossRowGridFunction{t}}}\Fkt{f_P}{\cliqueToCrossRowGridFunction{t}}]$ every long jump over $W_\rom{1}$ in the graph  $D_\rom{1}$ contains a vertex of some tile in $\mathcal{F}_{\rom{1},i}\setminus\mathcal{F}_{\rom{1},i-1}.$
	\end{itemize}
	This brings us to the final claim of \hyperlink{phaseone}{Phase~\rom{1}}.
	
	\begin{claim}\label{claim:nolongjumps}
		If there is a long jump over $W_\rom{1}$ in $D_\rom{1},$ then there exists a $\CrossRowGrid{t}$ as a butterfly minor in $D.$
	\end{claim}
	\begin{claimproof}
		Let $J$ be a long jump over $W_\rom{1}$ in $D_\rom{1}.$
		We fix a parametrisation of $W,$ $\xi,\xi'\in[w+1],$ and $c\in[4]$ such that there exists a \hyperref[def:tile]{tile} $\tile_{\mathsf{s}}\in\tiling\coloneqq\tiling_{W_0,d_1,w,\xi,\xi'}$ of colour $c$ whose centre contains the start-vertex of $J.$
		Let $\tile_{\mathsf{t}}\in\tiling$ be the tile that contains the end-vertex of $J.$
		As $J$ is a long jump, note that $\tile_{\mathsf{s}}\neq \tile_{\mathsf{t}}.$
		We claim that every internal vertex of $J$ that belongs to $W$ belongs to some tile from $\mathcal{F}_{\rom{1},2^{11+8\Fkt{f_P}{\cliqueToCrossRowGridFunction{t}}}\Fkt{f_P}{\cliqueToCrossRowGridFunction{t}}}.$
		This is because otherwise we could find a directed path from the centre of $\tile_{\mathsf{s}}$ to the perimeter of $\widetilde{W_2},$ contradicting the construction in \hyperlink{stepone}{Step~\rom{1}} of \hyperlink{phaseone}{Phase~\rom{1}}, or we would have a directed path between two tiles of the same colour, where both of them are unmarked.
		This second outcome contradicts the construction in \hyperlink{stepone}{Step~\rom{2}} of \hyperlink{phaseone}{Phase~\rom{1}}.
		
		\textbf{Constructing $\mathcal{L}_0.$}
		Now, we create a family $\mathcal{L}_0$ of $2^{9+8\Fkt{f_P}{\cliqueToCrossRowGridFunction{t}}} \cdot \Fkt{f_P}{\cliqueToCrossRowGridFunction{t}}$ pairwise disjoint subpaths of $J$ with the following properties:
		\begin{enumerate}
			\item for every $L\in\mathcal{L}_0,$ let $\tile_{L,1}$ be the tile of $\tiling$ containing $\Start{L}$ and $\tile_{L,2}$ be the tiles of $\tiling$ containing $\End{L},$ then there exist distinct $i_{L,1},$ $i_{L,2}\in [2^{11+8\Fkt{f_P}{\cliqueToCrossRowGridFunction{t}}} \Fkt{f_P}{\cliqueToCrossRowGridFunction{t}}]$ such that $\Start{L}$ is a vertex of a tile from $\mathcal{F}_{\rom{1},i_{L,1}}\setminus\mathcal{F}_{\rom{1},i_{L,1}-1},$ and $\End{L}$ is a vertex of some tile in $\mathcal{F}_{\rom{1},i_{L,2}}\setminus\mathcal{F}_{\rom{1},i_{L,2}-1},$ and
			\item if $L,L'\in\mathcal{L}_0$ are distinct, then $\Set{i_{L,1},i_{L,2}}\cap\Set{i_{L',1},i_{L',2}}=\emptyset.$
		\end{enumerate}
		We do so iteratively.
		We start by initialising $\mathcal{L}_0 = \emptyset$ and $\mathcal{I}_0\coloneqq[2^{11+8\Fkt{f_P}{\cliqueToCrossRowGridFunction{t}}}\Fkt{f_P}{\cliqueToCrossRowGridFunction{t}}]$ and for every subset $\mathcal{I}'\subseteq\mathcal{I}_0,$ we define the family $\mathcal{F}_{\mathcal{I}'}\coloneqq \bigcup_{i\in\mathcal{I}'}\mathcal{F}_{\rom{1},i}\setminus\mathcal{F}_{\rom{1},i-1}.$
		Next, we add new paths to $\mathcal{L}_0$ while taking smaller and smaller subsets of $\mathcal{I}_0.$
		Also, define $t_{L_{0}} \coloneqq \Start{J}.$
		
		Let $q\in[2^{9+8\Fkt{f_P}{\cliqueToCrossRowGridFunction{t}}}\Fkt{f_P}{\cliqueToCrossRowGridFunction{t}}]$ and assume that the paths $L_1,\dots,$ $L_{q-1}$ together with the tiles, indices and the set $\mathcal{I}_{q-1}$ have already been constructed.
		Follow along $J,$ starting from $t_{L_{q-1}},$ until the next time we encounter the last vertex $s_{L_q}$ of some tile from $\mathcal{F}_{\mathcal{I}_{q-1}}$ before $J$ leaves said tile again.
		Let $\tile_{L_q,1}\in\tiling$ be the tile that contains $s_{L_q},$ and let $i_{L_q,1}\in\mathcal{I}_{q-1}$ be the integer such that $s_{L_q}$ belongs to a tile of $\mathcal{F}_{\rom{1},i_{L_q,q}}\setminus \mathcal{F}_{\rom{1},i_{L_q,q}-1}.$
		Then, let $L_q$ be the shortest subpath of $J$ that starts in $s_{L_q}$ and ends in a vertex $t_{L_q}$ which belongs to a tile from $\mathcal{F}_{\rom{1},i_{L_q,2}}\setminus \mathcal{F}_{\rom{1},i_{L_q,2}-1},$ where $i_{L_q,2}\in\mathcal{I}_{q-1}\setminus\Set{i_{L-q,1}}.$
		We choose $\tile_{L_q,2}\in\tiling$ to be the tile that contains $t_{L_q}$ and set $\mathcal{I}_q\coloneqq \mathcal{I}_{q-1}\setminus\Set{i_{L_q,1},i_{L_q,2}}.$
		Note that $t_{L_q}J$ still contains a vertex from some tile in $\mathcal{F}_{\rom{1},j}\setminus\mathcal{F}_{\rom{1},j-1}$ for every $j\in\mathcal{I}_q.$
		Add $L_q$ to $\mathcal{L}_0.$
		
		With every iteration we remove exactly two members from $\mathcal{I}_0$ and, due to $\Abs{\mathcal{I}_0}=2^{11+8\Fkt{f_P}{\cliqueToCrossRowGridFunction{t}}}\Fkt{f_P}{\cliqueToCrossRowGridFunction{t}},$ this means that by the time we reach some $q$ for which $\mathcal{I}_q=\emptyset,$ we have indeed constructed $2^{10+8\Fkt{f_P}{\cliqueToCrossRowGridFunction{t}}}\Fkt{f_P}{\cliqueToCrossRowGridFunction{t}}$ paths as required.
		
		\textbf{Obtaining $\mathcal{L}_5.$}
		There exist $c'\in[4]$ and a linkage $\mathcal{L}_1\subseteq \mathcal{L}_0$ of size $2^{8+8\Fkt{f_P}{\cliqueToCrossRowGridFunction{t}}}\Fkt{f_P}{\cliqueToCrossRowGridFunction{t}}$ such that each path $L\in\mathcal{L}_1$ has at least its start- or end-vertex in $\Class_{c'}.$
		Thus, we can find a linkage $\mathcal{L}_2\subseteq\mathcal{L}_1$ of size $2^{7+8\Fkt{f_P}{\cliqueToCrossRowGridFunction{t}}}\Fkt{f_P}{\cliqueToCrossRowGridFunction{t}}$ such that every path in $\mathcal{L}_2$ starts in a tile of $\Class_{c'},$ or every path in $\mathcal{L}_2$ ends in a tile of $\Class_{c'}.$
		Without loss of generality, we may assume that every path in $\mathcal{L}_2$ starts in a tile of $\Class_{c'},$ because the other case follows with similar arguments.
		
		Let $\widetilde{W}'$ be the smallest slice of $W$ such that $\widetilde{W}'$ contains all tiles from $\Class_{c'},$ but no tile from $\Class_{c'}$ meets the perimeter of $\widetilde{W}'.$
		Then, let $\widetilde{\tiling} \coloneqq \TierIITiling{\tiling,c',w}{\widetilde{W}'}$ be the \hyperref[def:tilingII]{tier \rom{2} tiling} of $\widetilde{W}\coloneqq\InducedSubgraph{\widetilde{W}'}{\tiling,c',w}.$
		Since the paths in $\mathcal{L}_2$ are pairwise disjoint, we can extend each $L\in\mathcal{L}_2$ such that it starts on the centre of the tile of $\widetilde{\tiling}$ which encloses its endpoint in $W,$ while making sure that the resulting family of paths is still at least half-integral.
		Similarly, wherever necessary, we may extend the paths through $W$ such that each of them also ends in a tile of $\widetilde{\tiling}.$
		Indeed, we can even guarantee that the start- and end-vertices of the resulting paths are mutually at $\widetilde{W}$-distance at least four.
		Let $\mathcal{L}_3$ be the resulting half-integral linkage.
		
		Next, consider the four-colouring $\Set{\widetilde{\Class_1},\dots,\widetilde{\Class_4}}$ of $\widetilde{\tiling}.$
		There exists $\widetilde{c}\in[4]$ and a family $\mathcal{L}_4\subseteq\mathcal{L}_3$ of size $2^{5+8\Fkt{f_P}{\cliqueToCrossRowGridFunction{t}}}\Fkt{f_P}{\cliqueToCrossRowGridFunction{t}}$ such that every path in $\mathcal{L}_4$ starts at the centre of some tile from $\Class_{\widetilde{c}}.$
		It follows from the construction of $\mathcal{L}_0$ that no two paths in $\mathcal{L}_4$ start in the same tile.
		
		By a similar argument, there exists a family $\mathcal{L}_5\subseteq\mathcal{L}_4$ of size $2^{4+8\Fkt{f_P}{\cliqueToCrossRowGridFunction{t}}}\Fkt{f_P}{\cliqueToCrossRowGridFunction{t}}$ such that either none, or all paths in $\mathcal{L}_5$ end in tiles of $\widetilde{\Class_{\widetilde{c}}}.$
		
		\textbf{Obtaining $\CrossRowGrid{t}.$}
		If none of the paths in $\mathcal{L}_5$ end in tiles of $\widetilde{\Class_{\widetilde{c}}},$ we can extend every path in $\mathcal{L}_5$ towards the perimeter of $\widetilde{W}$ such that the resulting family $\mathcal{L}_6$ of paths remains at worst half-integral, and the start- and end-vertices of the resulting paths are mutually at $\widetilde{W}$-distance at least four.
		By \cref{thm:half_integral}, we obtain a family $\mathcal{L}_7$ of size $2^{3+8\Fkt{f_P}{\cliqueToCrossRowGridFunction{t}}}\Fkt{f_P}{\cliqueToCrossRowGridFunction{t}}$ such that $\V{\mathcal{L}_7}\subseteq\V{\mathcal{L}_6},$ and the paths in $\mathcal{L}_7$ are pairwise vertex disjoint.
		Hence, \cref{lem:longjumps1} yields the existence of a $\CrossRowGrid{t}$ butterfly minor and our claim follows.
		
		If all of the paths in $\mathcal{L}_5$ end in tiles of $\widetilde{\Class_{\widetilde{c}}},$ we consider two subcases.
		Let $\mathcal{X}$ be the family of all tiles of $\widetilde{\tiling}\setminus\widetilde{\Class_{\widetilde{c}}}$ that contain an internal vertex of some path in $\mathcal{L}_5$ but no start- or end-vertex of any path in $\mathcal{L}_5.$
		
		Recall the following two definitions:
		\begin{enumerate}
			\item If $L$ and $P$ are directed paths, we say that $P$ is a \emph{long jump of $L$} if $P$ is a $w$-long jump over $W$ and $P\subseteq L.$
			\item $P$ is a \emph{jump of $L$}, if $P$ is a directed $\V{W}$-path.
		\end{enumerate}
		
		If $\Abs{\mathcal{X}}\geq 2^8\Fkt{f_P}{\cliqueToCrossRowGridFunction{t}},$ then we can use the technique from the first part of the proof of \cref{lem:longjumps1} to construct a half-integral family $\mathcal{L}_6$ such that
		\begin{enumerate}
			\item $\Abs{\mathcal{L}_6}=2^{4+8\Fkt{f_P}{\cliqueToCrossRowGridFunction{t}}}\Fkt{f_P}{\cliqueToCrossRowGridFunction{t}},$ and
			\item for every $L\in\mathcal{L}_6,$ every endpoint $u$ of a jump of $L$ with $u\in\V{\widetilde{W}}$ belongs to a tile from $\widetilde{\Class_{\widetilde{c}}}\cup\mathcal{X}.$
		\end{enumerate}
		Then, we apply \cref{thm:half_integral} to obtain a linkage $\mathcal{L}_7$ of size $2^{3+8\Fkt{f_P}{\cliqueToCrossRowGridFunction{t}}}\Fkt{f_P}{\cliqueToCrossRowGridFunction{t}}$ with $\V{\mathcal{L}_7}\subseteq\V{\mathcal{L}_6}$ such that the paths in $\mathcal{L}_7$ link the same two sets of vertices as the paths in $\mathcal{L}_6$ do.
		Finally, \cref{lem:longjumps1} yields the existence of a $\CrossRowGrid{t}$ as a butterfly minor.
		
		If $\Abs{\mathcal{X}}<2^8\Fkt{f_P}{\cliqueToCrossRowGridFunction{t}},$ then we find a subwall $\widetilde{W}'$ of $\widetilde{W}$ of order $d_1-2^8\Fkt{f_P}{\cliqueToCrossRowGridFunction{t}}\Brace{2w+1}$ that does not contain a vertex of any tile in $\mathcal{X}.$
		We do this by removing, for every tile $\tile\in\mathcal{X},$ all edges and vertices of the horizontal and vertical paths of $\tile$ that are not used by other cycles or paths.
		For each tile we remove during this procedure, we remove a row and a column of tiles and thereby reduce the number of distinct tiles which contain start-vertices of paths in $\mathcal{L}_5$ by a factor of at most $\frac{1}{2}.$
		However, because $\Abs{\mathcal{X}}<2^8\Fkt{f_P}{\cliqueToCrossRowGridFunction{t}},$ we can still find, after potentially expanding the start and end sections of the paths in $\mathcal{L}_5$ in order to reach the slightly shifted perimeters of their tiles, a half-integral family $\mathcal{L}_6$ of size $2^4\Fkt{f_P}{\cliqueToCrossRowGridFunction{t}}$ of paths that start and end in tiles of $\widetilde{\Class_{\widetilde{c}}}$ and that are otherwise disjoint from $\widetilde{W}'.$
		By applying \cref{thm:half_integral} we can transform this family into a linkage $\mathcal{L}_7$ of size $2^3\Fkt{f_P}{\cliqueToCrossRowGridFunction{t}}$ and thus an application of \cref{lem:longjumpsphase2} yields a $\CrossRowGrid{t}$ as a butterfly~minor.\end{claimproof}
	
	Concluding \hyperlink{phaseone}{Phase~\rom{1}}, \cref{claim:nolongjumps} either yields $\CrossRowGrid{t}$ as a butterfly minor and therefore finishes the proof, or $W_\rom{1}$ is in fact clean, meaning that $W_\rom{1}$ has no long jump in $D_\rom{1}.$
	Consequently we may bound the function $\WallApexFkt{}$ from the statement of \cref{thm:directed_flat_wall} as follows:
	\begin{equation*}
		\WallApexFkt{t} \leq \Brace{\cliqueToCrossRowGridFunction{t}}^{28} \cdot
		2^{61+2^{10}\Brace{\cliqueToCrossRowGridFunction{t}}^8}.
	\end{equation*}
	
	\hypertarget{phasetwo}{\paragraph{Phase~\rom{2}}}
	
	With $W_\rom{1}$ we have found a wall of still sufficient size but without any long jumps.
	Therefore, me may now find $t$ slices of $W_\rom{1},$ mutually still far enough apart from each other within $W_\rom{1},$ and can ask if among them there is one that is flat.
	If so, then we have found the desired flat wall within $W.$
	If not, each of the $t$ slices contains a cross and an application of \cref{lem:many_cross_slices_yield_minor} yields the desired $\CrossRowGrid{t}$ butterfly minor.
	The only technical part that remains is to provide sufficient definitions for these slices and their mutual distance.
	
	To meet the requirements from \hyperlink{phaseone}{Phase~\rom{1}} and have enough space left in $W_\rom{1},$ let us make the following assumption:
	\begin{equation*}
		d_2\geq t\Brace{r+4+2^{32+(3t)^{30}}}.
	\end{equation*}
	
	We partition $W_\rom{1}$ further into smaller slices.
	First we partition $W_\rom{1}$ into $t$ slices $S_i$ of width $r+4+2^{32+(3t)^{30}}.$
	Each $S_i$ is then partitioned into a slice $H_i$ of width $r+4+2^{31+(3t)^{30}}$ that contains the left perimeter cycle of the slice $S_i,$ and a slice $G_i$ of width $2$ containing the right perimeter cycle of $S_i.$
	For every $i\in[t]$ we may now further partition $H_i.$
	Let $N_{i,L}\subseteq H_i$ be the slice of width $2^{30+t^{30}}$ containing the left cycle of $\per{H_i},$ let $N_{i,R}$ be the slice of width $2^{30+t^{30}}$ containing the right cycle of $\per{H_i},$ and let $N_i'\coloneqq H_i-N_{L,i}-N_{R,i}$ be the remaining slice of width $r+4.$
	Finally, let $N_i$ be the slice obtained from $N_i'$ by removing the two leftmost and the two rightmost vertical cycles.
	Then, $N_i$ is a slice of width $r.$
	
	For every $i \in [t]$ we show that if there exists a directed path $P_i$ with one endpoint in $N_i,$ the other endpoint in $W_\rom{1}-N_i',$ and which is internally disjoint from $W_\rom{1},$ then $H_i$ contains a cross slice such that the cross lies in a strip of size at most $2^{30+t^{30}} + 2.$
	In this case, there is a vertical path $Q$ in one of the two components of $N_i'-N_i$ which separates the start- and the end-vertex of $P_i$ in $W_\rom{1}.$
	As there are no long jumps over $W_\rom{1}$ in $D_\rom{1},$ we further know that there exists $Y\in\Set{L,R}$ such that the endpoint of $P_i$ not in $N_i$ lies in $N_{Y,i}$ within a strip of height at most $2^{30+t^{30}} + 2.$
	By \cref{lem:local_jump_yields_cross_slice}, this yields a cross slice within $H_i.$
	Let $J \subseteq [t]$ contain every $i \in [t]$ for that such a $P_i$ exists.
	
	Suppose $J = [t],$ then we can use \cref{lem:many_cross_slices_yield_minor} and obtain a $\CrossRowGrid{t}$ butterfly minor, a contradiction.
	The row condition is met as the strips are all of height at most $2^{30+t^{30}} + 2$ leaving a strip of height at least $t.$
	
	So, there is at least one $i \in [t]$ for which such a path does not exist, let $J' \coloneqq [t]\setminus J.$
	Suppose for every $i \in J'$ the strong component of $D_\rom{1}-\per{N_i'}$ that contains $N_i$ has a cross slice.
	Since there is no long jump over $W_\rom{1},$ these components are pairwise disjoint and also disjoint from the cross slices found in the $S_j,$ $j\in J,$ and thus we can again apply \cref{lem:many_cross_slices_yield_minor} to obtain a $\CrossRowGrid{t}$ butterfly minor, a contradiction.
	
	Thus, there must exist some $i\in J'$ for which $N_i'$ has no cross slice.
	In particular, this means that the component of $D_\rom{1}-\per{N_i}$ containing the remaining vertices of $N_i-\per{N_i}$ must be free of cross slices.
	Hence, by \cref{lem:no_cross_slice_implies_flat}, $N_i$ is a flat wall of order $r$ in $D_\rom{1}=D-A$ which completes the proof.
	
	Let us combine all assumptions on the $d_1$ and $d_2$ to obtain the following bound on $\WallFkt{t,r}$:
	\begin{equation*}
		\WallFkt{t,r}\leq\Brace{2^{140}t^{72}}^{2^{10}t^8+2^{12}}\Brace{r+4+2^{32+\Brace{3t}^{30}}}.\qedhere
	\end{equation*} 
\end{proof}

\section{Conclusion}

The presented proof is held close to the ones of the other directed flat wall theorems~\cite{giannopoulou2020directed,giannopoulou2021flat} and thus, it often obtains clique minors, when, with a probably better function, a cross-row grid could be obtained instead.
Therefore, it might be possible to obtain overall better functions by choosing a procedure more specific to the structure of $\CrossRowGrid{t}.$

\bibliographystyle{alphaurl}
\bibliography{literature}
\end{document}